\documentclass{article}
\usepackage[UKenglish]{babel}
\usepackage[a4paper,top=3cm,bottom=2cm,left=3cm,right=3cm,marginparwidth=1.75cm]{geometry}

\usepackage{amsfonts,amsmath,amsthm,amssymb} 
\usepackage{mathrsfs} 
\usepackage{stmaryrd} 

\RequirePackage[OT1]{fontenc}
\usepackage[square,numbers]{natbib}
\RequirePackage[colorlinks,citecolor=blue,urlcolor=blue]{hyperref}

\usepackage{comment} 
\usepackage{graphicx} 

\allowdisplaybreaks[1]

\newcommand{\lm}[2]{\bar{#1}_{#2}}
\newcommand{\dop}{\mathrm{d}}
\renewcommand{\Re}{\mathbf{R}}
\renewcommand{\epsilon}{\varepsilon}
\newcommand{\vech}{\mathrm{vech}}

\newcommand{\E}[1]{\mathbf{E}\left[{#1}\right]}
\newcommand{\CE}[2]{\mathbf{E}\left[\left.{#1}\right|{#2}\right]}

\newcommand{\parens}[1]{\left({#1}\right)}
\newcommand{\lparens}[1]{\left({#1}\right.}
\newcommand{\rparens}[1]{\left.{#1}\right)}
\newcommand{\crotchet}[1]{\left[{#1}\right]}
\newcommand{\lcrotchet}[1]{\left[{#1}\right.}
\newcommand{\rcrotchet}[1]{\left.{#1}\right]}
\newcommand{\tuborg}[1]{\left\{{#1}\right\}}

\newcommand{\norm}[1]{\left\|{#1}\right\|}
\newcommand{\abs}[1]{\left|{#1}\right|}
\newcommand{\labs}[1]{\left|{#1}\right.}
\newcommand{\rabs}[1]{\left.{#1}\right|}

\newcommand{\cp}{\overset{P}{\to}}
\newcommand{\cl}{\overset{\mathcal{L}}{\to}}
\newcommand{\tr}{\mathrm{tr}}
\newcommand{\ip}[2]{#1\left\llbracket #2\right\rrbracket}

\newtheorem{theorem}{Theorem}[subsection]
\newtheorem{proposition}[theorem]{Proposition}
\newtheorem{lemma}[theorem]{Lemma}
\newtheorem{corollary}[theorem]{Corollary}

\newtheorem*{remark}{Remark}

\title{Adaptive estimation and noise detection for an ergodic diffusion with observation noises}
\author{Shogo H. Nakakita\\
	Graduate School of Engineering Science, Osaka University
	\and Masayuki Uchida\\
	Graduate School of Engineering Science and MMDS, Osaka University}
\date{\today}

\begin{document}
\maketitle

\begin{abstract}
	We research adaptive maximum likelihood-type estimation for an ergodic diffusion process where the observation is contaminated by noise. This methodology leads to the asymptotic independence of the estimators for the variance of observation noise, the diffusion parameter and the drift one of the latent diffusion process. 
	Moreover, it can lessen the computational burden compared to simultaneous 
	maximum likelihood-type  estimation. In addition to adaptive estimation, we propose a test to see if noise exists or not, and analyse real data as the example such that data contains observation noise with statistical significance.
\end{abstract}

\section{Introduction}

We consider a $d$-dimensional ergodic diffusion process defined by the following stochastic differential equation
\begin{align}
\dop X_t = b(X_t, \beta)\dop t + a(X_t, \alpha)\dop w_t,\ X_0 = x_0,
\end{align}
where $(w_t)_{t\ge 0}$ is a $r$-dimensional standard Wiener process, $x_0$ is a $\Re^d$-valued random variable independent of $(w_t)$, $\alpha\in\Theta_1\subset \Re ^{m_{1}}$, $\beta\in\Theta_2\subset \Re ^{m_{2}}$ with $\Theta_1$ and $\Theta_2$ being compact and convex. Moreover, $b:\Re^d\times\Theta_2\to \Re ^d$, $a:\Re^d\times \Theta_1\to \Re^d\otimes\Re^r$. We denote $\theta:=(\alpha,\beta)\in\Theta_1\times \Theta_2 =:\Theta$ and $\theta^\star=(\alpha^\star,\beta^\star)$ as the true value of $\theta$ which belongs to $\mathrm{Int}(\Theta)$.

We deal with the problem of parametric inference for $\theta$ with $(Y_{ih_n})_{i=1,\cdots,n}$ defined by the following model
\begin{align}
Y_{ih_n}=X_{ih_n}+\Lambda^{1/2} \epsilon_{ih_n},\ i=0,\cdots,n,
\end{align}
where $h_n>0$ is the discretised step, $\Lambda\in \Re^d\otimes\Re^{d}$ is a positive semi-definite matrix and $(\epsilon_{ih_n})_{i=1,\cdots,n}$ is a sequence of $\Re^d$-valued i.i.d. random variables with $\mathbf{E}[\epsilon_{ih_n}]=\mathbf{0}$ and 
$\mathrm{Var}[ \epsilon_{ih_n} ] =I_k$. Let $\Theta_\epsilon\in\Re^{d(d+1)/2}$ be the convex and compact parameter space such that $\theta_\epsilon:=\mathrm{vech}(\Lambda)\in\Theta_\epsilon$ and $\Lambda_\star$ be the true value of $\Lambda$ such that $\theta_\epsilon^\star:=\mathrm{vech}(\Lambda_\star)\in\mathrm{Int}(\Theta_\epsilon)$. We denote $\vartheta := (\theta, \theta_\epsilon)$ and $\Xi:=\Theta\times\Theta_\epsilon$. With respect to the sampling scheme, we assume that $h_n\to0$ and $nh_n\to\infty$ as $n\to\infty$. 


Our main concern with these settings is the adaptive maximum likelihood (ML)-type estimation scheme in the form of  
\begin{align}
\hat{\Lambda}_n&=\frac{1}{2n}\sum_{i=0}^{n-1}\parens{Y_{(i+1)h_n}-Y_{ih_n}}^{\otimes2},\\
\mathbb{L}_{1,n}(\hat{\alpha}_n|\hat{\Lambda}_n)&=\sup_{\alpha\in\Theta_1}\mathbb{L}_{1,n}(\alpha|\hat{\Lambda}_n),\\
\mathbb{L}_{2,n}(\hat{\beta}_n|\hat{\Lambda}_n,\hat{\alpha}_n)&=\sup_{\beta\in\Theta_2}\mathbb{L}_{2,n}(\beta|\hat{\Lambda}_n,\hat{\alpha}_n),
\end{align}
where $A^{\otimes 2} = AA^T$ for any matrix $A$, $A^T$ indicates the transpose of $A$, $\mathbb{L}_{1,n}$ and $\mathbb{L}_{2,n}$ are quasi-likelihood functions, which are defined in Section 3.


The composition of the model above is quite analogous to that of discrete-time state space models (e.g., see \citep{PPC09}) in terms of expression of endogenous perturbation in the system of interest and exogenous noise attributed to observation separately. As seen in the assumption $h_n\to0$, this model that we consider is for the situation where high-frequency observation holds, and this requirement enhances the flexibility of modelling since our setting includes the models with non-linearity, dependency of the innovation on state space itself. In addition, adaptive estimation which also becomes possible through the high-frequency setting has the advantage in easing computational burden in comparison to simultaneous one. Fortunately, the number of situations where requirements are satisfied has been grown gradually, and will continue to soar because of increase in the amount of real-time data and progress of observation technology these days.

The idea of modelling with diffusion process concerning observational noise is no new phenomenon. For instance, in the context of high-frequency financial data analysis, the researchers have addressed the existence of "microstructure noise" with large variance with respect to time increment jeopardising the premise that what we observe are purely diffusions. The energetic research of the modelling with "diffusion + noise" has been conducted in the decade: some research have examined the asymptotics of this model in the framework of fixed time interval such that $nh_n=T=1$ (e.g., \citep{GlJ01a}, \citep{GlJ01b}, \citep{JLMPV09}, \citep{PV09} and \citep{O17}); and \citep{Fa14} and \citep{Fa16} research the parametric inference of this model with ergodicity and the asymptotic framework $T\to \infty$. For parametric estimation for discrete observed diffusion processes without measurement errors, see \citep{Fl89}, \citep{Y92}, \citep{Y11}, \citep{BS95}, \citep{K97} and references therein.

Our research is focused on the statistical inference for an ergodic diffusion plus noise. 
We give the estimation methodology with adaptive estimation that relaxes computational burden 
and that has been researched for ergodic diffusions so far (see \citep{Y92}, \citep{Y11}, \citep{K95}, \citep{UY12}, \citep{UY14}) in comparison to the simultaneous estimation of \citep{Fa14} and \citep{Fa16}. In previous researches the simultaneous asymptotic normality of $\hat{\Lambda}_n$, $\hat{\alpha}_n$ and $\hat{\beta}_n$ has not been shown, but our method allows us to see asymptotic normality and asymptotic independence of them with the different convergence rates. 


As the real data analysis, we analyse the 2-dimensional wind data \citep{NWTC} and try to model the dynamics with 2-dimensional Ornstein-Uhlenbeck process. We utilise the fitting of our diffusion-plus-noise modelling and that of diffusion modelling with estimation methodology called local Gaussian approximation method (LGA method) which has been investigated for these decades (for instance, see \citep{Y92}, \citep{K95} and \citep{K97}). The results of fitting are as follows: diffusion-plus-noise fitting gives
\begin{align}
\dop \crotchet{\begin{matrix}
	X_t\\
	Y_t
	\end{matrix}}
&=\parens{\crotchet{\begin{matrix}
		-3.77 & -0.32\\
		-0.40 & -5.01
		\end{matrix}}
	\crotchet{\begin{matrix}
		X_t\\
		Y_t
		\end{matrix}}
	+ \crotchet{\begin{matrix}
		3.60\\
		-2.54
		\end{matrix}}}\dop t+\crotchet{\begin{matrix}
	13.41 & -0.29\\
	-0.29 & 12.62
	\end{matrix}}\dop w_t,
\end{align}
with $\parens{X_0,Y_0}=\parens{-2.53,0.36}$ and the estimation of the noise variance
\begin{align}
\hat{\Lambda}_n=\crotchet{\begin{matrix}
	6.67\times 10^{-3} & 3.75\times 10^{-5}\\
	3.75\times 10^{-5} & 6.79\times 10^{-3}\\
	\end{matrix}};
\end{align} and the diffusion fitting with LGA method which is asymptotic efficient if $\Lambda=O$ gives
\begin{align}
\dop \crotchet{\begin{matrix}
	X_t\\
	Y_t
	\end{matrix}}
&=\parens{\crotchet{\begin{matrix}
		-67.53 &  -9.29\\
		-10.37 & -104.45
		\end{matrix}}
	\crotchet{\begin{matrix}
		X_t\\
		Y_t
		\end{matrix}}
	+ \crotchet{\begin{matrix}
		63.27\\
		-50.24
		\end{matrix}}}\dop t+\crotchet{\begin{matrix}
	43.82 & 0.13\\
	0.13 & 44.22
	\end{matrix}}\dop w_t
\end{align}
with the same initial value.
It seems that there is considerable difference between these estimates: however, we cannot evaluate which is the more trustworthy fitting only with these results. It results from the fact that we cannot distinguish a diffusion from a diffusion-plus-noise; if $\Lambda_{\star}=O$, then the observation is not contaminated by noise and the estimation of LGA should be adopted for its asymptotic efficiency; but if $\Lambda_{\star}\neq O$, what we observe is no more a diffusion process and the LGA method loses its theoretical validity. Therefore, it is necessary to compose the statistical hypothesis test with $H_0: \Lambda = O$ and $H_1: \Lambda \neq O$. In addition to estimation methodology, we also research this problem of hypothesis test and propose a test which has the consistency property.



In Section 2, we check the assumption and notation across the paper. Section 3 gives the main 
results 
of this paper. Section 4 examines the result of Section 3 with simulation. In Section 5 we analyse the real data analysis for wind data named MetData with our estimators and LGA as discussed above and test whether noise does exist.

\section{Local means, notations and assumptions}
\subsection{Local means}
We partition the observation into $k_n$ blocks containing $p_n$ samples and examine the property of the following local means such that
\begin{align}
\lm{Z}{j}{}=\frac{1}{p_n}\sum_{i=0}^{p_n-1}Z_{j\Delta_n+ih_n},\ j=0,\cdots,k_n-1,
\end{align}
where $\tuborg{Z_{ih_n}}_{i}$ is an arbitrary sequence of random variables on the mesh $\{ih_n\}_i$ as $\tuborg{Y_{ih_n}}_{i}$, $\tuborg{X_{ih_n}}_{i}$ and $\tuborg{\epsilon_{ih_n}}_{i}$; 
and $\Delta_n=p_nh_n$. 
Note that $k_n p_n =n$ and $k_n \Delta_n =n h_n$.

In the same way as \citep{Fa14} and \citep{Fa16},
our estimation method is based on these local means with respect to the observation $\tuborg{Y_{ih_n}}_{i=1,\cdots,n}$. The idea is so straightforward; taking means of the data $\tuborg{Y_{ih_n}}$ in each partition should reduce the influence of the noise term $\tuborg{\epsilon_{ih_n}}$ because of LLN and then we will obtain the information of the latent process $\tuborg{X_{ih_n}}$. 

\subsection{Notations and assumptions}

We set the following notations.
\begin{enumerate}
	\item For a matrix $A$, $A^T$ denotes the transpose of $A$ and $A^{\otimes 2}:=AA^T$. For same size matrices $A$ and $B$, $\ip{A}{B}:=\tr\parens{AB^T}$. 
	\item For any vector $v$, $v^i$ denotes the $i$-th component of $v$. Similarly, $M^{i,j}$, $M^{i,\cdot}$ and $M^{\cdot,j}$ denote the $(i,j)$-th component, the $i$-th row vector and $j$-th column vector of a matrix $M$ respectively.
	\item $c(x,\alpha):=\parens{a(x,\alpha)}^{\otimes 2}.$
	\item $C$ is a positive generic constant independent of all other variables. If it depends on fixed other variables, e.g. an integer $k$, we will express as $C(k)$.
	\item $a(x):=a(x,\alpha^\star)$ and $b(x):=b(x,\beta^\star)$.
	\item A $\Re$-valued function $f$ on $\Re^d$ is a \textit{polynomial growth function} if for all $x\in \Re^d$,
	\begin{align*}
	\abs{f(x)}\le C\parens{1+\norm{x}}^C.
	\end{align*}
	$g:\Re^d\times \Theta\to \Re$ is a \textit{polynomial growth function uniformly in $\theta\in\Theta$} if for all $x\in \Re^d$,
	\begin{align*}
	\sup_{\theta\in\Theta}\abs{g(x,\theta)}\le C\parens{1+\norm{x}}^C.
	\end{align*}
	Similarly we say $h:\Re^d\times \Xi\to \Re$ is a \textit{polynomial growth function uniformly in $\vartheta\in\Xi$} if for all $x\in \Re^d$,
	\begin{align*}
	\sup_{\vartheta\in\Xi}\abs{h(x,\vartheta)}\le C\parens{1+\norm{x}}^C.
	\end{align*}
	\item For any $\Re$-valued sequence $u_n$, $R:\Theta\times\Re\times \Re^d\to \Re$ denotes a function with a constant $C$ such that
	\begin{align*}
	\abs{R(\theta,u_n,x)}\le Cu_n\parens{1+\norm{x}}^C
	\end{align*}
	for all $x\in\Re^d$ and $\theta\in\Theta$.
	\item Let us define $\vartheta:=\parens{\theta,\theta_\epsilon}\in \Xi$.
	\item Let us denote for any $\mu$-integrable function $f$ on $\Re^d$, $
	\mu(f(\cdot)) := \int f(x)\mu(\dop x).$
	\item 	We set \begin{align*}
	\mathbb{Y}_1(\alpha)&:=-\frac{1}{2}\nu_0\parens{\mathrm{tr}\parens{\parens{c(\cdot,\alpha)}^{-1}c(\cdot,\alpha^\star)-I_d}+\log\frac{\det c(\cdot,\alpha)}{\det c(\cdot,\alpha^\star)}},\\
	\tilde{\mathbb{Y}}_1(\alpha)&:=-\frac{1}{2}\nu_0\parens{\mathrm{tr}\parens{\parens{c^\dagger(\cdot,\alpha,\Lambda_{\star})}^{-1}c^\dagger(\cdot,\alpha^{\star},\Lambda_{\star})-I_d}+\log\frac{\det c^\dagger(\cdot,\alpha,\Lambda_\star)}{\det c^\dagger(\cdot,\alpha^\star,\Lambda_{\star})}},\\
	\mathbb{Y}_2(\beta)&:=-\frac{1}{2}\nu_0\parens{\ip{\parens{c(\cdot,\alpha^\star)}^{-1}}{b(\cdot,\beta)-b(\cdot,\beta^\star)^{\otimes2}}}, \\
	\tilde{\mathbb{Y}}_2(\beta)&:=-\frac{1}{2}\nu_0\parens{\ip{\parens{c^\dagger(\cdot,\alpha^\star,\Lambda_{\star})}^{-1}}{b(\cdot,\beta)-b(\cdot,\beta^\star)^{\otimes2}}}, 
	\end{align*}
	where $
	c^\dagger(\cdot,\alpha,\Lambda):=c(\cdot,\alpha)+3\Lambda.$
	\item 
	Let
	\begin{align*}
	&\left\{A_{\kappa}(x)\left|\kappa=1,\cdots,m_1,\ A_\kappa=(A_\kappa^{j_1,j_2})_{j_1,j_2}\right.\right\},\\
	&\left\{f_{\lambda}(x)\left|\lambda=1,\cdots,m_2,\ f_\lambda=(f^1_\lambda,\cdots,f^d_\lambda)\right.\right\}
	\end{align*}
	be
	sequences of $\Re^d\otimes \Re^d$-valued functions and $\Re^d$-valued ones respectively such that the components of themselves and their derivative with respect to $x$ are polynomial growth functions for all $\kappa$ and $\lambda$. 
	Then we define the following matrix 
	\begin{align*}
	W_1^{(l_1,l_2),(l_3,l_4)}&:=\sum_{k=1}^{d}\parens{\Lambda_{\star}^{1/2}}^{l_1,k}\parens{\Lambda_{\star}^{1/2}}^{l_2,k}\parens{\Lambda_{\star}^{1/2}}^{l_3,k}\parens{\Lambda_{\star}^{1/2}}^{l_4,k}
	\parens{\E{\abs{\epsilon_0^k}^4}-3}\notag\\
	&\qquad+\frac{3}{2}\parens{\Lambda_{\star}^{l_1,l_3}\Lambda_{\star}^{l_2,l_4}+\Lambda_{\star}^{l_1,l_4}\Lambda_{\star}^{l_2,l_3}},
	\end{align*}
	and matrix-valued functionals
	\begin{align*}
	\parens{W_2^{(\tau)}(\tuborg{A_{\kappa}})}^{\kappa_1,\kappa_2}&:=\begin{cases}
	\nu_0\parens{\tr\tuborg{\parens{\bar{A}_{\kappa_1}c\bar{A}_{\kappa_2}c}(\cdot)}} \\
	\qquad\text{ if }\tau\in(1,2),\\
	\nu_0\parens{\tr\tuborg{\parens{\bar{A}_{\kappa_1}c\bar{A}_{\kappa_2}c+4\bar{A}_{\kappa_1}c\bar{A}_{\kappa_2}\Lambda_{\star}+12\bar{A}_{\kappa_1}\Lambda_{\star}\bar{A}_{\kappa_2}\Lambda_{\star}}(\cdot)}}\\
	\qquad\text{ if }\tau=2,
	\end{cases}\\
	\parens{W_3(\tuborg{f_{\lambda}})}^{\lambda_1,\lambda_2}&:=\nu_0\parens{\parens{f_{\lambda_1}c\parens{f_{\lambda_2}}^T}(\cdot)}, 
	\end{align*}
	where $\bar{A}_{\kappa}:=\frac{1}{2}\parens{A_{\kappa}+A_{\kappa}^T}$.
	\item $\cp$ and $\cl$ indicate convergence in probability and convergence in law respectively.
	\item For $f(x)$, $g(x,\theta)$ and $h(x,\vartheta)$, $f'(x):=\frac{\dop }{\dop x}f(x)$, $f''(x):=\frac{\dop^2}{\dop x^2}f(x)$, $\partial_{x}g(x,\theta):=\frac{\partial}{\partial x}g(x,\theta)$, $\partial_{\theta}g(x,\theta):=\frac{\partial}{\partial \theta}g(x,\theta)$, $\partial_{x}h(x,\vartheta):=\frac{\partial}{\partial x}h(x,\vartheta)$ and  $\partial_{\vartheta}h(x,\vartheta):=\frac{\partial}{\partial \vartheta}h(x,\vartheta)$.
\end{enumerate}

We make the following assumptions.
\begin{enumerate}
	\item[(A1)] $b$ and $a$ are continuously differentiable for 4 times,  and the components of themselves as well as their derivatives are polynomial growth functions uniformly in $\theta\in\Theta$. Furthermore, there exists $C>0$ such that for all $x\in\Re^d$,
	\begin{align*}
	&\norm{b(x)}+\norm{b'(x)}+\norm{b''(x)}\le C(1+\norm{x}), \\ 
	&\norm{a(x)}+\norm{a'(x)}+\norm{a''(x)}\le C(1+\norm{x}). 
	\end{align*}
	\item[(A2)] $X$ is ergodic and the invariant measure $\nu_0$ has $k$-th moment for all $k>0$.
	\item[(A3)] For all $k>0$, $\sup_{t\ge0}\E{\norm{X_t}^k}<\infty$.
	\item[(A4)] For any $k>0$, $\epsilon_{ih_n}$ has $k$-th moment and the component of $\epsilon_{ih_n}$ are independent of the other components for all $i$. In addition, the marginal distribution of each component is symmetric.
	\item[(A5)] $\inf_{x,\alpha} \det c(x,\alpha)>0$.
	\item[(A6)] There exist positive constants $\chi$ and $\tilde{\chi}$ such that $\mathbb{Y}_1(\alpha)\le -\chi\norm{\alpha-\alpha^\star}^2$,
	$\tilde{\mathbb{Y}}_1(\alpha|\Lambda_{\star})\le-\chi\norm{\alpha-\alpha^\star}^2$, 
	$\mathbb{Y}_2(\beta)\le -\tilde{\chi}\norm{\beta-\beta^\star}^2$ and $\tilde{\mathbb{Y}}_2(\beta)\le -\tilde{\chi}\norm{\beta-\beta^\star}^2$.
	\item[(A7)] The components of $b$, $a$, $\partial_xb$, $\partial_{\beta}b$, $\partial_xa$, $\partial_{\alpha}a$, $\partial_x^2b$, $\partial_{\beta}^2b$, $\partial_x\partial_{\beta}b$, $\partial_x^2a$, 
	$\partial_{\alpha}^2a$ and $\partial_x\partial_{\alpha}a$ are polynomial growth functions uniformly in $\theta\in\Theta$.
	\item[(AH)] $h_n=p_n^{-\tau},\ \tau\in(1,2]$ and $h_n\to0$, $p_n\to\infty$, $k_n\to\infty$, $\Delta_n=p_nh_n\to0$, $nh_n\to\infty$ as $n\to\infty$.
	\item[(T1)] If the index set $\mathcal{J}:=\tuborg{i\in\tuborg{1,\cdots,d}:\Lambda_{\star}^{i,i}>0}$ is not null, then the submatrix of $\Lambda_{\star}$ such that $\Lambda_{\star,\mathrm{sub}}:=\crotchet{\Lambda_{\star}^{i_1,j_2}}_{i_1,i_2\in \mathcal{J}}$
	is positive definite.
\end{enumerate}

\section{Main results}
\subsection{Adaptive ML-type estimation}
Firstly, we construct an estimator for $\Lambda$ such that
$\hat{\Lambda}_n:=\frac{1}{2n}\sum_{i=0}^{n-1}\parens{Y_{(i+1)h_n}-Y_{ih_n}}^{\otimes2}$.

\begin{lemma}\label{lem311}
	Under (A1)-(A4), $h_n\to0$ and $nh_n\to\infty$ as $n\to\infty$, $\hat{\Lambda}_n$ is consistent.
\end{lemma}
We propose the following quasi-likelihood functions such that
{\small\begin{align}
	&\mathbb{L}_{1,n}(\alpha|\Lambda):=-\frac{1}{2}\sum_{j=1}^{k_n-2}
	\parens{\ip{\parens{\frac{2}{3}\Delta_n c_n^{\tau}(\lm{Y}{j-1},\alpha,\Lambda)}^{-1}}{\parens{\lm{Y}{j+1}-\lm{Y}{j}}^{\otimes 2}}
		+\log\det \parens{ c_n^{\tau}(\lm{Y}{j-1},\alpha,\Lambda)}}, \\
	&\mathbb{L}_{2,n}(\beta|\Lambda,\alpha)
	:=-\frac{1}{2}\sum_{j=1}^{k_n-2}
	\ip{\parens{\Delta_nc_n^\tau(\lm{Y}{j-1},\alpha,\Lambda)}^{-1}}{\parens{\lm{Y}{j+1}-\lm{Y}{j}-\Delta_nb(\lm{Y}{j-1},\beta)}^{\otimes 2}},
	\end{align}}
where $c_n^\tau(x,\alpha,\Lambda):=c(x,\alpha)+3\Delta_n^{\frac{2-\tau}{\tau-1}}\Lambda. $ We define the estimators $\hat{\alpha}_n$ and $\hat{\beta}_n$,  where
\begin{align}
\mathbb{L}_{1,n}(\hat{\alpha}_n|\hat{\Lambda}_n)&=\sup_{\alpha\in\Theta_1}\mathbb{L}_{1,n}(\alpha|\hat{\Lambda}_n),\\
\mathbb{L}_{2,n}(\hat{\beta}_n|\hat{\Lambda}_n,\hat{\alpha}_n)&=\sup_{\beta\in\Theta_2}\mathbb{L}_{2,n}(\beta|\hat{\Lambda}_n,\hat{\alpha}_n).
\end{align}
The consistency of these estimators is given by the next theorem.

\begin{theorem}\label{thm312}
	Under (A1)-(A7) and (AH), $\hat{\alpha}_n$ and $\hat{\beta}_n$ are consistent.
\end{theorem}

Let us denote
\begin{align}
I^{\tau}(\vartheta^{\star}):=\crotchet{\begin{matrix}
	W_{1} & O & O\\
	O & I^{(2,2),\tau}(\vartheta^{\star}) & O\\
	O & O & I^{(3,3),\tau}(\vartheta^{\star})
	\end{matrix}}, \ J^{\tau}(\vartheta^{\star}):=\crotchet{\begin{matrix}
	I & O & O\\
	O & J^{(2,2),\tau}(\vartheta^{\star}) & O\\
	O & O & J^{(3,3),\tau}(\vartheta^{\star})
	\end{matrix}},
\end{align}
where for $i_1,i_2\in\tuborg{1,\cdots,m_1}$, $j_1,j_2\in\tuborg{1,\cdots,m_2}$,
\begin{align}
I^{(2,2),\tau}(\vartheta^{\star})&:=\begin{cases}
W_2^{(\tau)}\parens{\tuborg{\frac{3}{4}\parens{c}^{-1}
		\parens{\partial_{\alpha^{k_1}}c}\parens{ c}^{-1}(\cdot,\alpha^\star)}_{k_1}}&\text{ if }\tau\in(1,2),\\
W_2^{(\tau)}\parens{\tuborg{\frac{3}{4}\parens{ c^{\dagger}}^{-1}\parens{\partial_{\alpha^{k_1}}c}\parens{ c^{\dagger}}^{-1}(\cdot,\vartheta^{\star})}_{k_1}}&\text{ if }\tau = 2,
\end{cases}\\
I^{(3,3),\tau}(\vartheta^{\star})&:=\begin{cases}
W_3\parens{\tuborg{\parens{\partial_{\beta^{k_2}}b}^T\parens{c}^{-1}(\cdot,\theta^\star)}_{k_2}} & \text{ if }\tau\in(1,2),\\
W_3\parens{\tuborg{\parens{\partial_{\beta^{k_2}}b}^T\parens{c^{\dagger}}^{-1}(\cdot,\vartheta^\star)}_{k_2}}& \text{ if }\tau=2,
\end{cases}\\
J^{(2,2),\tau}(\vartheta^{\star})&:=\begin{cases}
\crotchet{\frac{1}{2}\nu_0\parens{\tr\tuborg{\parens{c}^{-1}\parens{\partial_{\alpha^{i_1}}c}\parens{c}^{-1}\parens{\partial_{\alpha^{i_2}}c}}(\cdot,\alpha^{\star})}
}_{i_1,i_2} & \text{ if }\tau\in(1,2), \\
\crotchet{\frac{1}{2}\nu_0\parens{\tr\tuborg{\parens{c^{\dagger}}^{-1}\parens{\partial_{\alpha^{i_1}}c}\parens{c^{\dagger}}^{-1}
			\parens{\partial_{\alpha^{i_2}}c}}(\cdot,\vartheta^{\star})}}_{i_1,i_2}&\text{ if }\tau = 2, 
\end{cases}\\
J^{(3,3),\tau}(\vartheta^{\star})&:=\begin{cases}
\crotchet{\nu_0\parens{\ip{\parens{c}^{-1}}{
			\parens{\partial_{\beta^{j_1}}b}\parens{\partial_{\beta^{j_2}}b}^T}(\cdot,\theta^\star)}}_{j_1,j_2} & \text{ if }\tau\in(1,2), \\
\crotchet{\nu_0\parens{\ip{\parens{c^{\dagger}}^{-1}}{\parens{\partial_{\beta^{j_1}}b}\parens{\partial_{\beta^{j_2}}b}^T}(\cdot,\vartheta^\star)}}_{j_1,j_2} & \text{ if }\tau=2.
\end{cases}
\end{align}

In addition, let us denote $\hat{\theta}_{\epsilon,n}:=\vech\hat{\Lambda}_{n}$ and $\theta_{\epsilon}^{\star}:=\vech \Lambda_{\star}$.

\begin{theorem}\label{thm313}
	Under (A1)-(A7), (AH) and $k_n\Delta_n^2\to0$, the following convergence in distribution holds:
	\begin{align*}
	\crotchet{\begin{matrix}
		\sqrt{n}\parens{\hat{\theta}_{\epsilon,n}-\theta_{\epsilon}^{\star}}\\
		\sqrt{k_n}\parens{\hat{\alpha}_{n}-\alpha^{\star}}\\
		\sqrt{k_n\Delta_n}\parens{\hat{\beta}_n-\beta^{\star}}
		\end{matrix}}
	\cl N\parens{\mathbf{0},\parens{J^{\tau}(\vartheta^{\star})}^{-1}I^{\tau}(\vartheta^{\star})\parens{J^{\tau}(\vartheta^{\star})}^{-1}}.
	\end{align*}
\end{theorem}\qquad\\

\noindent \textbf{Remark.} This theorem shows the difference of the convergence rates with respect to $\hat{\theta}_{\epsilon,n}$, $\hat{\alpha}_n$ and $\hat{\beta}_n$ which is essentially significant to construct adaptive estimation approach.

\subsection{Test for noise detection}
We formulate the statistical hypothesis test such that $H_0: \Lambda_{\star}=O$ and $H_1: \Lambda_{\star}\neq O$.
We define $S_{t}:= \sum_{l=1}^{d}X_t^l$ and $\mathscr{S}_{ih_n} := \sum_{l=1}^{d}Y_{ih_n}^l$
and $\crotchet{X_t^1,\cdots,X_t^d,S_{t}}$ is also an ergodic diffusion.
Furthermore, 
{\small \begin{align}
	Z_{n}:=\sqrt{\frac{2p_n}{3\sum_{j=1}^{k_n-2}\parens{\lm{\mathscr{S}}{j+1}-\lm{\mathscr{S}}{j}}^4}}\parens{\sum_{i=0}^{n-1}\parens{\mathscr{S}_{(i+1)h_n}-\mathscr{S}_{ih_n}}^2
		-\sum_{0\le 2i\le n-2}\parens{\mathscr{S}_{(2i+2)h_n}-\mathscr{S}_{2ih_n}}^2},
	\end{align}}
where $\lm{\mathscr{S}}{j}:=\frac{1}{p_n}\sum_{i=0}^{p_n-1}\mathscr{S}_{j\Delta_n+ih_n}$, and consider the hypothesis test with rejection region $Z_{n}\ge z_{\alpha}$ 
where $z_{\alpha}$ is the upper $\alpha$ point of $N(0,1)$.

\begin{theorem}\label{thm321} Under $H_0$, (A1)-(A5), (AH) and $nh_n^2\to0$, 
	\begin{align*}
	Z_{n}\cl N(0,1).
	\end{align*}
\end{theorem}

\begin{theorem}\label{thm322} Under $H_1$, (A1)-(A5), (AH), (T1) and $nh_n^2\to0$, the test is consistent, i.e., for all $\alpha\in(0,1)$,
	\begin{align*}
	P(Z_{n}\ge z_{\alpha})\to 1.
	\end{align*}
\end{theorem}

\section{Example and simulation results}
\subsection{Case of small noise}
First of all, we consider the following 2-dimensional Ornstein-Uhlenbeck process
\begin{align}
\dop \crotchet{\begin{matrix}
	X_t^1\\
	X_t^2
	\end{matrix}}=\parens{\crotchet{\begin{matrix}
		\beta_1 & \beta_3\\
		\beta_2 & \beta_4
		\end{matrix}}\crotchet{\begin{matrix}
		X_t^1\\
		X_t^2
		\end{matrix}}+\crotchet{\begin{matrix}
		\beta_5\\
		\beta_6
		\end{matrix}}}\dop t
+ \crotchet{\begin{matrix}
	\alpha_1 & \alpha_2\\
	\alpha_2 & \alpha_3
	\end{matrix}}\dop w_t,\ \crotchet{\begin{matrix}
	X_0^1\\
	X_0^2
	\end{matrix}}=\crotchet{\begin{matrix}
	1\\
	1
	\end{matrix} }, 
\end{align}
where the true values of the diffusion parameters $\parens{\alpha_1^{\star}, \alpha_2^{\star},\alpha_3^{\star}}=\parens{1, 0.1,1}$ and the drift one $\parens{\beta_1^{\star}, \beta_2^{\star},\beta_3^{\star},\beta_4^{\star},\beta_5^{\star},\beta_6^{\star}}=\parens{-1,-0.1,-0.1,-1,1,1}$,
and the multivariate normal noise and the several levels of $\Lambda$ such that $\Lambda_{\star,-\infty}=O, \Lambda_{\star,-i}=10^{-i}I_2$
for all $i=\{4,5,6,7,8\}$. We check the performance of our estimator and the test constructed in Section 3, and compare our estimator (local mean method, LMM) with the estimator by LGA.
We show the setting and result of simulation in the following tables. With respect to the estimator for the noise variance, let us check the case of $\Lambda_{\star,-4}$. The empirical mean and standard deviation of $\hat{\Lambda}_{n}^{1,1}$ with $\Lambda_{\star}^{1,1}=10^{-4}$ are $1.32\times10^{-4}$ and $3.21\times 10^{-5}$; those of $\hat{\Lambda}_{n}^{1,2}$ with $\Lambda_{\star}^{1,2}=0$ are $6.29\times 10^{-6}$ and $6.31\times 10^{-6}$; and those of $\hat{\Lambda}_{n}^{2,2}$ with $\Lambda_{\star}^{1,2}=10^{-4}$ are $1.33\times10^{-4}$ and $3.25\times 10^{-5}$.
\begin{table}[h]
	\begin{center}
		\caption{Setting in Section 4}
		\begin{tabular}{c|c}
			quantity & approximation \\\hline
			$n$ & $10^6$\\
			$h_n$ & $6.309573\times 10^{-5}$\\
			$T_n$ &  $63.09573$\\
			$nh_n^2$ & $0.003981072$\\
			$\tau$ & $2$\\
			$p_n$ & $125$\\
			$k_n$ & $8000$\\
			$\Delta_n$ & $0.007886967$\\
			$k_n\Delta_n^2$ & $0.497634$\\
			iteration & 1000
		\end{tabular}\qquad\\\qquad\\
	\end{center}
\end{table}

In the first place, we examine the performance of the diffusion estimators. It can be seen that neither estimator with our method nor LGA dominates the other in terms of standard deviation where $\Lambda_{\star,-\infty}$, $\Lambda_{\star,-8}$ and $\Lambda_{\star,-7}$. Note that the powers of the test statistics are not large in these settings. It reflects that it is indifferent to choose either our estimators which are consistent even if there is no noise or the estimators with LGA by counting on the result of noise detection test which are asymptotically efficient if observation is not contaminated by noise. In contrast to these sizes of variance of noise, the results of simulation with the setting $\Lambda_{\star,-6}$, $\Lambda_{\star,-5}$ and $\Lambda_{\star,-4}$ shows that our estimators dominate the estimators with LGA in terms of standard deviation, and simultaneously the test for noise detection performs high power. 

We also see the same behaviour in estimation for drift parameters. In this case, our estimators are dominant in all the setting of noise variance, but the performance of LGA estimators are close to them where $\Lambda_{\star,-\infty}$, $\Lambda_{\star,-8}$ and $\Lambda_{\star,-7}$. With the larger variance of noise, the estimators with local means method are far fine compared to the others.

\begin{remark}
	\normalfont With these results,  we can see that the test works well as a criterion to select the estimation methods with local means and LGA: when adopting $H_0:\Lambda_{\star}=O$, we are essentially free to adopt either estimation; if rejecting $H_0$, we are strongly motivated to select our estimator.
\end{remark}

\begin{table}[!ht]
	\caption{test statistics performance with small noise (section 4.1)}
	\centering
	\begin{tabular}{c|ccc}
		& ratio of $Z_n>z_{0.05}$ & ratio of $Z_n>z_{0.01}$ & ratio of $Z_n>z_{0.001}$ \\\hline 
		$\Lambda_\star=O$ & 0.050 & 0.008 & 0.002\\ 
		$\Lambda_\star=10^{-8}I_2$ & 0.065 & 0.010 & 0.002\\ 
		$\Lambda_\star=10^{-7}I_2$ & 0.257 & 0.088 & 0.016\\ 
		$\Lambda_\star=10^{-6}I_2$ & 1.000 & 1.000 & 1.000\\ 
		$\Lambda_\star=10^{-5}I_2$ & 1.000 & 1.000 & 1.000\\ 
		$\Lambda_\star=10^{-4}I_2$ & 1.000 & 1.000 & 1.000\\
	\end{tabular}
\end{table}
\begin{table}
	\caption{comparison of diffusion estimators with small noise (section 4.1)}
	\centering
	\begin{tabular}{c|cc|cc}
		& \multicolumn{2}{c|}{$\hat{\alpha}_{1,LMM}\ (1)$} & \multicolumn{2}{c}{$\hat{\alpha}_{1,LGA}\ (1)$}\\
		& mean & (SD) & mean & (SD) \\\hline
		$\Lambda_{\star}=O$  &  0.99824985 & ( 0.00886499 ) &  1.00386452 & ( 0.00655552 ) \\
		$\Lambda_{\star}=10^{-8}I_2$  &  0.99824991 & ( 0.00886493 ) &  1.00402408 & ( 0.00664984 ) \\
		$\Lambda_{\star}=10^{-7}I_2$  &  0.99825004 & ( 0.0088648 ) &  1.00545845 & ( 0.00759832 ) \\
		$\Lambda_{\star}=10^{-6}I_2$  &  0.9982504 & ( 0.00886444 ) &  1.01968251 & ( 0.02036034 ) \\
		$\Lambda_{\star}=10^{-5}I_2$  &  0.99825186 & ( 0.00886329 ) &  1.15201362 & ( 0.15208385 ) \\
		$\Lambda_{\star}=10^{-4}I_2$  &  0.9982576 & ( 0.00886065 ) &  2.04583439 & ( 1.04583867 )
	\end{tabular}\\
	\begin{tabular}{c|cc|cc}
		& \multicolumn{2}{c|}{$\hat{\alpha}_{2,LMM}\ (0.1)$} & \multicolumn{2}{c}{$\hat{\alpha}_{2,LGA}\ (0.1)$}\\
		& mean & (SD) & mean & (SD) \\\hline
		$\Lambda_{\star}=O$  &  0.09724903 & ( 0.00662529 ) &  0.09886816 & ( 0.00655923 ) \\
		$\Lambda_{\star}=10^{-8}I_2$  &  0.09724902 & ( 0.0066253 ) &  0.09885307 & ( 0.0065606 ) \\
		$\Lambda_{\star}=10^{-7}I_2$  &  0.09724899 & ( 0.00662534 ) &  0.09871406 & ( 0.00657648 ) \\
		$\Lambda_{\star}=10^{-6}I_2$  &  0.09724893 & ( 0.00662545 ) &  0.09734969 & ( 0.00688858 ) \\
		$\Lambda_{\star}=10^{-5}I_2$  &  0.09724854 & ( 0.00662592 ) &  0.08624082 & ( 0.01486974 ) \\
		$\Lambda_{\star}=10^{-4}I_2$  &  0.09724738 & ( 0.00662794 ) &  0.04868686 & ( 0.05142515 )
	\end{tabular}\\
	\begin{tabular}{c|cc|cc}
		& \multicolumn{2}{c|}{$\hat{\alpha}_{3,LMM}\ (1)$} & \multicolumn{2}{c}{$\hat{\alpha}_{3,LGA}\ (1)$}\\
		& mean & (SD) & mean & (SD) \\\hline
		$\Lambda_{\star}=O$  &  0.99852799 & ( 0.00884562 ) &  1.01075847 & ( 0.01572314 ) \\
		$\Lambda_{\star}=10^{-8}I_2$  &  0.99852801 & ( 0.00884561 ) &  1.01091749 & ( 0.01583072 ) \\
		$\Lambda_{\star}=10^{-7}I_2$  &  0.99852807 & ( 0.00884559 ) &  1.01234286 & ( 0.01683396 ) \\
		$\Lambda_{\star}=10^{-6}I_2$  &  0.9985282 & ( 0.00884551 ) &  1.02647333 & ( 0.02877967 ) \\
		$\Lambda_{\star}=10^{-5}I_2$  &  0.99852891 & ( 0.00884529 ) &  1.15802723 & ( 0.15834474 ) \\
		$\Lambda_{\star}=10^{-4}I_2$  &  0.99853216 & ( 0.00884527 ) &  2.04915044 & ( 1.04916696 )
	\end{tabular}\qquad\\\qquad\\
\end{table}
\begin{table}
	\caption{comparison of drift estimators with small noise (section 4.1) (1)}
	\centering
	\begin{tabular}{c|cc|cc}
		& \multicolumn{2}{c|}{$\hat{\beta}_{1,LMM}\ (-1)$} & \multicolumn{2}{c}{$\hat{\beta}_{1,LGA}\ (-1)$}\\
		& mean & (SD) & mean & (SD) \\\hline
		$\Lambda_{\star}=O$  &  -1.07745389 & ( 0.20541201 ) &  -1.09906229 & ( 0.21557911 ) \\
		$\Lambda_{\star}=10^{-8}I_2$  &  -1.07745439 & ( 0.20541243 ) &  -1.09939945 & ( 0.21575873 ) \\
		$\Lambda_{\star}=10^{-7}I_2$  &  -1.07745489 & ( 0.20541362 ) &  -1.10251229 & ( 0.21772954 ) \\
		$\Lambda_{\star}=10^{-6}I_2$  &  -1.07745561 & ( 0.20541561 ) &  -1.13381666 & ( 0.23865148 ) \\
		$\Lambda_{\star}=10^{-5}I_2$  &  -1.07745453 & ( 0.20541629 ) &  -1.44795936 & ( 0.51414853 ) \\
		$\Lambda_{\star}=10^{-4}I_2$  &  -1.07745019 & ( 0.20541843 ) &  -4.5950918 & ( 3.6829753 )
	\end{tabular}
	\begin{tabular}{c|cc|cc}
		& \multicolumn{2}{c|}{$\hat{\beta}_{2,LMM}\ (-0.1)$} & \multicolumn{2}{c}{$\hat{\beta}_{2,LGA}\ (-0.1)$}\\
		& mean & (SD) & mean & (SD) \\\hline
		$\Lambda_{\star}=O$  &  -0.09696664 & ( 0.19396176 ) &  -0.10388415 & ( 0.19861024 ) \\
		$\Lambda_{\star}=10^{-8}I_2$  &  -0.0969669 & ( 0.19396159 ) &  -0.10383271 & ( 0.1986811 ) \\
		$\Lambda_{\star}=10^{-7}I_2$  &  -0.09696699 & ( 0.19396148 ) &  -0.10351819 & ( 0.19892269 ) \\
		$\Lambda_{\star}=10^{-6}I_2$  &  -0.09696965 & ( 0.19396442 ) &  -0.10023654 & ( 0.20165131 ) \\
		$\Lambda_{\star}=10^{-5}I_2$  &  -0.09697097 & ( 0.1939614 ) &  -0.06696848 & ( 0.23358494 ) \\
		$\Lambda_{\star}=10^{-4}I_2$  &  -0.09697883 & ( 0.19395681 ) &  0.25865695 & ( 0.6942442 )
	\end{tabular}
	\begin{tabular}{c|cc|cc}
		& \multicolumn{2}{c|}{$\hat{\beta}_{3,LMM}\ (-0.1)$} & \multicolumn{2}{c}{$\hat{\beta}_{3,LGA}\ (-0.1)$}\\
		& mean & (SD) & mean & (SD) \\\hline
		$\Lambda_{\star}=O$  &  -0.09690757 & ( 0.19419626 ) &  -0.10499729 & ( 0.19540939 ) \\
		$\Lambda_{\star}=10^{-8}I_2$  &  -0.09690763 & ( 0.19419608 ) &  -0.10496151 & ( 0.19544103 ) \\
		$\Lambda_{\star}=10^{-7}I_2$  &  -0.09690675 & ( 0.1941972 ) &  -0.10464195 & ( 0.19574205 ) \\
		$\Lambda_{\star}=10^{-6}I_2$  &  -0.09690773 & ( 0.19419702 ) &  -0.1012991 & ( 0.19850394 ) \\
		$\Lambda_{\star}=10^{-5}I_2$  &  -0.09690361 & ( 0.19419826 ) &  -0.06856264 & ( 0.23062938 ) \\
		$\Lambda_{\star}=10^{-4}I_2$  &  -0.09689737 & ( 0.19419747 ) &  0.25619939 & ( 0.69297105 )
	\end{tabular}
\end{table}

\begin{table}[!ht]
	\caption{comparison of drift estimators with small noise (section 4.1) (2)}
	\centering
	\begin{tabular}{c|cc|cc}
		& \multicolumn{2}{c|}{$\hat{\beta}_{4,LMM}\ (-1)$} & \multicolumn{2}{c}{$\hat{\beta}_{4,LGA}\ (-1)$}\\
		& mean & (SD) & mean & (SD) \\\hline
		$\Lambda_{\star}=O$  &  -1.07104327 & ( 0.20327131 ) &  -1.09418662 & ( 0.21647653 ) \\
		$\Lambda_{\star}=10^{-8}I_2$  &  -1.07104339 & ( 0.20327139 ) &  -1.09452347 & ( 0.21671314 ) \\
		$\Lambda_{\star}=10^{-7}I_2$  &  -1.07104315 & ( 0.20327116 ) &  -1.09769575 & ( 0.21858651 ) \\
		$\Lambda_{\star}=10^{-6}I_2$  &  -1.07104262 & ( 0.20327284 ) &  -1.12882247 & ( 0.23877234 ) \\
		$\Lambda_{\star}=10^{-5}I_2$  &  -1.07104344 & ( 0.20327336 ) &  -1.4409219 & ( 0.50979684 ) \\
		$\Lambda_{\star}=10^{-4}I_2$  &  -1.07104396 & ( 0.20327515 ) &  -4.56995034 & ( 3.66041359 )
	\end{tabular}\\
	\begin{tabular}{c|cc|cc}
		& \multicolumn{2}{c|}{$\hat{\beta}_{5,LMM}\ (1)$} & \multicolumn{2}{c}{$\hat{\beta}_{5,LGA}\ (1)$}\\
		& mean & (SD) & mean & (SD) \\\hline
		$\Lambda_{\star}=O$  &  1.06540697 & ( 0.28270108 ) &  1.09393868 & ( 0.29249589 ) \\
		$\Lambda_{\star}=10^{-8}I_2$  &  1.06540771 & ( 0.28270021 ) &  1.09421427 & ( 0.29263493 ) \\
		$\Lambda_{\star}=10^{-7}I_2$  &  1.06540748 & ( 0.28270061 ) &  1.09676246 & ( 0.29409259 ) \\
		$\Lambda_{\star}=10^{-6}I_2$  &  1.06541101 & ( 0.28270294 ) &  1.12218182 & ( 0.309151 ) \\
		$\Lambda_{\star}=10^{-5}I_2$  &  1.06540409 & ( 0.28270572 ) &  1.37801129 & ( 0.51477587 ) \\
		$\Lambda_{\star}=10^{-4}I_2$  &  1.06539429 & ( 0.28270944 ) &  3.94023887 & ( 3.11535507 )
	\end{tabular}\\
	\begin{tabular}{c|cc|cc}
		& \multicolumn{2}{c|}{$\hat{\beta}_{6,LMM}\ (1)$} & \multicolumn{2}{c}{$\hat{\beta}_{6,LGA}\ (1)$}\\
		& mean & (SD) & mean & (SD) \\\hline
		$\Lambda_{\star}=O$  &  1.05919339 & ( 0.27975853 ) &  1.07133789 & ( 0.28745577 ) \\
		$\Lambda_{\star}=10^{-8}I_2$  &  1.0591935 & ( 0.27975834 ) &  1.07158363 & ( 0.28763788 ) \\
		$\Lambda_{\star}=10^{-7}I_2$  &  1.05919364 & ( 0.27975865 ) &  1.07419761 & ( 0.28889439 ) \\
		$\Lambda_{\star}=10^{-6}I_2$  &  1.05919376 & ( 0.27975905 ) &  1.0993641 & ( 0.30222648 ) \\
		$\Lambda_{\star}=10^{-5}I_2$  &  1.0591954 & ( 0.279757 ) &  1.35275129 & ( 0.49735104 ) \\
		$\Lambda_{\star}=10^{-4}I_2$  &  1.05920135 & ( 0.27974785 ) &  3.89597877 & ( 3.0734647 ) 
	\end{tabular}
\end{table}

\begin{table}
	\caption{Estimators with large noise (section 4.2)}
	\centering
		\begin{tabular}{c|c|cc|cc}
			& & \multicolumn{2}{c|}{LMM} & \multicolumn{2}{c}{LGA}\\
			& true value & mean & (SD) & mean & (SD)\\\hline
			$\mathrm{\hat{\Lambda}^{1,1}}$ & $\mathrm{1}$ & $\mathrm{1.000106}$ & $\mathrm{(0.001678)}$ 
			& - & - \\ 
			$\mathrm{\hat{\Lambda}^{1,2}}$ & $\mathrm{0}$ & $\mathrm{1.796561\times10^{-5}}$ & $\mathrm{(0.001226)}$
			& - & - \\ 
			$\mathrm{\hat{\Lambda}^{2,2}}$ & $\mathrm{1}$ & $\mathrm{1.000030}$ & $\mathrm{(0.001826)}$ 
			& - & - \\
			$\hat{\alpha}_1$ & $\mathrm{1}$ & $\mathrm{1.017632}$ & $\mathrm{(0.03965)}$ 
			& $\mathrm{178.068993}$ & $\mathrm{(177.0733)}$ \\
			$\hat{\alpha}_2$ & $\mathrm{0.1}$ & $\mathrm{0.09688719}$ & $\mathrm{(0.02594)}$ 
			& $\mathrm{0.31344261}$ & $\mathrm{(9.9737)}$\\
			$\hat{\alpha}_3$ & $\mathrm{1}$ & $\mathrm{1.018471}$ & $\mathrm{(0.03846)}$ 
			& $\mathrm{177.962836}$ & $\mathrm{(176.9738)}$\\
			$\hat{\beta}_1$ & $\mathrm{-1}$ & $\mathrm{-1.052772}$ & $\mathrm{(0.1915)}$ 
			& $\mathrm{3.51\times10^7}$ & $\mathrm{(1.11\times 10^9)}$ \\
			$\hat{\beta}_2$ & $\mathrm{-0.1}$ & $\mathrm{-0.1096751}$ & $\mathrm{(0.2008)}$ 
			& $\mathrm{1.37\times10^8}$ & $\mathrm{(4.34\times 10^9)}$\\
			$\hat{\beta}_3$ & $\mathrm{-0.1}$ & $\mathrm{-0.08940294}$ & $\mathrm{(0.1916)}$ 
			& $\mathrm{1.27\times10^8}$ & $\mathrm{(4.03\times 10^9)}$\\
			$\hat{\beta}_4$ & $\mathrm{-1}$ & $\mathrm{-1.051771}$ & $\mathrm{(0.1975)}$ 
			& $\mathrm{-4.57\times10^7}$ & $\mathrm{(1.44\times 10^9)}$\\
			$\hat{\beta}_5$ & $\mathrm{1}$ & $\mathrm{1.037821}$ & $\mathrm{(0.2786)}$ 
			& $\mathrm{3.89\times10^6}$ & $\mathrm{(1.23\times 10^8)}$\\
			$\hat{\beta}_6$ & $\mathrm{1}$ & $\mathrm{1.049327}$ & $\mathrm{(0.2770)}$ 
			& $\mathrm{1.57\times10^7}$ & $\mathrm{(4.96\times 10^8)}$\\
		\end{tabular}
\end{table}

\clearpage

\subsection{Case of large noise}
Secondly we consider the problem with the identical setting as the previous one except for the variance of noise. We set the variance as $\Lambda_{\star}=I_2$ which is much larger than those in the previous subsection. In simulation, the empirical power of the test for noise detection is 1. We compare the estimation with our method (local mean method, LMM) and that with local Gaussian approximation (LGA) again.

Obviously all the estimators with LMM
dominate the others. Moreover, the standard deviations of our estimators are close to those with settings of small noise 
in the subsection above. It shows that our estimator is robust even if the variance of noise is so large that we cannot imagine the undermined diffusion process seemingly.

\section{Real data analysis: Met Data of NWTC}
We analyse the wind data called Met Data provided by National Wind Technology Center in United States. Met Data is the dataset recording several quantities related to wind such as velocity, speed, and temperature at the towers named M2, M4 and M5 with recording facilities in some altitudes. We especially focus the 2-dimensional data with 0.05-second resolution representing wind velocity labelled Sonic x and Sonic y (119M) at the M5 tower, from 00:00:00 on 1st July, 2017 to 20:00:00 on 5th July, 2017. For detail, see \citep{NWTC}. 
We fit the 2-dimensional Ornstein-Uhlenbeck process such that
\begin{align}
\dop \crotchet{\begin{matrix}
	X_t\\
	Y_t
	\end{matrix}}=\parens{\crotchet{\begin{matrix}
		\beta_1 & \beta_3\\
		\beta_2 & \beta_4
		\end{matrix}}\crotchet{\begin{matrix}
		X_t\\
		Y_t
		\end{matrix}}+\crotchet{\begin{matrix}
		\beta_5\\
		\beta_6
		\end{matrix}}}\dop t
+ \crotchet{\begin{matrix}
	\alpha_1 & \alpha_2\\
	\alpha_2 & \alpha_3
	\end{matrix}}\dop w_t,\ \crotchet{\begin{matrix}
	X_0\\
	Y_0
	\end{matrix}}=\crotchet{\begin{matrix}
	-2.53\\
	0.36
	\end{matrix}},
\end{align}
where $\crotchet{X_0,Y_0}^T$ equals to 
the initial 
value of the data. We summarise some relevant quantities as follows.
\begin{table}[h]
	\begin{center}
		\caption{Relevant Quantities in Section 5}
		\begin{tabular}{c|c}
			quantity & approximation \\\hline
			$n$ & 8352000\\
			$h_n$ & $6.944444\times 10^{-6}$\\
			$T_n$ &  58\\
			$nh_n^2$ & $4.027778\times 10^{-4}$\\
			$\tau$ & $2$\\
			$p_n$ & $379$\\
			$k_n$ & $22036$\\
			$\Delta_n$ & $2.631944\times 10^{-3}$\\
			$k_n\Delta_n^2$ & $0.1526463$
		\end{tabular}
	\end{center}
\end{table}
We have taken 2 hours as the time unit and fixed $\tau=2$.

Our test for noise detection results in $Z=441.7846$ and $p<10^{-16}$; therefore, for any $\alpha\ge 10^{-16}$, the alternative hypothesis $\Lambda\neq O$ is adopted. Our estimator gives the fitting such that
\begin{align}
\dop \crotchet{\begin{matrix}
	X_t\\
	Y_t
	\end{matrix}}
&=\parens{\crotchet{\begin{matrix}
		-3.77 & -0.32\\
		-0.40 & -5.01
		\end{matrix}}
	\crotchet{\begin{matrix}
		X_t\\
		Y_t
		\end{matrix}}
	+ \crotchet{\begin{matrix}
		3.60\\
		-2.54
		\end{matrix}}}\dop t+\crotchet{\begin{matrix}
	13.41 & -0.29\\
	-0.29 & 12.62
	\end{matrix}}\dop w_t,
\end{align}
with $\parens{X_0,Y_0}=\parens{-2.53,0.36}$ and the estimation of the noise variance
\begin{align}
\hat{\Lambda}_n=\crotchet{\begin{matrix}
	6.67\times 10^{-3} & 3.75\times 10^{-5}\\
	3.75\times 10^{-5} & 6.79\times 10^{-3}\\
	\end{matrix}};
\end{align} and the diffusion fitting with LGA method which is asymptotic efficient if $\Lambda=O$ gives
\begin{align}
\dop \crotchet{\begin{matrix}
	X_t\\
	Y_t
	\end{matrix}}
&=\parens{\crotchet{\begin{matrix}
		-67.53 &  -9.29\\
		-10.37 & -104.45
		\end{matrix}}
	\crotchet{\begin{matrix}
		X_t\\
		Y_t
		\end{matrix}}
	+ \crotchet{\begin{matrix}
		63.27\\
		-50.24
		\end{matrix}}}\dop t+\crotchet{\begin{matrix}
	43.82 & 0.13\\
	0.13 & 44.22
	\end{matrix}}\dop w_t
\end{align}
with the same initial value.
What we see here is that these estimators give obviously different values with the data. If $\Lambda=O$, then we should have the reasonably similar values to each other. Since we have already obtained the result $\Lambda\neq O$, there is no reason to regard the latter estimate should be adopted.

\section{Conclusion}

Our contribution is composed of three parts:
proofs of the asymptotic properties
for adaptive estimation of diffusion-plus-noise models and noise detection test, the simulation study of the 
asymptotic results
developed above, 
and the real data analysis showing that there exists situation where 
the proposed method
should be adopted. 
The adaptive ML-type estimators 
introduced in Section 3.1 are so simple that it is only necessary for us to optimise the quasi-likelihood functions quite similar to the Gaussian likelihood after we compute the much simpler estimator for the variance of noise. The test for noise detection is nonparametric; therefore, there is no need to set any model structure or quantities other than $\tau$ and time unit. We could check our methodology works well in simulation section regardless of the size of variance of noise : the estimators could perform better than or at least as well as LGA method. The real data analysis shows that our methodology is certainly necessary to analyse some high-frequency data.

As mentioned in the introduction, high-frequency setting of observation can relax some complexness and difficulty of state-space modelling. It results in a simple and unified methodology for both linear and nonlinear models since we can write the quasi-likelihood functions whether the model is linear or not. The innovation in state-space model can be dependent on the latent process itself; therefore, we can let the processes be with fat-tail which has been regarded as a stylised fact in financial econometrics these decades. The increase in amount of real-time data seen today will continue 
at so brisk a pace that diffusion-plus-noise modelling with these desirable properties will gain more usefulness in wide range of situations.

\begin{figure}[h]
	\centering
	\includegraphics[bb=0 0 720 480,width=10cm]{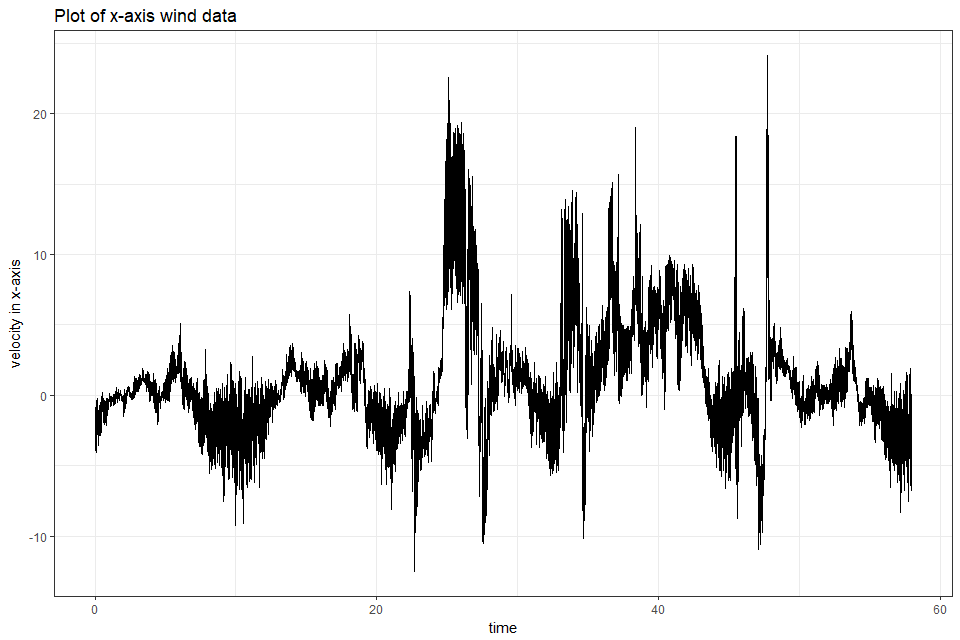}
	\caption{plot of x-axis, Met Data}
	\centering
	\includegraphics[bb=0 0 720 480,width=10cm]{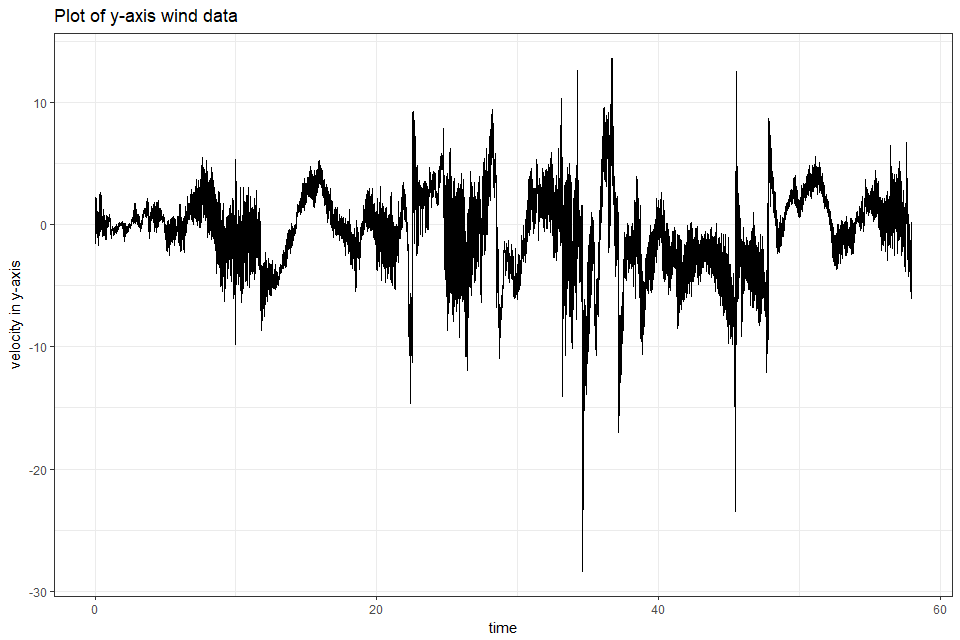}
	\caption{plot of y-axis, Met Data}
\end{figure}

\clearpage

\section{Proofs}

We set some notations which only appear in the proof section.
\begin{enumerate}
	\item Let us denote some $\sigma$-fields such that $\mathcal{G}_t:=\sigma(w_s;s\le t,x_0)$, $\mathcal{G}_j^n:=\mathcal{G}_{j\Delta_n}$, $\mathcal{A}_j^n:=\sigma(\epsilon_{ih_n};i\le jp_n-1)$, $\mathcal{H}_j^n:=\mathcal{G}_j^n\vee \mathcal{A}_j^n$.
	\item We define the following $\Re^r$-valued random variables which appear in the expansion:
	\begin{align*}
	\zeta_{j+1,n}&=\frac{1}{p_n}\sum_{i=0}^{p_n-1}\int_{j\Delta_n+ih_n}^{(j+1)\Delta_n}\dop w_s,\\ 
	\zeta_{j+2,n}'&=\frac{1}{p_n}\sum_{i=0}^{p_n-1}\int_{(j+1)\Delta_n}^{(j+1)\Delta_n+ih_n}\dop w_s,\\
	\xi_{j,n}&=\frac{1}{\Delta_n^{3/2}}\int_{j\Delta_n}^{(j+1)\Delta_n}(s-j\Delta_n)\dop w_s,\\
	\xi_{j+1,n}'&=\frac{1}{\Delta_n^{3/2}}\int_{(j+1)\Delta_n}^{(j+2)\Delta_n}((j+2)\Delta_n-s)\dop w_s,\\
	\xi_{i+1,j,n}'&=\frac{1}{h_n^{3/2}}\int_{j\Delta_n+(i+1)h_n}^{j\Delta_n+(i+2)h_n}(j\Delta_n+(i+2)h_n-s)\dop w_s.
	\end{align*}
	\item $I_{j,k,n}:=I_{j,k}=[j\Delta_n+kh_n,j\Delta_n+(k+1)h_n),\ j=0,\cdots, k_n-1,\ k=0,\cdots,p_n-1$.
	\item We set the following empirical functionals:
	\begin{align*}
	\bar{M}_n(f(\cdot,\vartheta))&:=\frac{1}{k_n}\sum_{j=0}^{k_n-1}f(\lm{Y}{j},\vartheta),\\
	\bar{D}_n(f(\cdot,\vartheta))&:=\frac{1}{k_n\Delta_n}\sum_{j=1}^{k_n-2}f(\lm{Y}{j-1},\vartheta)
	\parens{\lm{Y}{j+1}-\lm{Y}{j}-\Delta_nb(\lm{Y}{j-1})},\\
	\bar{Q}_n(A(\cdot,\vartheta))&=\frac{1}{k_n\Delta_n}\sum_{j=1}^{k_n-2}\ip{A(\lm{Y}{j-1},\vartheta)}{\parens{\lm{Y}{j+1}-\lm{Y}{j}}^{\otimes2}}.
	\end{align*}
	\item 
	Let us define $D_{j,n}:=\frac{1}{2p_n}\sum_{i=0}^{p_n-1}\parens{\parens{Y_{j\Delta_n+(i+1)h_n}-Y_{j\Delta_n+ih_n}}^{\otimes2}-\Lambda_{\star}}$, 
	and for $1\le l_2\le l_1\le d$,$D_n^{l_1,l_2}:= \frac{1}{k_n}\sum_{j=0}^{k_n-1}D_{j,n}^{l_1,l_2}$ and $D_n:=\left[D_n^{1,1},D_n^{2,1},\cdots, D_n^{d,d-1},D_n^{d,d}\right]=\vech\parens{\hat{\Lambda}_n-\Lambda_{\star}}$.
	\item We denote
	\begin{align*}
	&\left\{A_{\kappa}\left|\kappa=1,\cdots,m_1,\ A_\kappa=(A_\kappa^{j_1,j_2})_{j_1,j_2}\right.\right\},\\
	&\left\{f_{\lambda}\left|\lambda=1,\cdots,m_2,\ f_\lambda=(f^1_\lambda,\cdots,f^d_\lambda)\right.\right\},
	\end{align*}
	which are sequences of $\Re^d\otimes \Re^d$-valued functions and $\Re^d$-valued ones such that the components of themselves and their derivatives with respect to $x$ are polynomial growth functions for all $\kappa$ and $\lambda$.
	\item 
	Let us define
	\begin{align*}
	&\left\{A_{\kappa,n}(x)\left|\kappa=1,\cdots,m_1,\ A_{\kappa,n}=(A_{\kappa,n}^{j_1,j_2})_{j_1,j_2}\right.\right\},\\
	&\left\{f_{\lambda,n}(x)\left|\lambda=1,\cdots,m_2,\ f_{\lambda,n}=(f^1_{\lambda,n},\cdots,f^d_{\lambda,n})\right.\right\},
	\end{align*}
	which are sequences of the functions such that the components of the functions and their derivatives with respect to $x$ are polynomial growth functions and there exist a $\Re$-valued sequence $\tuborg{v_n}_n$ s.t. $v_n\to0$ and $C>0$ such that for all $x\in\Re^d$ and for the sequences $\tuborg{A_{\kappa}}$ and $\tuborg{f_{\lambda}}$ discussed above,
	\begin{align*}
	\sum_{\kappa=1}^{m_1}\norm{A_{\kappa,n}(x)-A_{\kappa}(x)}+\sum_{\lambda=1}^{m_2}\norm{f_{\lambda,n}(x)-f_{\lambda}(x)}\le v_n\parens{1+\norm{x}^C}.
	\end{align*}
	\item Denote
	\begin{align*}
	W^{(\tau)}\parens{\tuborg{A_{\kappa}}_{\kappa},\tuborg{f_{\lambda}}_{\lambda}}:=\crotchet{\begin{matrix}
		W_1 & O & O\\
		O & W_2^{\tau}\parens{\tuborg{A_{\kappa}}_{\kappa}} & O \\
		O & O & W_3\parens{\tuborg{f_{\lambda}}_{\lambda}}
		\end{matrix}}.
	\end{align*}
\end{enumerate}

\subsection{Conditional expectation of supremum}
The following two propositions are multidimensional extensions of those in \citep{Gl00}.
\begin{proposition}\label{pro711}
	 Under (A1), for all $k\ge 1$, there exists a constant $C(k)$ such that for all $t\ge 0$,
	\begin{align*}
	\CE{\sup_{s\in[t,t+1]}\norm{X_s}^k}{\mathcal{G}_t}\le C(k)(1+\norm{X_t}^k).
	\end{align*}
\end{proposition}

\begin{proof}
	It is enough to see the case $k\ge 2$ because of H\"{o}lder's inequality. For $s\in[t,t+1]$,
	\begin{align*}
	X_s=X_t+\int_{t}^{s}b(X_u)\dop u + \int_{t}^{s}a(X_u)\dop w_u.
	\end{align*}
	Therefore, for H\"{o}lder's inequality, Burkholder-Davis-Gundy inequality and (A1), the following evaluation holds:
	\begin{align*}
	\CE{\sup_{u\in[t,s]}\norm{X_u}^k}{\mathcal{G}_t}&\le C(k)\parens{1+\norm{X_t}^k}+C(k)\int_{t}^{s}\CE{\norm{X_u}^k}{\mathcal{G}_t}\dop u.
	\end{align*}
	By putting $\phi(s)=\CE{\sup_{u\in[t,s]}\norm{X_u}^k}{\mathcal{G}_t}$, we obtain
	\begin{align*}
	\phi(s)\le C(k)\parens{1+\norm{X_t}^k}+C(k)\int_{t}^{s}\CE{\norm{X_u}^{k}}{\mathcal{G}_t}\dop u\le C(k)\parens{1+\norm{X_t}^k}+C(k)\int_{t}^{s}\phi(s)\dop u
	\end{align*}
	and Gronwall's inequality leads to
	\begin{align*}
	\phi(s)\le C(k)(1+\norm{X_t}^k)e^{C(k)(s-t)}
	\end{align*}
	and this verifies the statement.
\end{proof}

\begin{proposition}\label{pro712} Under (A1) and for a function $f$ whose components are in $\mathcal{C}^1(\Re^d)$, assume that there exist $C>0$ such that
\begin{align*}
	\norm{f'(x)}&\le C(1+\norm{x})^C.
\end{align*}
Then for any $k\in\mathbf{N}$,
\begin{align*}
	\CE{\sup_{s\in[j\Delta_n,(j+1)\Delta_n]}\norm{f(X_s)-f(X_{j\Delta_n})}^k}{\mathcal{G}_j^n}\le C(k)\Delta_n^{k/2}\parens{1+\norm{X_{j\Delta_n}}^{C(k)}}.
\end{align*}
Especially for $f(x)=x$,
\begin{align*}
	\CE{\sup_{s\in[j\Delta_n,(j+1)\Delta_n]}\norm{X_s-X_{j\Delta_n}}^k}{\mathcal{G}_j^n}\le C(k)\Delta_n^{k/2}\parens{1+\norm{X_{j\Delta_n}}^{k}}.
\end{align*}
\end{proposition}

\begin{proof}
In the first place, we consider the case $f(x)=x$ and define the following random variable such that
\begin{align*}
	\delta_{j,n}:=\sup_{s\in[j\Delta_n,(j+1)\Delta_n]}\norm{X_s-X_{j\Delta_n}}.
\end{align*}
Then by Proposition \ref*{pro711}, we obtain the following evaluation of the conditional expectation such that
\begin{align*}
	\CE{\delta_{j,n}^k}{\mathcal{G}_j^n}&\le C(k)\Delta_n^{k/2}\parens{1+\norm{X_{j\Delta_n}}^{k}}
\end{align*}
because of BDG inequality, H\"{o}lder inequality and (A1).
Therefore, we could have obtained the case $f(x)=x$. Next, we examine $f$ satisfying the conditions and set the functional
\begin{align*}
	\delta_{j,n}(f)=\sup_{s\in[j\Delta_n,(j+1)\Delta_n]}\norm{f(X_s)-f(X_{j\Delta_n})}.
\end{align*}
Then for Taylor's theorem and the condition for $f'$, we obtain
\begin{align*}
	\delta_{j,n}(f)\le \sup_{s\in[j\Delta_n,(j+1)\Delta_n]}C\parens{1+\norm{X_s}}^C\delta_{j,n}.
\end{align*}
Because of Proposition \ref*{pro711}, H\"{o}lder inequality and Taylor's theorem, we have the following evaluation
\begin{align*}
	\CE{\delta_{j,n}^k(f)}{\mathcal{G}_j^n}\le C(k)\Delta_n^{k/2}\parens{1+\norm{X_{j\Delta_n}}^{C(k)}}
\end{align*}
and we obtain the result.
\end{proof}

\begin{proposition}\label{pro713}
	Under (A1), for all $t_3\ge t_2\ge t_1\ge 0$ where there exists a constant $C$ such that $t_3-t_1\le C$ and $l\ge 2$, we have
	\begin{align*}
	\mathrm{(i)}\ &\sup_{s_1,s_2\in[t_1,t_2]}\norm{\CE{b(X_{s_1})-b(X_{s_2})}{\mathcal{G}_{t_1}}}\le C(t_2-t_1)
	\parens{1+\norm{X_{t_1}}^3},\\
	\mathrm{(ii)}\ &\sup_{s_1,s_2\in[t_1,t_2]}\norm{\CE{a(X_{s_1})-a(X_{s_2})}{\mathcal{G}_{t_1}}}\le C(t_2-t_1)
	\parens{1+\norm{X_{t_1}}^3},\\
	\mathrm{(iii)}\ &\norm{\CE{\int_{t_2}^{t_3}\parens{b(X_s)-b(X_{t_2})}\dop s}{\mathcal{G}_{t_1}}}\le C(t_3-t_2)^2\parens{1+\CE{\norm{X_{t_2}}^6}{\mathcal{G}_{t_1}}}^{1/2},\\
	\mathrm{(iv)}\ &\CE{\norm{\int_{t_2}^{t_3}\parens{b(X_s)-b(X_{t_2})}\dop s}^l}{\mathcal{G}_{t_1}}
	\le C(l)(t_3-t_2)^{3l/2}\parens{1+\CE{\norm{X_{t_2}}^{2l}}{\mathcal{G}_{t_1}}},\\
	\mathrm{(v)}\ &\CE{\norm{\int_{t_1}^{t_2}\parens{\int_{t_1}^{s}\parens{a(X_{u})-a(X_{t_1})}\dop w_u}\dop s}^l
	}{\mathcal{G}_{t_1}}\le C(l)\parens{t_2-t_1}^{2l}\parens{1+\norm{X_{t_1}}^{2l}}.
	\end{align*} 
\end{proposition}

\begin{proof} (i), (ii): Let $L$ be the infinitesimal generator of the diffusion process. Since Ito-Taylor expansion, for all $s\in [t_1,t_2]$,
	\begin{align*}
	\CE{b(X_s)}{\mathcal{G}_{t_1}}=b(X_{t_1})+\int_{t_1}^{s}\CE{Lb(X_u)}{\mathcal{G}_{t_1}}\dop u
	\end{align*}
	and the second term has the evaluation
	\begin{align*}
	\sup_{s\in[t_1,t_2]}\norm{\int_{t_1}^{s}\CE{Lb(X_u)}{\mathcal{G}_{t_1}}\dop u}\le C\Delta_n\parens{1+\norm{X_{t_1}}^3}.
	\end{align*}
	Therefore, we have
	\begin{align*}
	\sup_{s_1,s_2\in[t_1,t_2]}\norm{\CE{b(X_{s_1})-b(X_{s_2})}{\mathcal{G}_{t_1}}}\le C(t_2-t_1)
	\parens{1+\norm{X_{t_1}}^3}
	\end{align*}
	and identically
	\begin{align*}
	\sup_{s_1,s_2\in[t_1,t_2]}\norm{\CE{a(X_{s_1})-a(X_{s_2})}{\mathcal{G}_{t_1}}}\le C(t_2-t_1)
	\parens{1+\norm{X_{t_1}}^3}.
	\end{align*}
	(iii): Using (i) and H\"{o}lder's inequality, we have the result.\\
	(iv): Because of Proposition \ref*{pro712} and H\"{o}lder's inequality, We obtain the proof.\\
	(v): For convexity, we have
	\begin{align*}
		&\CE{\norm{\int_{t_1}^{t_2}\parens{\int_{t_1}^{s}\parens{a(X_{u})-a(X_{t_1})}\dop w_u}\dop s}^l
		}{\mathcal{G}_{t_1}}\\
		&\le C(l)\sum_{i=1}^{d}\sum_{j=1}^{r}\CE{\abs{\int_{t_1}^{t_2}\parens{\int_{t_1}^{s}\parens{a^{i,j}(X_{u})-a^{i,j}(X_{t_1})}
					\dop w_u^j}\dop s}^{l}}{\mathcal{G}_{t_1}}.
	\end{align*}
	H\"{o}lder's inequality, Fubini's theorem, BDG theorem and Proposition \ref*{pro712} give the result.
\end{proof}

\subsection{Propositions for ergodicity and evaluations of expectation}
\begin{lemma}\label{lem721}
	Assume (A1)-(A3) hold. Let $f$ be a function in $\mathcal{C}^1(\Re^d\times\Xi)$ and assume that $f$, the components of $\partial_{x} f$ and $\nabla_{\vartheta} f$ are polynomial growth functions uniformly in $\vartheta\in\Xi$. 
	Then the following convergence holds:
	\begin{align*}
		\frac{1}{k_n}\sum_{j=0}^{k_n-1}f(X_{j\Delta_n},\vartheta)\cp \nu_0\parens{f(\cdot,\vartheta)}\ \text{uniformly in }\vartheta.
	\end{align*}
\end{lemma}

\begin{proof}
	(A2) verifies the pointwise convergence in probability such that for all $\vartheta\in\Xi$,
	\begin{align*}
		\frac{1}{k_n\Delta_n}\int_{0}^{k_n\Delta_n}f(X_s,\vartheta)\dop s\cp \nu_0(f(\cdot,\vartheta)).
	\end{align*}
	Let $D_n(\vartheta)$ be a random variable such that
	\begin{align*}
		D_n(\vartheta):=\frac{1}{k_n\Delta_n}\sum_{j=0}^{k_n-1}\int_{j\Delta_n}^{(j+1)\Delta_n}
		\parens{f(X_s,\vartheta)-f(X_{j\Delta_n},\vartheta)}\dop s.
	\end{align*}
	(A3) and Proposition \ref*{pro712} lead to for all $\vartheta\in\Xi$,
	\begin{align*}
		\sup_{j=0,\cdots,k_n-1}\E{\sup_{s\in[j\Delta_n,(j+1)\Delta_n]}\abs{f(X_s,\vartheta)-f(X_{j\Delta_n},\vartheta)}}
		\le C\Delta_n^{1/2}.
	\end{align*}
	It results in for all $\vartheta$,
	\begin{align*}
		\E{\abs{D_n(\vartheta)}}=\E{\abs{\frac{1}{k_n\Delta_n}\sum_{j=0}^{k_n-1}\int_{j\Delta_n}^{(j+1)\Delta_n}
				\parens{f(X_s,\vartheta)-f(X_{j\Delta_n},\vartheta)}\dop s}}\le C\Delta_n^{1/2}\to 0
	\end{align*}
	and hence for all $\vartheta$,
	\begin{align*}
		\frac{1}{k_n}\sum_{j=0}^{k_n-1}f(X_{j\Delta_n},\vartheta)
		=\frac{1}{k_n\Delta_n}\int_{0}^{k_n\Delta_n}f(X_s,\vartheta)\dop s-D_n(\vartheta)\cp \nu_0(f(\cdot,\vartheta)).
	\end{align*}
	With respect to the uniform convergence, the discussion is identical to that in Lemma 8 \citep{K97}.
	
\end{proof}

\subsection{Characteristics of local means}

The following propositions, lemmas and corollary are multidimensional extensions of those in \citep{Gl00} and \citep{Fa14}.

\begin{lemma}\label{lem731}
	$\xi_{j,n}$ and $\xi_{j+1,n}'$ are independent of each other and Gaussian. $\xi_{j,n}$ is $\mathcal{G}_{j+1}^n$-measurable and independent of $\mathcal{G}_{j}^n$, and $\xi_{j+1,n}'$ is $\mathcal{G}_{j+2}^n$-measurable and independent of $\mathcal{G}_{j+1}^n$. Furthermore, the evaluation of following conditional expectations holds:
	\begin{align*}
	\CE{\xi_{j,n}}{\mathcal{G}_j^n}=\CE{\xi_{j+1,n}'}{\mathcal{G}_j^n}&=\mathbf{0},\\
	\CE{\xi_{j,n}\xi_{j,n}^T}{\mathcal{G}_j^n}=\CE{\xi_{j+1,n}'\parens{\xi_{j+1,n}'}^T}{\mathcal{G}_j^n}&=\frac{1}{3}I_r,\\
	\CE{\xi_{j,n}\parens{\xi_{j,n}'}^T}{\mathcal{G}_j^n}&=\frac{1}{6}I_r.
	\end{align*}
\end{lemma}

\begin{proof}
	Measurability and independence are obvious. Note that the following integral evaluation holds:
	\begin{align*}
	\int_{j\Delta_n}^{(j+1)\Delta_n}\frac{(s-j\Delta_n)^2}{\Delta_n^3}\mathrm{d}s
	&=\int_{(j+1)\Delta_n}^{(j+2)\Delta_n}\frac{((j+2)\Delta_n-s)^2}{\Delta_n^3}\mathrm{d}s
	=\frac{1}{3}.
	\end{align*}
	Therefore, $\xi_{j,n}$ has the following distribution because of
	Wiener integral and the independence between the components of the Wiener process $(w_t)_t$:
	\begin{align*}
	\xi_{j,n}&\sim N_r\parens{\mathbf{0},\frac{1}{3}I_r},\\
	\xi_{j+1,n}'&\sim N_r\parens{\mathbf{0},\frac{1}{3}I_r}.
	\end{align*}
	In addition, the following equality holds:
	\begin{align*}
	\int_{j\Delta_n}^{(j+1)\Delta_n}\frac{(s-j\Delta_n)((j+1)\Delta_n-s)}{\Delta_n^3}\mathrm{d}s&=\frac{1}{6}.
	\end{align*}
	Therefore the independence of the components of the Wiener processes leads to the evaluation.
\end{proof}

\begin{lemma}\label{lem732}
	$\zeta_{j+1,n}$ and $\zeta_{j+1,n}'$ are $\mathcal{G}_{j+1}^n$-measurable, independent of $\mathcal{G}_{j}^n$ and Gaussian. These random variables have the following decomposition:
	\begin{align*}
	\zeta_{j+1,n}&=\frac{1}{p_n}\sum_{k=0}^{p_n-1}(k_n+1)\int_{I_{j,k}}\dop w_t,\\
	\zeta_{j+1,n}'&=\frac{1}{p_n}\sum_{k=0}^{p_n-1}(p_n-1-k)\int_{I_{j,k}}\dop w_t.
	\end{align*}
	In addition, the evaluation of the following conditional expectations holds:
	\begin{align*}
	\CE{\zeta_{j,n}}{\mathcal{G}_j^n}=\CE{\zeta_{j+1,n}'}{\mathcal{G}_j^n}&=\mathbf{0},\\
	\CE{\zeta_{j+1,n}\parens{\zeta_{j+1,n}}^T}{\mathcal{G}_j^n}&=m_n\Delta_nI_r,\\
	\CE{\zeta_{j+1,n}'\parens{\zeta_{j+1,n}'}^T}{\mathcal{G}_j^n}
	&=m_n'\Delta_nI_r,\\
	\CE{\zeta_{j+1,n}\parens{\zeta_{j+1,n}'}^T}{\mathcal{G}_j^n}&=\chi_n\Delta_nI_r,
	\end{align*}
	where $m_n=\parens{\frac{1}{3}+\frac{1}{2p_n}+\frac{1}{6p_n^2}}$, $m_n'=\parens{\frac{1}{3}-\frac{1}{2p_n}+\frac{1}{6p_n^2}}$ and $\chi_n=\frac{1}{6}\parens{1-\frac{1}{p_n^2}}$.
\end{lemma}

\begin{proof}
	Measurability and independence are trivial. The decomposition is also obvious if we consider the overlapping parts of sum of integral. 
	Note the following integral:
	\begin{align*}
		\int_{j\Delta_n}^{(j+1)\Delta_n}\parens{\sum_{k=0}^{p_n-1}\frac{k_n+1}{p_n}\mathbf{1}_{I_{j,k}}(s)}^2\dop s
		&=\Delta_n\parens{\frac{1}{3}+\frac{1}{2p_n}+\frac{1}{6p_n^2}},\\
		\int_{j\Delta_n}^{(j+1)\Delta_n}\parens{\sum_{k=0}^{p_n-1}\frac{p_n-1-k_n}{p_n}\mathbf{1}_{I_{j,k}}(s)}^2\dop s
		&=\Delta_n\parens{\frac{1}{3}-\frac{1}{2p_n}+\frac{1}{6p_n^2}}.
	\end{align*}
	For the independence of the components of the Wiener process and Wiener integral, we have
	\begin{align*}
		\zeta_{j+1,n}&\sim N\parens{\mathbf{0},\Delta_n\parens{\frac{1}{3}+\frac{1}{2p_n}+\frac{1}{6p_n^2}}I_r},\\
		\zeta_{j+1,n}'&\sim N\parens{\mathbf{0},\Delta_n\parens{\frac{1}{3}-\frac{1}{2p_n}+\frac{1}{6p_n^2}}I_r},
	\end{align*}
	and hence the first, second and third equalities hold.
	With respect to the fourth equality, the independence among the components and the independent increments of the Wiener processes lead to the evaluation.
\end{proof}

\begin{proposition}\label{pro733}
	Under (A1),
	\begin{align*}
		\frac{1}{\Delta_n}\int_{j\Delta_n}^{(j+1)\Delta_n}X_s\dop s-\lm{X}{j}{}=
		\sqrt{h_n}\parens{\frac{1}{p_n}\sum_{i=0}^{p_n-1}a(X_{j\Delta_n+ih_n})\xi_{i,j,n}'}+e_{j,n}^{(1)},
	\end{align*}
	where
	\begin{align*}
		^\exists C>0,\ \norm{\CE{e_{j,n}^{(1)}}{\mathcal{H}_j^n}}&\le Ch_n(1+\norm{X_{j\Delta_n}})\\
		^\forall l\ge 2,\ ^\exists C(l)>0,\ \CE{\norm{e_{j,n}^{(1)}}^l}{\mathcal{H}_j^n}&\le C(l)h_n^l\parens{1+\norm{X_{j\Delta_n}}^{2l}}.
	\end{align*}
\end{proposition}

\begin{proof}
	It is enough to consider the conditional expectation with respect to $\mathcal{G}_j^n$. Since $\Delta_n=p_nh_n$, as \citep{Fa14},
	\begin{align*}
		R_{j,n}&:=\frac{1}{\Delta_n}\int_{j\Delta_n}^{(j+1)\Delta_n}X_s\dop s-\lm{X}{j}\\
		&=\sqrt{h_n}\parens{\frac{1}{p_n}\sum_{k=0}^{p_n-1}a(X_{j\Delta_n+kh_n})\frac{1}{h_n^{3/2}}\int_{I_{j,k}}\parens{\int_{j\Delta_n+kh_n}^{s}\dop w_u}\dop s}\\
		&\qquad+\frac{1}{p_n}\sum_{k=0}^{p_n-1}\frac{1}{h_n}\int_{I_{j,k}}\parens{\int_{j\Delta_n+kh_n}^{s}\parens{a(X_{u})-a(X_{j\Delta_n+kh_n})}\dop w_u}\dop s\\
		&\qquad+\frac{1}{p_n}\sum_{k=0}^{p_n-1}\frac{1}{h_n}\int_{I_{j,k}}\parens{\int_{j\Delta_n+kh_n}^{s}b(X_u)\dop u}\dop s.
	\end{align*}
	The fact $tw_t=\int_{0}^{t}w_s\dop s+\int_{0}^{t}s\dop w_s$ easily derived from Ito's formula leads to the next transformation as follows:
	\begin{align*}
		\int_{I_{j,k}}\parens{\int_{j\Delta_n+kh_n}^{s}\dop w_u}\dop s=h_n^{3/2}\xi_{k,j,n}'.
	\end{align*}
	Therefore, we have
	\begin{align*}
		R_{j,n}&=
		\sqrt{h_n}\parens{\frac{1}{p_n}\sum_{k=0}^{p_n-1}a(X_{j\Delta_n+kh_n})\xi_{k,j,n}'}\\
		&\qquad+\frac{1}{p_n}\sum_{k=0}^{p_n-1}\frac{1}{h_n}\int_{I_{j,k}}\parens{\int_{j\Delta_n+kh_n}^{s}\parens{a(X_{u})-a(X_{j\Delta_n+kh_n})}\dop w_u}\dop s\\
		&\qquad+\frac{1}{p_n}\sum_{k=0}^{p_n-1}\frac{1}{h_n}\int_{I_{j,k}}\parens{\int_{j\Delta_n+kh_n}^{s}b(X_u)\dop u}
		\dop s.
	\end{align*}
	Hence it is necessary to evaluate the second term and the third one of this right hand side. Let us denote the second term as $e_{1,j,n}^{(1)}$ and the third one as $e_{2,j,n}^{(1)}$.
	
	Fubini's theorem verifies the following evaluation
	\begin{align*}
		\CE{e_{1,j,n}^{(1)}}{\mathcal{G}_j^n}=\mathbf{0}.
	\end{align*}
	Furthermore, for $l\ge 2$, because of convexity of $\norm{\cdot}^l$, Proposition \ref*{pro712} and Proposition \ref*{pro713}, we have
	\begin{align*}
		\CE{\norm{e_{1,j,n}^{(1)}}^l}{\mathcal{G}_j^n}
		\le C(l)h_n^{l}\parens{1+\norm{X_{j\Delta_n}}^{2l}}.
	\end{align*}
	With respect to $e_{2,j,n}^{(1)}$, we obtain
	\begin{align*}
		\norm{e_{2,j,n}^{(1)}}\le Ch_n\parens{1+\sup_{s\in[j\Delta_n,(j+1)\Delta_n]}\norm{X_s}}.
	\end{align*}
	Hence for $l\ge1$ and Proposition \ref*{pro711}, we have
	\begin{align*}
		\CE{\norm{e_{2,j,n}^{(1)}}^l}{\mathcal{G}_j^n}\le C(l)h_n^l\parens{1+\norm{X_{j\Delta_n}}^l}.
	\end{align*}
	Therefore we obtain the result.
\end{proof}

\begin{proposition}\label{pro734}
	Under (A1),
	\begin{align*}
		\frac{1}{\Delta_n}\int_{j\Delta_n}^{(j+1)\Delta_n}X_s\dop s-X_{j\Delta_n}=a(X_{j\Delta_n})\sqrt{\Delta_n}\xi_{j,n}'
		+e_{j,n}^{(2)},
	\end{align*}
	where
	\begin{align*}
		^\exists C>0,\ \norm{\CE{e_{j,n}^{(2)}}{\mathcal{H}_j^n}}&\le C\Delta_n\parens{1+\norm{X_{j\Delta_n}}},\\
		^\forall l\ge2,\ ^\exists C(l)>0,\ \CE{\norm{e_{j,n}^{(2)}}^l}{\mathcal{H}_j^n}&\le C(l)\Delta_n^l\parens{1+\norm{X_{j\Delta_n}}^{2l}}.
	\end{align*}
\end{proposition}

\begin{proof}
	The proof is essentially identical to the previous one.
\end{proof}

\begin{proposition}\label{pro735}
	Under (A1) and (AH), for all $j\le k_n-1$,
	\begin{align*}
		\lm{Y}{j}{}-X_{j\Delta_n}=a(X_{j\Delta_n})\sqrt{\Delta_n}\xi_{j,n}'+e_{j,n}'+\Lambda_{\star}^{1/2}\lm{\epsilon}{j},
	\end{align*}
	where 
	\begin{align*}
		^\exists C>0,\ \norm{\CE{e_{j,n}'}{\mathcal{G}_j^n}}
		&\le C\Delta_n\parens{1+\norm{X_{j\Delta_n}}},\\
		^\forall l\ge2,\ ^\exists C(l),\ \CE{\norm{e_{j,n}'}^l}{\mathcal{G}_j^n}&\le C(l)\Delta_n^{l}\parens{1+\norm{X_{j\Delta_n}}^{2l}}.
	\end{align*}
	Furthermore, for all $l\ge2$,
	\begin{align*}
		\CE{\norm{\lm{Y}{j}{}-X_{j\Delta_n}}^l}{\mathcal{H}_j^n}
		\le C(l)\Delta_n^{l/2}
			\parens{1+\norm{X_{j\Delta_n}}^{2l}}.
	\end{align*}
\end{proposition}

\begin{proof}
	It is enough to see conditional expectation with respect to $\mathcal{G}_j^n$. Note the following decomposition
	\begin{align*}
		\lm{Y}{j}{}-X_{j\Delta_n}
		&=\lm{X}{j}{}
		-\frac{1}{\Delta_n}\int_{j\Delta_n}^{(j+1)\Delta_n}X_s
		+\frac{1}{\Delta_n}\int_{j\Delta_n}^{(j+1)\Delta_n}X_s
		-X_{j\Delta_n}
		+\Lambda_{\star}^{1/2}\lm{\epsilon}{j}{}\\
		&=-\sqrt{h_n}\parens{\frac{1}{p_n}\sum_{i=0}^{p_n-1}a(X_{j\Delta_n+ih_n})\xi_{i,j,n}'}+e_{j,n}^{(1)}
		+a(X_{j\Delta_n})\sqrt{\Delta_n}\xi_{j,n}'
		+e_{j,n}^{(2)}
		+\Lambda_{\star}^{1/2}\lm{\epsilon}{j}{}.
	\end{align*}
	We define
	\begin{align*}
		r_{j,n}=\frac{1}{p_n}\sum_{i=0}^{p_n-1}a(X_{j\Delta_n+ih_n})\xi_{i,j,n}'.
	\end{align*}
	Obviously the following evaluation holds:
	\begin{align*}
		\CE{r_{j,n}}{\mathcal{G}_j^n}=\CE{\frac{1}{p_n}\sum_{i=0}^{p_n-1}a(X_{j\Delta_n+ih_n})\xi_{i,j,n}'}{\mathcal{G}_j^n}
		=\mathbf{0}.
	\end{align*}
	In addition, we check for any $l\ge 2$
	\begin{align*}
		\CE{\norm{r_{j,n}}^l}{\mathcal{G}_j^n}\le C(l)\parens{1+\norm{X_{j\Delta_n}}^l}.
	\end{align*}
	We have
	\begin{align*}
		\CE{\norm{r_{j,n}}^l}{\mathcal{G}_j^n}
		\le \frac{1}{p_n}\sum_{i=0}^{p_n-1}\CE{
			\norm{a(X_{j\Delta_n+ih_n})}^l\CE{\norm{\xi_{i,j,n}'}^l}{\mathcal{G}_{j\Delta_n+ih_n}}}{\mathcal{G}_j^n}.
	\end{align*}
	Because of Wiener integral and the evaluation
	\begin{align*}
		\int_{j\Delta_n+ih_n}^{j\Delta_n+(i+1)h_n}\frac{\parens{j\Delta_n+(i+1)h_n-s}^2}{h_n^3}=\frac{1}{3},
	\end{align*}
	$\xi_{i,j,n}$ is distributed as
	\begin{align*}
		\xi_{i,j,n}\sim N\parens{0,\frac{1}{3}I_r}.
	\end{align*}
	This result and Proposition \ref*{pro711} lead to
	\begin{align*}
	\CE{\norm{r_{j,n}}^l}{\mathcal{G}_j^n}
	&\le  C(l)\parens{1+\norm{X_{j\Delta_n}}^l}.
	\end{align*}
	Therefore we have
	\begin{align*}
		\CE{\norm{\sqrt{h_n}\parens{\frac{1}{p_n}\sum_{i=0}^{p_n-1}a(X_{j\Delta_n+ih_n})\xi_{i,j,n}'}}^l}{\mathcal{G}_j^n}
		&=\CE{h_n^{l/2}\norm{r_{j,n}}^{l}}{\mathcal{G}_j^n}\\
		&\le C(l)h_n^{l/2}\parens{1+\norm{X_{j\Delta_n}}^l}.
	\end{align*}
	Here we define
	\begin{align*}
		e_{j,n}':=-\sqrt{h_n}\parens{\frac{1}{p_n}\sum_{i=0}^{p_n-1}a(X_{j\Delta_n+ih_n})\xi_{i,j,n}'}+e_{j,n}^{(1)}+e_{j,n}^{(2)}
	\end{align*}
	and see $e_{j,n}'$ satisfies the statement. The evaluation above and Proposition \ref*{pro733}, Proposition \ref*{pro734} verify
	\begin{align*}
		\norm{\CE{e_{j,n}'}{\mathcal{G}_j^n}}
		&=\norm{\CE{-\sqrt{h_n}\parens{\frac{1}{p_n}\sum_{i=0}^{p_n-1}a(X_{j\Delta_n+ih_n})\xi_{i,j,n}'}}{\mathcal{G}_j^n}
			+\CE{e_{j,n}^{(1)}}{\mathcal{G}_j^n}+\CE{e_{j,n}^{(2)}}{\mathcal{G}_j^n}}\\
		&=\norm{\CE{e_{j,n}^{(1)}}{\mathcal{G}_j^n}+\CE{e_{j,n}^{(2)}}{\mathcal{G}_j^n}}\\
		&\le Ch_n\parens{1+\norm{X_{j\Delta_n}}}+C\Delta_n\parens{1+\norm{X_{j\Delta_n}}}.
	\end{align*}
	Note that (AH) leads to
	\begin{align*}
		\frac{h_n}{\Delta_n^2}=\frac{1}{p_n^{2-\tau}}=O(1)\ \Rightarrow\ h_n=O(\Delta_n^2).
	\end{align*}
	With this fact, for all $l\ge2$,
	\begin{align*}
		\CE{\norm{e_{j,n}'}^l}{\mathcal{G}_j^n}
		\le C(l)\Delta_n^{l}\parens{1+\norm{X_{j\Delta_n}}^{2l}}.
	\end{align*}
	In addition, since the sequence of $\epsilon_{ih_n}$ is i.i.d., $\CE{\norm{\lm{\epsilon}{j}{}}^l}{\mathcal{H}_j^n}=\E{\norm{\lm{\epsilon}{j}{}}^l}$ and
	\begin{align*}
		\CE{\norm{\lm{Y}{j}{}-X_{j\Delta_n}}^l}{\mathcal{H}_j^n}
		&\le C(l)\CE{\norm{a(X_{j\Delta_n})\sqrt{\Delta_n}\xi_{j,n}'}^l}{\mathcal{H}_j^n}
		+C(l)\CE{\norm{e_{j,n}'}^l}{\mathcal{H}_j^n}\\
		&\qquad+C(l)\CE{\norm{\Lambda_{\star}^{1/2}\lm{\epsilon}{j}{}}^l}{\mathcal{H}_j^n}\\
		&\le C(l)\parens{\Delta_n^{l/2}\parens{1+\norm{X_{j\Delta_n}}^{2l}}+\norm{\Lambda_{\star}^{1/2}}^l\E{\norm{\lm{\epsilon}{j}{}}^l}}.
	\end{align*}
	Finally Lemma \ref*{lem739} leads to
	\begin{align*}
		\CE{\norm{\lm{Y}{j}{}-X_{j\Delta_n}}^l}{\mathcal{H}_j^n}\le C(l)\Delta_n^{l/2}\parens{1+\norm{X_{j\Delta_n}}^{2l}}.
	\end{align*}It completes the proof.
\end{proof}

\begin{corollary}\label{cor736}
	Under (A1), (AH), assume the component of the function $f$ on $\Re^d\times\Xi$, $\partial_x f$ and $\partial_x^2 f$ are polynomial growth functions uniformly in $\vartheta\in\Xi$. Then there exists $C>0$ such that for all $j\le k_n-1$ and $\vartheta\in\Xi$,
	 \begin{align*}
	 	&\norm{\CE{f(\lm{Y}{j}{},\vartheta)-f(X_{j\Delta_n},\vartheta)}{\mathcal{H}_j^n}}\le C\Delta_n\parens{1+\norm{X_{j\Delta_n}}^{C}}.
	 \end{align*}
	 Moreover, for all $l\ge 2$,
	 \begin{align*}
	 	&\CE{\norm{f(\lm{Y}{j}{},\vartheta)-f(X_{j\Delta_n},\vartheta)}^l}{\mathcal{H}_j^n}\le C(l)\Delta_n^{l/2}\parens{1+\norm{X_{j\Delta_n}}^{C(l)}}.
	 \end{align*}
\end{corollary}

\begin{proof}
	It is enough to assume $q=1$. Taylor's theorem verifies the expansion such that
	\begin{align*}
		D_j&:=f(\lm{Y}{j}{},\vartheta)-f(X_{j\Delta_n},\vartheta)\\
		&=\partial_x f(X_{j\Delta_n},\vartheta)\parens{\lm{Y}{j}{}-X_{j\Delta_n}}\\
		&\qquad+\parens{\lm{Y}{j}{}-X_{j\Delta_n}}^T\parens{\int_{0}^{1}(1-u)\partial_x^2
			f\parens{X_{j\Delta_n}+u(\lm{Y}{j}{}-X_{j\Delta_n}),\vartheta}\dop u}\parens{\lm{Y}{j}{}-X_{j\Delta_n}}.
	\end{align*}
	Note the following inequality:
	\begin{align*}
		\norm{\lm{Y}{j}{}}&\le \norm{\lm{X}{j}{}}+\norm{\Lambda_{\star}^{1/2}\lm{\epsilon}{j}{}}\\
		&\le \sup_{s\in[j\Delta_n,(j+1)\Delta_n]}\norm{X_s}+\norm{\Lambda_{\star}^{1/2}}\norm{\lm{\epsilon}{j}{}},\\
		\norm{X_{j\Delta_n}}&\le \sup_{s\in[j\Delta_n,(j+1)\Delta_n]}\norm{X_s}\\
		&\le \sup_{s\in[j\Delta_n,(j+1)\Delta_n]}\norm{X_s}+\norm{\Lambda_{\star}^{1/2}}\norm{\lm{\epsilon}{j}{}}.
	\end{align*}
	Then for all $\vartheta\in\Xi$, Proposition \ref*{pro711}, Proposition \ref*{pro735} and Lemma \ref*{lem739} lead
	\begin{align*}
		\abs{\CE{D_j}{\mathcal{H}_j^n}}
		\le C\Delta_n\parens{1+\norm{X_{j\Delta_n}}^{C}}.
	\end{align*}
	For any $l\ge2$, Taylor's theorem gives the following evaluation:
	\begin{align*}
		\abs{D_j}^{l}&=\abs{\parens{\int_{0}^{1}\partial_xf\parens{X_{j\Delta_n}+u\parens{\lm{Y}{j}{}-X_{j\Delta_n}},\vartheta}\dop u}\parens{\lm{Y}{j}{}-X_{j\Delta_n}}}^l\\
		&\le C(l)\parens{1+\sup_{s\in[j\Delta_n,(j+1)\Delta_n]}\norm{X_s}^{C(l)}
			+\norm{\Lambda_{\star}^{1/2}}^{C(l)}\norm{\lm{\epsilon}{j}{}}^{C(l)}}
		\norm{\lm{Y}{j}{}-X_{j\Delta_n}}^l.
	\end{align*}
	It leads to
	\begin{align*}
		\CE{\abs{D_j}^l}{\mathcal{H}_j^n}\le C(l)\Delta_n^{l/2}\parens{1+\norm{X_{j\Delta_n}}^{C(l)}}
	\end{align*}
	and here we have the proof.
\end{proof}

\begin{proposition}\label{pro737}
	Under (A1) and (AH),
	\begin{align*}
		\lm{Y}{j+1}-\lm{Y}{j}-\Delta_nb(\lm{Y}{j})=a(X_{j\Delta_n})\parens{\zeta_{j+1,n}+\zeta_{j+2,n}'}+e_{j,n}
		+\Lambda_{\star}^{1/2}\parens{\lm{\epsilon}{j+1}-\lm{\epsilon}{j}},
	\end{align*}
	where $e_{j,n}$ is a $\mathcal{H}_{j+2}^n$-measurable random variable such that there exists $C>0$ and $C(l)>0$ for all
	 $l\ge2$ satisfying the inequalities
	\begin{align*}
		&\norm{\CE{e_{j,n}}{\mathcal{H}_j^n}}\le 
		C\Delta_n^2\parens{1+\norm{X_{j\Delta_n}}^{5}},\\
		&\CE{\norm{e_{j,n}}^l}{\mathcal{H}_j^n}\le C(l)\Delta_n^l\parens{1+\norm{X_{j\Delta_n}}^{3l}},\\
		&\norm{\CE{e_{j,n}\parens{\zeta_{j+1,n}}^T}{\mathcal{H}_j^n}}+\norm{\CE{e_{j,n}\parens{\zeta_{j+2,n}'}^T}{\mathcal{H}_j^n}}\le C\Delta_n^2\parens{1+\norm{X_{j\Delta_n}}^{3}}.
	\end{align*}
\end{proposition}

\begin{proof}
	First of all, note the decomposition
	\begin{align*}
		\lm{Y}{j+1}-\lm{Y}{j}=\lm{X}{j+1}-\lm{X}{j}+\Lambda_{\star}^{1/2}\parens{\lm{\epsilon}{j+1}-\lm{\epsilon}{j}}.
	\end{align*}
	We denote $C_j:=\lm{X}{j+1}-\lm{X}{j}$ and have
	\begin{align*}
		C_j&=\frac{1}{p_n}\sum_{k=0}^{p_n-1}\parens{X_{(j+1)\Delta_n+kh_n}-X_{j\Delta_n+kh_n}}\\
		&=\frac{1}{p_n}\sum_{k=0}^{p_n-1}\parens{X_{(j+1)\Delta_n+kh_n}-X_{(j+1)\Delta_n+(k-1)h_n}+X_{(j+1)\Delta_n+(k-1)h_n}
			-\cdots-X_{j\Delta_n+kh_n}}\\
		&=\frac{1}{p_n}\sum_{k=0}^{p_n-1}\parens{\int_{I_{j,k-1+p_n}}\dop X_s+\cdots+\int_{I_{j,k}}\dop X_s}\\
		&=\frac{1}{p_n}\sum_{k=0}^{p_n-1}\sum_{l=0}^{p_n-1}\int_{I_{j,k+l}}\dop X_s\\
		&=\frac{1}{p_n}\sum_{k=0}^{p_n-1}(k+1)\int_{I_{j,k}}\dop X_s
		+\frac{1}{p_n}\sum_{k=0}^{p_n-1}(p_n-1-k)\int_{I_{j+1,k}}\dop X_s,
	\end{align*}
	and
	\begin{align*}
		\int_{I_{j,k}}\dop X_s&=h_nb(X_{j\Delta_n+kh_n})+\int_{I_{j,k}}\parens{b(X_s)-b(X_{j\Delta_n+kh_n})}\dop s\\
		&\qquad+a(X_{j\Delta_n+kh_n})\int_{I_{j,k}}\dop w_s + \int_{I_{j,k}}\parens{a(X_s)-a(X_{j\Delta_n+kh_n})}\dop w_s.
	\end{align*}
	Using $\Delta_n=(k+1)h_n+(p_n-1-k)h_n$ for all $k$, we have the decomposition
	\begin{align*}
		&\lm{Y}{j+1}-\lm{Y}{j}-\Delta_nb(\lm{Y}{j})\\
		&=C_j-\Delta_nb(\lm{Y}{j})+\Lambda_{\star}^{1/2}\parens{\lm{\epsilon}{j+1}-\lm{\epsilon}{j}}\\
		&=\frac{1}{p_n}\sum_{k=0}^{p_n-1}(k+1)\int_{I_{j,k}}\dop X_s
		+\frac{1}{p_n}\sum_{k=0}^{p_n-1}(p_n-1-k)\int_{I_{j+1,k}}\dop X_s\\
		&\qquad-\frac{1}{p_n}\sum_{k=0}^{p_n-1}\parens{(k+1)h_n+(p_n-1-k)h_n}b(\lm{Y}{j})+\Lambda_{\star}^{1/2}\parens{\lm{\epsilon}{j+1}-\lm{\epsilon}{j}}\\
		&=a(X_{j\Delta_n})\parens{\zeta_{j+1,n}+\zeta_{j+2,n}'}+e_{j\Delta_n}+\Lambda_{\star}^{1/2}\parens{\lm{\epsilon}{j+1}-\lm{\epsilon}{j}},
	\end{align*}
	where $e_{j,n}=\sum_{l=1}^{4}e_{j,n}^{(l)},\ e_{j,n}^{(l)}=r_{j,n}^{(l)}+s_{j,n}^{(l)}$,
	\begin{align*}
		r_{j,n}^{(1)}&:=\frac{1}{p_n}\sum_{k=0}^{p_n-1}(k+1)h_n\parens{b(X_{j\Delta_n+kh_n})-b(\lm{Y}{j})},\\
		r_{j,n}^{(2)}&:=\parens{\frac{1}{p_n}\sum_{k=0}^{p_n-1}(k+1)a(X_{j\Delta_n+kh_n})\int_{I_{j,k}}\dop w_s}
		-a(X_{j\Delta_n})\zeta_{j+1,n},\\
		r_{j,n}^{(3)}&:=\frac{1}{p_n}\sum_{k=0}^{p_n-1}(k+1)\int_{I_{j,k}}\parens{b(X_s)-b(X_{j\Delta_n+kh_n})}\dop s,\\
		r_{j,n}^{(4)}&:=\frac{1}{p_n}\sum_{k=0}^{p_n-1}(k+1)\int_{I_{j,k}}\parens{a(X_s)-a(X_{j\Delta_n+kh_n})}\dop w_s,\\
		s_{j,n}^{(1)}&:=\frac{1}{p_n}\sum_{k=0}^{p_n-1}(p_n-1-k)h_n\parens{b(X_{(j+1)\Delta_n+kh_n})-b(\lm{Y}{j})},\\
		s_{j,n}^{(2)}&:=\parens{\frac{1}{p_n}\sum_{k=0}^{p_n-1}(p_n-1-k)a(X_{(j+1)\Delta_n+kh_n})\int_{I_{j+1,k}}\dop w_s}
		-a(X_{j\Delta_n})\zeta_{j+2,n}',\\
		s_{j,n}^{(3)}&:=\frac{1}{p_n}\sum_{k=0}^{p_n-1}(p_n-1-k)\int_{I_{j+1,k}}\parens{b(X_s)-b(X_{j\Delta_n+kh_n})}\dop s,\\
		s_{j,n}^{(4)}&:=\frac{1}{p_n}\sum_{k=0}^{p_n-1}(p_n-1-k)\int_{I_{j+1,k}}\parens{a(X_s)-a(X_{(j+1)\Delta_n+kh_n})}\dop w_s.
	\end{align*}

	\noindent \textbf{(Step 1):} We evaluate $\norm{\CE{e_{j,n}}{\mathcal{H}_j^n}}$. It is trivial that
	 $\CE{r_{j,n}^{(l)}}{\mathcal{H}_j^n}=\CE{s_{j,n}^{(l)}}{\mathcal{H}_j^n}=\mathbf{0}$ for $l=2,4$.
	 
	 Because the components of $b$, $\partial_xb$ and $\partial_x^2b$ are polynomial growth functions uniformly in $\theta\in\Theta$, Corollary \ref*{cor736} and Proposition \ref*{pro713} verify the evaluation
	 \begin{align*}
	 	\norm{\CE{r_{j,n}^{(1)}}{\mathcal{H}_j^n}}&\le C\Delta_n^2\parens{1+\norm{X_{j\Delta_n}}^{5}}.
	 \end{align*}
	 The next evaluation for $r_{j,n}^{(3)}$ holds because of Proposition \ref*{pro713}:
	 \begin{align*}
	 	\norm{\CE{r_{j,k}^{(3)}}{\mathcal{H}_j^n}}
 		&\le C\Delta_n^2\parens{1+\norm{X_{j\Delta_n}}^3}.
	 \end{align*}
	 Similarly, we can evaluate $s_{j,n}^{(1)}$ such as
	 \begin{align*}
	 	\norm{\CE{s_{j,n}^{(1)}}{\mathcal{H}_j^n}}\le C\Delta_n^2\parens{1+\norm{X_{j\Delta_n}}^{5}}
	 \end{align*}
	 and $s_{j,n}^{(3)}$ such as
	 \begin{align*}
	 	\norm{\CE{s_{j,n}^{(3)}}{\mathcal{H}_j^n}}\le C\Delta_n^2\parens{1+\norm{X_{j\Delta_n}}^3}.
	 \end{align*}
	 In sum up, we have
	 \begin{align*}
	 	\norm{\CE{e_{j,n}}{\mathcal{G}_j^n}}\le C\Delta_n^2\parens{1+\norm{X_{j\Delta_n}}^{5}}.
	 \end{align*}\\
	 
	 \noindent\textbf{(Step 2):} Now we evaluate $\CE{\norm{e_{j,n}}^l}{\mathcal{H}_j^n}$. Corollary \ref*{cor736} and Proposition \ref*{pro712} lead to the evaluations:
	 \begin{align*}
	 	\CE{\norm{r_{j,n}^{(1)}}^l}{\mathcal{H}_j^n}
	 	&\le C(l)\Delta_n^{l+l/2}\parens{1+\norm{X_{j\Delta_n}}^{3l}}
	 	,\\
	 	\CE{\norm{r_{j,n}^{(3)}}^l}{\mathcal{H}_j^n}
	 	&\le C(l)\Delta_n^lh_n^{l/2}\parens{1+\norm{X_{j\Delta_n}}^{2l}},\\
	 	\CE{\norm{s_{j,n}^{(1)}}^l}{\mathcal{H}_j^n}&\le C(l)\Delta_n^{l+l/2}\parens{1+\norm{X_{j\Delta_n}}^{3l}},\\
 		\CE{\norm{s_{j,n}^{(3)}}^l}{\mathcal{H}_j^n}&\le C(l)\Delta_n^{l+l/2}
 		\parens{1+\norm{X_{j\Delta_n}}^{2l}}.
	 \end{align*}
	 Because of Lemma \ref*{lem732}, we have the expression
	 \begin{align*}
	 	r_{j,n}^{(2)}&=\frac{1}{p_n}\sum_{k=0}^{p_n-1}(k+1)\parens{a(X_{j\Delta_n+kh_n})-a(X_{j\Delta_n})}\int_{I_{j,k}}\dop w_s,\\
 		s_{j,n}^{(2)}&=\frac{1}{p_n}\sum_{k=0}^{p_n-1}(p_n-1-k)\parens{a(X_{(j+1)\Delta_n+kh_n})-a(X_{j\Delta_n})}
 		\int_{I_{j+1,k}}\dop w_s,
	 \end{align*}
	 and then when we define
	 \begin{align*}
	 	f_1(s)&:=\frac{1}{p_n}\sum_{k=0}^{p_n-1}(k+1)\parens{a(X_{j\Delta_n+kh_n})-a(X_{j\Delta_n})}\mathbf{1}_{I_{j,k}}(s),\\
	 	f_2(s)&:=\frac{1}{p_n}\sum_{k=0}^{p_n-1}(p_n-1-k)\parens{a(X_{(j+1)\Delta_n+kh_n})-a(X_{j\Delta_n})}\mathbf{1}_{I_{j+1,k}}(s),
	 \end{align*}
	 $r_{j,n}^{(2)}$ and $s_{j,n}^{(2)}$ are the Ito integral of $f_1$ and that of $f_2$ respectively. Hence Proposition \ref*{pro712}, BDG inequality and H\"{o}lder's inequality give
	 \begin{align*}
	 	\CE{\norm{r_{j,n}^{(2)}}^l}{\mathcal{H}_j^n}\le C(l)\Delta_n^{l}\parens{1+\norm{X_{j\Delta_n}}^{2l}},
	 \end{align*}
	 and
	 \begin{align*}
	 	\CE{\norm{s_{j,n}^{(2)}}^l}{\mathcal{H}_j^n}&\le C(l)\Delta_n^{l}\parens{1+\norm{X_{j\Delta_n}}^{2l}}.
	 \end{align*}
	 Identically, when we define
	 \begin{align*}
	 	g_1(s)&:=\frac{1}{p_n}\sum_{k=0}^{p_n-1}(k+1)\parens{a(X_s)-a(X_{j\Delta_n+kh_n})}\mathbf{1}_{I_{j,k}}(s),\\
	 	g_2(s)&:=\frac{1}{p_n}\sum_{k=0}^{p_n-1}(p_n-1-k)\parens{a(X_s)-a(X_{(j+1)\Delta_n+kh_n})}\mathbf{1}_{I_{j+1,k}}(s),
	 \end{align*}
	 then $r_{j,n}^{(4)}$ and $s_{j,k}^{(4)}$ are the Ito integrals with respect to $g_1$ and $g_2$ and
	 \begin{align*}
	 	\CE{\norm{r_{j,n}^{(4)}}^l}{\mathcal{H}_j^n}
 		\le C(l)\Delta_n^{l}\parens{1+\norm{X_{j\Delta_n}}^{2l}},
	 \end{align*}
	 and identically
	 \begin{align*}
	  	\CE{\norm{s_{j,n}^{(4)}}^l}{\mathcal{H}_j^n}
	 	&\le C(l)\Delta_n^{l}\parens{1+\norm{X_{j\Delta_n}}^{2l}}.
	 \end{align*}
	 Then we obtain
	 \begin{align*}
	 	\CE{\norm{e_{j,n}}^l}{\mathcal{H}_j^n}&\le C(l)\Delta_n^l\parens{1+\norm{X_{j\Delta_n}}^{3l}}.
	 \end{align*}\\
	 
	 \noindent \textbf{(Step 3):} For $r_{j,n}^{(1)}\parens{\zeta_{j+1,n}}^T$, H\"{o}lder's inequality leads to
	 \begin{align*}
	 	\norm{\CE{r_{j,n}^{(1)}\parens{\zeta_{j+1,n}}^T}{\mathcal{H}_j^n}}&
	 	\le C\Delta_n^{2}\parens{1+\norm{X_{j\Delta_n}}^{3}},
	 \end{align*}
	 and the same evaluation for $r_{j,n}^{(1)}\parens{\zeta_{j+2,n}'}^T$, $r_{j,n}^{(3)}\parens{\zeta_{j+1,n}}^T$, 
	 $r_{j,n}^{(3)}\parens{\zeta_{j+2,n}'}^T$, $s_{j,n}^{(1)}\parens{\zeta_{j+1,n}}^T$,  
	 $s_{j,n}^{(1)}\parens{\zeta_{j+2,n}'}^T$, $s_{j,n}^{(3)}\parens{\zeta_{j+1,n}}^T$ and 
	 $s_{j,n}^{(3)}\parens{\zeta_{j+2,n}'}^T$ hold. Next, tower property of conditional expectation, independence of increments of
	  the Wiener process and {Proposition \ref*{pro713}} give
	 \begin{align*}
	 	\norm{\CE{r_{j,n}^{(2)}\parens{\zeta_{j+1,n}}^T}{\mathcal{H}_j^n}}
	 	\le C\Delta_n^2\parens{1+\norm{X_{j\Delta_n}}^3}.
	 \end{align*}
	 The same inequality holds for $s_{j,n}^{(2)}\parens{\zeta_{j+2,n}'}^T$. 
	 It is obvious that
	 \begin{align*}
	 	&\CE{r_{j,n}^{(2)}\parens{\zeta_{j+2,n}'}^T}{\mathcal{G}_j^n}=\CE{r_{j,n}^{(4)}\parens{\zeta_{j+2,n}'}^T}{\mathcal{G}_j^n}=\CE{s_{j,n}^{(2)}\parens{\zeta_{j+1,n}}^T}{\mathcal{G}_j^n}=\CE{s_{j,n}^{(4)}\parens{\zeta_{j+1,n}}^T}{\mathcal{G}_j^n}\\
	 	&=\mathbf{0}
	 \end{align*}
	 because of independence of increments of the Wiener process. Finally with the same argument as $\norm{\CE{r_{j,n}^{(2)}\parens{\zeta_{j+1,n}}^T}{\mathcal{H}_j^n}}$, we have
	 \begin{align*}
	 \norm{\CE{r_{j,n}^{(4)}\parens{\zeta_{j+1,n}}^T}{\mathcal{H}_j^n}}
	 \le C\Delta_n^2\parens{1+\norm{X_{j\Delta_n}}^2}
	 \end{align*} 
	 because $h_n^{1/2}=\Delta_n/\parens{h_n^{1/2}p_n}=\Delta_n/p_n^{1-\tau/2}\le C\Delta_n$. The same evaluation for 
	 $s_{j,n}^{(4)}\parens{\zeta_{j+2,n}'}^T$. Hence we obtain the result.
\end{proof}

\begin{corollary}\label{cor738}
	Under (A1) and (AH),
	\begin{align*}
	\lm{Y}{j+1}-\lm{Y}{j}-\Delta_nb(X_{j\Delta_n})=a(X_{j\Delta_n})\parens{\zeta_{j+1,n}+\zeta_{j+2,n}'}+e_{j,n}
	+\Lambda_{\star}^{1/2}\parens{\lm{\epsilon}{j+1}-\lm{\epsilon}{j}},
	\end{align*}
	where $e_{j,n}$ is a $\mathcal{H}_{j+2}^n$-measurable random variable such that there exists $C>0$ and $C(l)>0$ for all
	$l\ge2$ satisfying the inequalities
	\begin{align*}
	&\norm{\CE{e_{j,n}}{\mathcal{H}_j^n}}\le 
	C\Delta_n^2\parens{1+\norm{X_{j\Delta_n}}^{5}},\\
	&\CE{\norm{e_{j,n}}^l}{\mathcal{H}_j^n}\le C(l)\Delta_n^l\parens{1+\norm{X_{j\Delta_n}}^{3l}},\\
	&\norm{\CE{e_{j,n}\parens{\zeta_{j+1,n}}^T}{\mathcal{H}_j^n}}+\norm{\CE{e_{j,n}\parens{\zeta_{j+2,n}'}^T}{\mathcal{H}_j^n}}\le C\Delta_n^2\parens{1+\norm{X_{j\Delta_n}}^{3}}.
	\end{align*}
\end{corollary}

\begin{proof}
	It is enough to see $\Delta_nb(\lm{Y}{j})-\Delta_nb(X_{j\Delta_n})$ satisfies the evaluation for $e_{j,n}$. Corollary \ref*{cor736} and
	Proposition \ref*{pro737} give
	\begin{align*}
		&\norm{\CE{\Delta_nb(\lm{Y}{j})-\Delta_nb(X_{j\Delta_n})}{\mathcal{H}_j^n}}\le 
		C\Delta_n^2\parens{1+\norm{X_{j\Delta_n}}^{5}},\\
		&\CE{\norm{\Delta_nb(\lm{Y}{j})-\Delta_nb(X_{j\Delta_n})}^l}{\mathcal{H}_j^n}
		\le C(l)\Delta_n^l\parens{1+\norm{X_{j\Delta_n}}^{3l}}.
	\end{align*}
	With respect to the third evaluation, H\"{o}lder's inequality verifies the result.
\end{proof}

The following lemma summarises some useful evaluations for computation.
\begin{lemma}\label{lem739}
	Assume $f$ is a function whose components are in $\mathcal{C}^2(\Re^d\times\Xi)$ and the components of $f$ and $\partial_xf$ are polynomial growth functions in $\vartheta\in\Xi$. In addition, $g$ denotes a function whose components are in $\mathcal{C}^2(\Re^d)$ and that the components of $g$ and $\partial_xg$ are polynomial growth functions.  Under (A1), (A3), (A4) and (AH), the following uniform evaluation holds:
	\begin{align*}
	\mathrm{(i) }&^\forall l_1,l_2\in\mathbf{N}_0,\ 
	\sup_{j,n}\E{\sup_{\vartheta\in\Xi}\norm{f(\lm{Y}{j-1},\vartheta)}^{l_1}\parens{1+\norm{X_{j\Delta_n}}}^{l_2}}\le C(l_1,l_2),\\
	\mathrm{(ii) }&^\forall l\in \mathbf{N},\ ^\forall j\le k_n-2,\ \E{\norm{\zeta_{j+1,n}+\zeta_{j+2,n}'}^{l}}\le C(l)\Delta_n^{l/2},\\
	\mathrm{(iii) }&^\forall l\in\mathbf{N},\ ^\forall j\le k_n-1,\ \E{\norm{g(X_{(j+1)\Delta_n})-g(X_{j\Delta_n})}^{l}}\le C(l)\Delta_n^{l/2},\\
	\mathrm{(iv) }&^\forall l\in\mathbf{N},\ ^\forall j\le k_n-1,\ \E{\norm{g(\lm{Y}{j})-g(X_{j\Delta_n})}^{l}}\le C(l)\Delta_n^{l/2},\\
	\mathrm{(v) }&^\forall l\in\mathbf{N},\ ^\forall j\le k_n-2,\  \E{\norm{g(\lm{Y}{j+1})-g(\lm{Y}{j})}^{l}}\le C(l)\Delta_n^{l/2},\\
	\mathrm{(vi) }&^\forall l\in\mathbf{N},\ ^\forall j\le k_n-2,\ \E{\norm{e_{j,n}}^l}\le C(l)\Delta_n^l,\\
	\mathrm{(vii) }&^\forall l\in\mathbf{N},\ 
	\sup_{j,n}\parens{\frac{\mathbf{E}\left[\norm{\lm{\epsilon}{j}}^l\right]}{\Delta_n^{l/2}}}\le C(l).
	\end{align*}
\end{lemma}

\begin{proof}Simple computations and the results above lead to the proof.
\end{proof}

\subsection{Uniform law of large numbers}

The following propositions and theorems are multidimensional version of \citep{Fa14}.

\begin{proposition}\label{pro741} Assume $f$ is a function in $\mathcal{C}^2(\Re^d\times\Xi)$ and $f$, the components of $\partial_xf$, $\partial_x^2f$ and $\partial_{\vartheta} f$ are polynomial growth functions uniformly in $\vartheta\in\Xi$. Under (A1)-(A4), (AH),
	\begin{align*}
		\bar{M}_n(f(\cdot,\vartheta)) \cp \nu_0(f(\cdot,\vartheta))
		\text{ uniformly in }\vartheta.
	\end{align*}
\end{proposition}

\begin{proof}
	Lemma \ref*{lem721} leads to
	\begin{align*}
		\sup_{\vartheta\in\Xi}\abs{\frac{1}{k_n}\sum_{j=0}^{k_n-1}f(X_{j\Delta_n},\vartheta)-\nu_0(f(\cdot,\vartheta))}\cp 0.
	\end{align*}
	Hence it is enough to see $L^1$ convergence to $0$ of the following random variable
	\begin{align*}
		\sup_{\vartheta\in\Xi}\abs{\frac{1}{k_n}\sum_{j=0}^{k_n-1}f(\lm{Y}{j},\vartheta)-f(X_{j\Delta_n},\vartheta)}
	\end{align*}
	since it will lead to
	\begin{align*}
		\sup_{\vartheta\in\Xi}\abs{\frac{1}{k_n}\sum_{j=0}^{k_n-1}f(\lm{Y}{j},\vartheta)-\nu_0(f(\cdot,\vartheta))}& \le
		 \sup_{\vartheta\in\Xi}\abs{\frac{1}{k_n}\sum_{j=0}^{k_n-1}\parens{f(\lm{Y}{j},\vartheta)-f(X_{j\Delta_n},\vartheta)}}\\
		 &\qquad+\sup_{\vartheta\in\Xi}\abs{\frac{1}{k_n}\sum_{j=0}^{k_n-1}f(X_{j\Delta_n},\vartheta)-\nu_0(f(\cdot,\vartheta))}\\
		&\cp 0.
	\end{align*}
	We can use Taylor's theorem
	\begin{align*}
		A_j&:=\sup_{\vartheta\in\Xi}\abs{f(\lm{Y}{j},\vartheta)-f(X_{j\Delta_n},\vartheta)}\\
		&= \sup_{\vartheta\in\Xi}\abs{\parens{\int_{0}^{1}\partial_xf(X_{j\Delta_n}+u\parens{\lm{Y}{j}-X_{j\Delta_n}},\vartheta)\dop u}
			\parens{\lm{Y}{j}-X_{j\Delta_n}}}\\
		&\le C\parens{1+\sup_{s\in[j\Delta_n,(j+1)\Delta_n]}\norm{X_{s}}^C
			+\norm{\Lambda_{\star}^{1/2}}^C\norm{\lm{\epsilon}{j}}^C}
		\norm{\lm{Y}{j}-X_{j\Delta_n}}.
	\end{align*}
	Therefore because of (A4), H\"{o}lder's inequality, Proposition \ref*{pro735} and Lemma \ref*{lem739}, the next evaluation holds:
	\begin{align*}
		\CE{A_j}{\mathcal{H}_j^n}
		\le C\Delta_n^{1/2}\parens{1+\norm{X_{j\Delta_n}}^{C}}.
	\end{align*}
	(A3) leads to
	\begin{align*}
		\sup_{j}\E{A_j}\le C\Delta_n^{1/2};
	\end{align*}
	therefore, we obtain
	\begin{align*}
		\E{\sup_{\vartheta\in\Xi}\abs{\frac{1}{k_n}\sum_{j=0}^{k_n-1}\parens{f(\lm{Y}{j},\vartheta)-f(X_{j\Delta_n},\vartheta)}}}
		\to 0.
	\end{align*}
\end{proof}

\begin{theorem}\label{thm742}
	Assume $f=\parens{f^1,\cdots,f^d}$ is a function in $\mathcal{C}^2(\Re^d\times\Xi; \Re^d)$ and the components of $f$, $\partial_xf$, $\partial_x^2f$ and $\partial_{\vartheta} f$ are polynomial growth functions uniformly in $\vartheta\in\Xi$. Under (A1)-(A4), (AH),
	\begin{align*}
		\bar{D}_n(f(\cdot,\vartheta)) \cp 0\text{ uniformly in }\vartheta.
	\end{align*}
\end{theorem}

\begin{proof}
	We define the following random variables:
	\begin{align*}
		V_j^n(\vartheta)&:=f(\lm{Y}{j-1},\vartheta)\parens{\lm{Y}{j+1}-\lm{Y}{j}-\Delta_nb(\lm{Y}{j})},\\
		\tilde{D}_n(f(\cdot,\vartheta))&:=\frac{1}{k_n\Delta_n}\sum_{j=1}^{k_n-2}V_j^n(\vartheta)
	\end{align*}
	and then
	\begin{align*}
		\bar{D}_n(f(\cdot,\vartheta))
		&=\tilde{D}_n(\cdot,\vartheta)+\frac{1}{k_n}\sum_{j=1}^{k_n-2}f(\lm{Y}{j-1},\vartheta)
		\parens{b(\lm{Y}{j})-b(\lm{Y}{j-1})}.
	\end{align*}
	Hence it is enough to see the uniform convergences in probability of the first term and the second one in the right hand side. \\
	
	\noindent \textbf{(Step 1): }We can decompose the sum of $V_j^n(\vartheta)$ as follows:
	\begin{align*}
		\sum_{j=1}^{k_n-2}V_j^n(\vartheta)=\sum_{1\le 3j\le k_n-2}V_{3j}^n(\vartheta)+\sum_{1\le 3j+1\le k_n-2}V_{3j+1}^n(\vartheta)
		+\sum_{1\le 3j+2\le k_n-2}V_{3j+2}^n(\vartheta).
	\end{align*}
	To simplify notations, we only consider the first term of the right hand side and the other terms have the identical evaluation.
	Let us define the following random variables:
	\begin{align*}
		v_{3j,n}^{(1)}(\vartheta)&:=f(\lm{Y}{3j-1},\vartheta)a(X_{3j\Delta_n})\parens{\zeta_{3j+1,n}+\zeta_{3j+2,n}'},\\
		v_{3j,n}^{(2)}(\vartheta)&:=f(\lm{Y}{3j-1},\vartheta)\Lambda_{\star}^{1/2}\parens{\lm{\epsilon}{3j+1}-\lm{\epsilon}{3j}},\\
		v_{3j,n}^{(3)}(\vartheta)&:=f(\lm{Y}{3j-1},\vartheta)e_{3j,n},
	\end{align*}
	and recall Proposition \ref*{pro737} which states
	\begin{align*}
		\lm{Y}{j+1}-\lm{Y}{j}-\Delta_nb(\lm{Y}{j})=a(X_{j\Delta_n})\parens{\zeta_{j+1,n}+\zeta_{j+2,n}'}+e_{j,n}
		+\Lambda_{\star}^{1/2}\parens{\lm{\epsilon}{j+1}-\lm{\epsilon}{j}}.
	\end{align*}
	Therefore we have
	\begin{align*}
		V_{3j}^{n}(\vartheta)=v_{3j,n}^{(1)}(\vartheta)+v_{3j,n}^{(2)}(\vartheta)+v_{3j,n}^{(3)}(\vartheta).
	\end{align*}
	First of all, the pointwise convergence to 0 for all $\vartheta$ and we abbreviate $f(\cdot,\vartheta)$ as $f(\cdot)$. Since $V_{3j}^n$ is $\mathcal{H}_{3j+2}^{n}$-measurable and 
	hence $\mathcal{H}_{3j+3}^n$-measurable, the sequence of random variables $\tuborg{V_{3j}}_{1\le 3j\le k_n-2}$ are
	 $\tuborg{\mathcal{H}_{3j+3}^{n}}_{1\le 3j\le k_n-2}$-adopted, and hence it is enough to see
	\begin{align*}
		\frac{1}{k_n\Delta_n}\sum_{1\le 3j\le k_n-2}\CE{V_{3j}^{n}}{\mathcal{H}_{3(j-1)+3}^n}=\frac{1}{k_n\Delta_n}\sum_{1\le 3j\le k_n-2}\CE{V_{3j}^{n}}{\mathcal{H}_{3j}^n}&\cp 0,\tag{cp.1}\\
		\frac{1}{k_n^2\Delta_n^2}\sum_{1\le 3j\le k_n-2}\CE{\parens{V_{3j}^{n}}^2}{\mathcal{H}_{3j}^n}&\cp 0\tag{cp.2}
	\end{align*}
	because of Lemma 9 in \citep{GeJ93}. With respect to the conditional first moment of $v_{3j,n}^{(1)}$ and $v_{3j,n}^{(2)}$, it holds
	\begin{align*}
		\CE{v_{3j,n}^{(1)}}{\mathcal{H}_{3j}^n}
		&=\CE{f(\lm{Y}{3j-1})a(X_{3j\Delta_n})\parens{\zeta_{3j+1,n}+\zeta_{3j+2,n}'}}{\mathcal{H}_{3j}^n}=0,\\
		\CE{v_{3j,n}^{(2)}}{\mathcal{H}_{3j}^n}
		&=\CE{f(\lm{Y}{3j-1})\Lambda_{\star}^{1/2}\parens{\lm{\epsilon}{3j+1}-\lm{\epsilon}{3j}}}{\mathcal{H}_{3j}^n}=0.
	\end{align*}
	As for $v_{3j,n}^{(3)}$ we have
	\begin{align*}
		\abs{\CE{v_{3j,n}^{(3)}}{\mathcal{H}_{3j}^n}}\le \norm{f(\lm{Y}{3j-1})}C\Delta_n^2\parens{1+\norm{X_{3j\Delta_n}}^5}
	\end{align*}
	and Lemma \ref*{lem739} verifies
	\begin{align*}
		\E{\abs{\CE{v_{3j,n}^{(3)}}{\mathcal{H}_{3j}^n}}}
		\le C\Delta_n^2.
	\end{align*}
	As a result, we can see the following $L^1$ convergence and hence (cp.1):
	\begin{align*}
		\E{\abs{\frac{1}{k_n\Delta_n}\sum_{1\le 3j\le k_n-2}\CE{V_{3j}^n}{\mathcal{H}_{3j}^n}}}
		\to0.
	\end{align*}
	The conditional square moment of $v_{3j,n}^{(1)}$ can be evaluated as
	\begin{align*}
		\CE{\abs{v_{3j,n}^{(1)}}^2}{\mathcal{H}_{3j}^n}
		\le C\Delta_n\norm{f(\lm{Y}{3j-1})}^2\norm{a(X_{3j\Delta_n})}^2
	\end{align*}
	because of Lemma \ref*{lem739}; therefore Lemma \ref*{lem739} again gives
	\begin{align*}
		\E{\abs{\frac{1}{k_n^2\Delta_n^2}\sum_{1\le 3j\le k_n-2}\CE{\abs{v_{3j,n}^{(1)}}^2}{\mathcal{H}_{3j}^n}}}
		\to0.
	\end{align*}
	Lemma \ref*{lem739} gives the following evaluation for $v_{3j,n}^{(2)}$:
	\begin{align*}
		\CE{\abs{v_{3j,n}^{(2)}}^2}{\mathcal{H}_{3j}^n}
		\le C\Delta_n\norm{f(\lm{Y}{3j-1})}^2
	\end{align*}
	and then by Lemma \ref*{lem739}, 
	\begin{align*}
		\E{\abs{\frac{1}{k_n^2\Delta_n^2}\sum_{1\le 3j\le k_n-2}\CE{\abs{v_{3j,n}^{(2)}}^2}{\mathcal{H}_{3j}^n}}}
		\to0.
	\end{align*}
	With respect to $v_{3j,n}^{(3)}$,
	\begin{align*}
		\CE{\abs{v_{3j,n}^{(3)}}^2}{\mathcal{H}_{3j}^n}
		\le \norm{f(\lm{Y}{3j-1})}^2C\Delta_n^2\parens{1+\norm{X_{j\Delta_n}}^{6}}
	\end{align*}
	and hence by Lemma \ref*{lem739},
	\begin{align*}
		\E{\abs{\frac{1}{k_n^2\Delta_n^2}\sum_{1\le 3j\le k_n-2}\CE{\abs{v_{3j,n}^{(3)}}^2}{\mathcal{H}_{3j}^n}}}\to0.
	\end{align*}
	Therefore we obtain
	\begin{align*}
		\frac{1}{k_n^2\Delta_n^2}\sum_{1\le 3j\le k_n-2}\CE{\parens{V_{3j}^{n}}^2}{\mathcal{H}_{3j}^n}&\le \frac{1}{k_n^2\Delta_n^2}\sum_{1\le 3j\le k_n-2}C\CE{\parens{v_{3j,n}^{(1)}}^2+\parens{v_{3j,n}^{(2)}}^2+\parens{v_{3j,n}^{(3)}}^2}{\mathcal{H}_{3j}^n}\\
		&\cp 0.
	\end{align*}
	Here we have the pointwise convergence in probability of $\tilde{D}_n(f(\cdot,\vartheta))$ for all $\vartheta$ because of Lemma 9 in \citep{GeJ93}.
	
	Next, we consider the uniform convergence in probability of $\tilde{D}_n(f(\cdot,\vartheta))$. Let us define
	\begin{align*}
		S_n^{(l)}(\vartheta):=\frac{1}{k_n\Delta_n}\sum_{1\le 3j\le k_n-2}v_{3j,n}^{(l)}(\vartheta),\ l=1,2,3.
	\end{align*}
	We will see for all $l$, $S_n^{(l)}(\vartheta)$ uniformly converges to 0 in probability. Firstly we examine $S_n^{(3)}$: Lemma \ref*{lem739} gives
	\begin{align*}
		\E{\sup_{\vartheta\in\Xi}\abs{\nabla_\vartheta v_{3j,n}^{(3)}(\vartheta)}}
		\le C\Delta_n.
	\end{align*}
	Hence we obtain
	\begin{align*}
		\sup_{n\in\mathbf{N}}\E{\sup_{\vartheta\in\Xi}\abs{\nabla_\vartheta S_{n}^{(3)}(\vartheta)}}
		\le C
		<\infty.
	\end{align*}
	Therefore it holds
	\begin{align*}
		S_n^{(3)}(\vartheta)\cp0\text{ uniformly in }\vartheta
	\end{align*}
	as the discussion in \citep{K97} or Proposition A1 in \citep{Gl06}.
	
	For $S_n^{(l)},\ l=1,2$ we see the following inequalities hold \citep{IH81}: there exist $C>0$ and $\kappa>\dim \Xi$ such that
	\begin{align*}
		\mathrm{(1) }&^\forall \vartheta\in\Xi,\ ^\forall n\in\mathbf{N},\ \E{\abs{S_n^{(l)}(\vartheta)}^{\kappa}}\le C,\\
		\mathrm{(2) }&^\forall \vartheta,\vartheta'\in\Xi,\ ^\forall n\in\mathbf{N},\ 
		\E{\abs{S_n^{(l)}(\vartheta)-S_n^{(l)}(\vartheta')}^\kappa}\le C\norm{\vartheta-\vartheta'}^{\kappa}.
	\end{align*}
	Assume $\kappa=2k,\ k\in\mathbf{N}$ in the following discussion.
	
	For $l=1$, Burkholder's inequality gives that for all $\kappa$, there exists $C=C(\kappa)$ such that
	\begin{align*}
		\E{\abs{S_n^{(1)}(\vartheta)}^{\kappa}}
		\le \frac{C(\kappa)}{\parens{k_n\Delta_n}^{\kappa}}k_n^{\kappa/2-1}
		\sum_{1\le 3j\le k_n-2}\E{\abs{v_{3j,n}^{(1)}(\vartheta)}^{\kappa}}.
	\end{align*}
	Lemma \ref*{lem739} verifies
	\begin{align*}
		\E{\abs{v_{3j,n}^{(1)}(\vartheta)}^{\kappa}}
		\le C(\kappa)\Delta_n^{\kappa/2},
	\end{align*}
	and therefore
	\begin{align*}
		\E{\abs{S_n^{(1)}(\vartheta)}^{\kappa}}
		\le \frac{C(\kappa)}{\parens{k_n\Delta_n}^{\kappa/2}}.
	\end{align*}
	With respect to $\E{\abs{S_n^{(1)}(\vartheta)-S_n^{(1)}(\vartheta')}^\kappa}$, identically
	\begin{align*}
		\E{\abs{S_n^{(1)}(\vartheta)-S_n^{(1)}(\vartheta')}^\kappa}
		\le \frac{C(\kappa)}{\parens{k_n\Delta_n}^{\kappa}}k_n^{\kappa/2-1}
		\sum_{1\le 3j\le k_n-2}\E{\abs{v_{3j,n}^{(1)}(\vartheta)-v_{3j,n}^{(1)}(\vartheta')}^{\kappa}}
	\end{align*}
	and
	\begin{align*}
		\E{\abs{v_{3j,n}^{(1)}(\vartheta)-v_{3j,n}^{(1)}(\vartheta')}^{\kappa}}\le C(\kappa)\Delta_n^{\kappa/2}\norm{\vartheta-\vartheta'}^{\kappa},
	\end{align*}
	hence
	\begin{align*}
		\E{\abs{S_n^{(1)}(\vartheta)-S_n^{(1)}(\vartheta')}^\kappa}\le \frac{C(\kappa)}{\parens{k_n\Delta_n}^{\kappa/2}}\norm{\vartheta-\vartheta'}^{\kappa}.
	\end{align*}
	This result gives uniform convergence in probability of $S_n^{(1)}$. For $S_n^{(2)}$, as same as $S_n^{(1)}$, Burkholder's inequality gives
	\begin{align*}
		\E{\abs{S_n^{(2)}(\vartheta)}^{\kappa}}&\le \frac{C(\kappa)}{\parens{k_n\Delta_n}^{\kappa}}k_n^{\kappa/2-1}
		\sum_{1\le 3j\le k_n-2}\E{\abs{v_{3j,n}^{(2)}(\vartheta)}^{\kappa}}
	\end{align*}
	and Lemma \ref*{lem739} verifies
	\begin{align*}
		\E{\abs{v_{3j,n}^{(2)}(\vartheta)}^{\kappa}}
		\le C(\kappa)\Delta_n^{\kappa/2};
	\end{align*}
	therefore, the identical evaluation holds. The discussion for $\E{\abs{S_n^{(2)}(\vartheta)-S_n^{(2)}(\vartheta')}^\kappa}$ is also identical to $\E{\abs{S_n^{(1)}(\vartheta)-S_n^{(1)}(\vartheta')}^\kappa}$. It leads to uniform convergences in probability of $S_n^{(2)}$ and as a result $\tilde{D}_n(\vartheta)$.\\
	
	\noindent \textbf{(Step 2): } We see the uniform convergence in probability of the following random variable,
	\begin{align*}
		\frac{1}{k_n}\sum_{1\le 3j\le k_n-2}f(\lm{Y}{j-1},\vartheta)\parens{b(\lm{Y}{j})-b(\lm{Y}{j-1})}.
	\end{align*}
	By Lemma \ref*{lem739}, it is easily shown that
	\begin{align*}
		\E{\sup_{\vartheta\in\Xi}\abs{\frac{1}{k_n}\sum_{1\le 3j\le k_n-2}
				f(\lm{Y}{j-1},\vartheta)\parens{b(\lm{Y}{j})-b(\lm{Y}{j-1})}}}\to0.
	\end{align*}
	It completes the proof.
\end{proof}

\begin{theorem}\label{thm743}
	Assume $A=\parens{A^{i,j}}_{i,j}$ is a function in $\mathcal{C}^2(\Re^d\times\Xi;\Re^d\otimes\Re^d)$ and the components of $f$, $\partial_xf$, $\partial_x^2f$ and $\partial_{\vartheta} f$ are polynomial growth functions uniformly in $\vartheta\in\Xi$. Under (A1)-(A4), (AH), if $\tau\in(1,2)$,
	\begin{align*}
	\bar{Q}_n(A(\cdot,\vartheta))
	\cp \frac{2}{3}\nu_0\parens{\ip{A(\cdot,\vartheta)}{c(\cdot,\alpha^\star)}}\text{ uniformly in }\vartheta,
	\end{align*}
	and if $\tau=2$,
	\begin{align*}
	\bar{Q}_n(A(\cdot,\vartheta))
	\cp \frac{2}{3}\nu_0\parens{\ip{A(\cdot,\vartheta)}{c^\dagger(\cdot,\alpha^\star,\Lambda_{\star})}}
	\text{ uniformly in }\vartheta.
	\end{align*}
\end{theorem}

\begin{proof}
	Note that Proposition \ref*{pro737} gives
	\begin{align*}
		\parens{\lm{Y}{j+1}-\lm{Y}{j}}^{\otimes2}&=\parens{\lm{Y}{j+1}-\lm{Y}{j}}\parens{\lm{Y}{j+1}-\lm{Y}{j}}^T\\
		&=\parens{\Delta_nb(\lm{Y}{j})+a(X_{j\Delta_n})\parens{\zeta_{j+1,n}+\zeta_{j+2,n}'}+e_{j,n}+\Lambda_{\star}^{1/2}\parens{\lm{\epsilon}{j+1}-\lm{\epsilon}{j}}}\\
		&\qquad\times\parens{\Delta_nb(\lm{Y}{j})+a(X_{j\Delta_n})\parens{\zeta_{j+1,n}+\zeta_{j+2,n}'}+e_{j,n}+\Lambda_{\star}^{1/2}\parens{\lm{\epsilon}{j+1}-\lm{\epsilon}{j}}}^T\\
		&=\parens{a(X_{j\Delta_n})\parens{\zeta_{j+1,n}+\zeta_{j+2,n}'}}^{\otimes 2}
		+\parens{\Lambda_{\star}^{1/2}\parens{\lm{\epsilon}{j+1}-\lm{\epsilon}{j}}}^{\otimes2}+\parens{\Delta_nb(\lm{Y}{j})+e_{j,n}}^{\otimes2}\\
		&\qquad+a(X_{j\Delta_n})\parens{\zeta_{j+1,n}+\zeta_{j+2,n}'}\parens{\Lambda_{\star}^{1/2}\parens{\lm{\epsilon}{j+1}-\lm{\epsilon}{j}}}^T\\
		&\qquad+\parens{\Lambda_{\star}^{1/2}\parens{\lm{\epsilon}{j+1}-\lm{\epsilon}{j}}}\parens{\zeta_{j+1,n}+\zeta_{j+2,n}'}^Ta(X_{j\Delta_n})^T\\
		&\qquad+a(X_{j\Delta_n})\parens{\zeta_{j+1,n}+\zeta_{j+2,n}'}\parens{\Delta_nb(\lm{Y}{j})+e_{j,n}}^T\\
		&\qquad+\parens{\Delta_nb(\lm{Y}{j})+e_{j,n}}\parens{\zeta_{j+1,n}+\zeta_{j+2,n}'}^Ta(X_{j\Delta_n})^T\\
		&\qquad+\parens{\Delta_nb(\lm{Y}{j})+e_{j,n}}\parens{\Lambda_{\star}^{1/2}\parens{\lm{\epsilon}{j+1}-\lm{\epsilon}{j}}}^T\\
		&\qquad+\parens{\Lambda_{\star}^{1/2}\parens{\lm{\epsilon}{j+1}-\lm{\epsilon}{j}}}\parens{\Delta_nb(\lm{Y}{j})+e_{j,n}}^T.
	\end{align*}
	We define
	\begin{align*}
		q_{j,n}^{(1)}(\vartheta)&:=\ip{A(\lm{Y}{j-1},\vartheta)}{\parens{a(X_{j\Delta_n})\parens{\zeta_{j+1,n}+\zeta_{j+2,n}'}}^{\otimes 2}},\\
		q_{j,n}^{(2)}(\vartheta)&:=\ip{A(\lm{Y}{j-1},\vartheta)}{\parens{\Lambda_{\star}^{1/2}\parens{\lm{\epsilon}{j+1}-\lm{\epsilon}{j}}}^{\otimes2}},\\
		q_{j,n}^{(3)}(\vartheta)&:=\ip{A(\lm{Y}{j-1},\vartheta)}{\parens{\Delta_nb(\lm{Y}{j})+e_{j,n}}^{\otimes2}},\\
		q_{j,n}^{(4)}(\vartheta)&:=\ip{A(\lm{Y}{j-1},\vartheta)}{a(X_{j\Delta_n})\parens{\zeta_{j+1,n}+\zeta_{j+2,n}'}\parens{\Lambda_{\star}^{1/2}\parens{\lm{\epsilon}{j+1}-\lm{\epsilon}{j}}}^T},\\
		&\qquad+\ip{A(\lm{Y}{j-1},\vartheta)}{\parens{\Lambda_{\star}^{1/2}\parens{\lm{\epsilon}{j+1}-\lm{\epsilon}{j}}}\parens{\zeta_{j+1,n}+\zeta_{j+2,n}'}^Ta(X_{j\Delta_n})^T},\\
		q_{j,n}^{(5)}(\vartheta)&:=\ip{A(\lm{Y}{j-1},\vartheta)}{a(X_{j\Delta_n})\parens{\zeta_{j+1,n}+\zeta_{j+2,n}'}\parens{\Delta_nb(\lm{Y}{j})+e_{j,n}}^T}\\
		&\qquad+\ip{A(\lm{Y}{j-1},\vartheta)}{\parens{\Delta_nb(\lm{Y}{j})+e_{j,n}}
			\parens{\zeta_{j+1,n}+\zeta_{j+2,n}'}^Ta(X_{j\Delta_n})^T},\\
		q_{j,n}^{(6)}(\vartheta)&:=\ip{A(\lm{Y}{j-1},\vartheta)}{\parens{\Delta_nb(\lm{Y}{j})+e_{j,n}}\parens{\Lambda_{\star}^{1/2}\parens{\lm{\epsilon}{j+1}-\lm{\epsilon}{j}}}^T},\\
		&\qquad+\ip{A(\lm{Y}{j-1},\vartheta)}{\parens{\Lambda_{\star}^{1/2}\parens{\lm{\epsilon}{j+1}-\lm{\epsilon}{j}}}\parens{\Delta_nb(\lm{Y}{j})+e_{j,n}}^T},
	\end{align*}
	and
	\begin{align*}
		q_{j,n}(\vartheta)&:=\sum_{i=1}^{6}q_{j,n}^{(i)}(\vartheta),
	\end{align*}
	then we obtain
	\begin{align*}
		\bar{Q}_{n}(A(\cdot,\vartheta))=\frac{1}{k_n\Delta_n}\sum_{j=1}^{k_n-2}q_{j,n}(\vartheta).
	\end{align*}
	We examine the following quantities which divide the summation into three parts: for $l=0,1,2$,
	\begin{align*}
		T_{l,n}^{(i)}(\vartheta):=\frac{1}{k_n\Delta_n}\sum_{1\le 3j+l\le k_n-2}q_{3j+l,n}^{(i)}(\vartheta)\text{ for }i=1,\cdots,6.
	\end{align*}
	Firstly we see the pointwise-convergence in probability with respect to $\vartheta$.\\
	
	\noindent\textbf{(Step 1): } We examine $T_{0,n}^{(1)}(\vartheta)$ and consider to show convergence in probability with Lemma 9 in \citep{GeJ93}. Lemma \ref*{lem732} gives
	\begin{align*}
		\CE{q_{3j,n}^{(1)}(\vartheta)}{\mathcal{H}_{3j}^{n}}
		=\parens{\frac{2}{3}+\frac{1}{3p_n^2}}\Delta_n\ip{A(\lm{Y}{3j-1},\vartheta)}{c(X_{3j\Delta_n})}.
	\end{align*}
	Note that Lemma \ref*{lem721} gives
	\begin{align*}
		&\frac{1}{k_n\Delta_n}\sum_{1\le 3j\le k_n-2}\parens{\frac{2}{3}+\frac{1}{3p_n^2}}
		\Delta_n\ip{A(X_{3j\Delta_n},\vartheta)}{c(X_{3j\Delta_n})}
		\cp\frac{1}{3}\times\frac{2}{3}\times\nu_0\parens{\ip{A(\cdot,\vartheta)}{c(\cdot)}},
	\end{align*}
	and we can obtain
	\begin{align*}
		\frac{1}{k_n}\sum_{1\le 3j\le k_n-2}\parens{\frac{2}{3}+\frac{1}{3p_n^2}}
		\ip{\parens{A(X_{3j\Delta_n},\vartheta)-A(X_{(3j-1)\Delta_n},\vartheta)}}{c(X_{3j\Delta_n})}&\cp0\\
		\frac{1}{k_n}\sum_{1\le 3j\le k_n-2}\parens{\frac{2}{3}+\frac{1}{3p_n^2}}
		\ip{\parens{A(\lm{Y}{3j-1},\vartheta)-A(X_{(3j-1)\Delta_n},\vartheta)}}{c(X_{3j\Delta_n})}&\cp0
	\end{align*}
	because of Lemma \ref*{lem739}; hence we have
	\begin{align*}
		\frac{1}{k_n\Delta_n}\sum_{1\le 3j\le k_n-2}\CE{q_{3j,n}^{(1)}(\vartheta)}{\mathcal{H}_{3j}^{n}}\cp \frac{2}{9}\nu_0\parens{\ip{A(\cdot,\vartheta)}{c(\cdot)}}.
	\end{align*}
	Next we have
	\begin{align*}
		\CE{\abs{q_{3j,n}^{(1)}(\vartheta)}^2}{\mathcal{H}_{3j}^{n}}
		\le C\Delta_n^2\norm{A(\lm{Y}{3j-1},\vartheta)}^2\norm{a(X_{3j\Delta_n})}^4
	\end{align*}
	because of Lemma \ref*{lem739}, and then we obtain
	\begin{align*}
		\E{\abs{\frac{1}{k_n^2\Delta_n^2}\sum_{1\le 3j\le k_n-2}\CE{\abs{q_{3j,n}^{(1)}(\vartheta)}^2}{\mathcal{H}_{3j}^{n}}}}
		\to0
	\end{align*}
	also because of Lemma \ref*{lem739}. Therefore, Lemma 9 in \citep{GeJ93} gives
	\begin{align*}
		T_{0,n}^{(1)}(\vartheta)\cp\frac{2}{9}\nu_0\parens{\ip{A(\cdot,\vartheta)}{c(\cdot)}}
	\end{align*}
	and identical convergences for $T_{1,n}^{(1)}(\vartheta)$ and $T_{2,n}^{(1)}(\vartheta)$ can be given. Hence
	\begin{align*}
		T_{0,n}^{(1)}(\vartheta)+T_{1,n}^{(1)}(\vartheta)+T_{2,n}^{(1)}(\vartheta)
		\cp\frac{2}{3}\nu_0\parens{\ip{A(\cdot,\vartheta)}{c(\cdot)}}.
	\end{align*}
	
	\noindent\textbf{(Step 2): } For $T_{0,n}^{(2)}(\vartheta)$, we also see the pointwise convergence in probability with \citep{GeJ93}. Firstly,
	\begin{align*}
		\CE{q_{3j,n}^{(2)}(\vartheta)}{\mathcal{H}_{3j}^n}
		=\frac{2}{p_n}\ip{A(\lm{Y}{j-1},\vartheta)}{\Lambda_{\star}}
	\end{align*}
	and then Proposition \ref*{lem721} leads to
	\begin{align*}
		\frac{1}{k_n\Delta_n}\sum_{1\le 3j\le k_n-2}\CE{q_{3j,n}^{(2)}(\vartheta)}{\mathcal{H}_{3j}^n}
		\cp\begin{cases}
		0 & \text{ if }\tau\in(1,2)\\
		\frac{2}{3}\nu_0\parens{\ip{A(\cdot,\vartheta)}{\Lambda_{\star}}} & \text{ if }\tau=2
		\end{cases}
	\end{align*}
	because
	\begin{align*}
		\frac{1}{p_n\Delta_n}=\frac{1}{p_n^2h_n}=\frac{1}{p_n^{2-\tau}}\to\begin{cases}
			0 & \text{ if }\tau\in(1,2)\\
			1 & \text{ if }\tau=2.
		\end{cases}
	\end{align*}
	Because of Lemma \ref*{lem739}, we also easily have the conditional second moment evaluation such that
	\begin{align*}
		\E{\abs{\frac{1}{k_n^2\Delta_n^2}\sum_{1\le 3j\le k_n-2}\CE{\abs{q_{3j,n}^{(2)}(\vartheta)}^2}{\mathcal{H}_{3j}^n}}}
		\to0,
	\end{align*}
	therefore this $L^1$ convergence verifies convergence in probability and by Lemma 9 in \citep{GeJ93},
	\begin{align*}
		T_{0,n}^{(2)}(\vartheta)\cp\begin{cases}
		0 & \text{ if }\tau\in(1,2)\\
		\frac{2}{3}\nu_0\parens{\ip{A(\cdot,\vartheta)}{\Lambda_{\star}}} & \text{ if }\tau=2
		\end{cases}
	\end{align*}
	and then
	\begin{align*}
		T_{0,n}^{(2)}(\vartheta)+T_{1,n}^{(2)}(\vartheta)+T_{2,n}^{(2)}(\vartheta)\cp\begin{cases}
		0 & \text{ if }\tau\in(1,2)\\
		2\nu_0\parens{\ip{A(\cdot,\vartheta)}{\Lambda_{\star}}} & \text{ if }\tau=2.
		\end{cases}
	\end{align*}
	
	\noindent\textbf{(Step 3):} For $T_{0,n}^{(3)}(\vartheta)$, we can obtain the following $L^1$ convergence:
	\begin{align*}
		\E{\abs{\frac{1}{k_n\Delta_n}\sum_{1\le 3j\le k_n-2}q_{3j,n}^{(3)}(\vartheta)}}
		\to0.
	\end{align*}
	Hence it leads to
	\begin{align*}
		T_{0,n}^{(3)}(\vartheta)\cp0
	\end{align*}
	and
	\begin{align*}
		T_{0,n}^{(3)}(\vartheta)+T_{1,n}^{(3)}(\vartheta)+T_{2,n}^{(3)}(\vartheta)\cp 0.
	\end{align*}
	
	\noindent\textbf{(Step 4): } We show pointwise convergence in probability of $T_{0,n}^{(4)}(\vartheta)$ with Lemma 9 in \citep{GeJ93}. First of all,
	\begin{align*}
		\CE{q_{3j,n}^{(4)}(\vartheta)}{\mathcal{H}_{3j}^n}
		=0
	\end{align*}
	since $w$ and $\epsilon$ are independent. Secondly, for conditional second moment, we have
	\begin{align*}
		\CE{\abs{q_{3j,n}^{(4)}(\vartheta)}^2}{\mathcal{H}_{3j}^n}\le C\Delta_n^2\norm{A(\lm{Y}{j-1},\vartheta)}^2\norm{a(X_{j\Delta_n})}^2,
	\end{align*}
	because of Lemma \ref*{lem739}; therefore, 
	\begin{align*}
		\E{\abs{\frac{1}{k_n^2\Delta_n^2}\sum_{1\le 3j\le k_n-2}\CE{\abs{q_{3j,n}^{(4)}(\vartheta)}^2}{\mathcal{H}_{3j}^n}}}\to0
	\end{align*}
	by Lemma \ref*{lem739}. Then Lemma 9 in \citep{GeJ93} verifies
	\begin{align*}
		T_{0,n}^{(4)}\cp0
	\end{align*}
	and
	\begin{align*}
		T_{0,n}^{(4)}+T_{1,n}^{(4)}+T_{2,n}^{(4)}\cp0.
	\end{align*}
	
	\noindent\textbf{(Step 5): } We can evaluate $L^1$ convergence of $T_{0,n}^{(5)}(\vartheta)$ and $T_{0,n}^{(6)}(\vartheta)$ with Lemma \ref*{lem739} such that
	\begin{align*}
		\E{\abs{\frac{1}{k_n\Delta_n}\sum_{1\le 3j\le k_n-2}q_{3j,n}^{(5)}}}
		&\to0,\\
		\E{\abs{\frac{1}{k_n\Delta_n}\sum_{1\le 3j\le k_n-2}q_{3j,n}^{(6)}(\vartheta)}}
		&\to0.
	\end{align*}
	This leads to
	\begin{align*}
		T_{0,n}^{(5)}(\vartheta)&\cp0\\
		T_{0,n}^{(6)}(\vartheta)&\cp0
	\end{align*}
	and
	\begin{align*}
		T_{0,n}^{(5)}(\vartheta)+T_{1,n}^{(5)}(\vartheta)+T_{2,n}^{(5)}(\vartheta)&\cp0\\
		T_{0,n}^{(6)}(\vartheta)+T_{1,n}^{(6)}(\vartheta)+T_{2,n}^{(6)}(\vartheta)&\cp0.
	\end{align*}Here we have pointwise convergence in probability of $\bar{Q}_n(A(\cdot,\vartheta))$ for all $\vartheta$.\\
	
	\noindent\textbf{(Step 6): }Finally we see the uniform convergence. It can be obtained as
	\begin{align*}
		\sup_{n\in\mathbf{N}}\E{\sup_{\vartheta\in\Xi}\abs{\nabla\bar{Q}_n(A(\cdot,\vartheta))}}\le C
	\end{align*}
	whose computation is verified by Lemma \ref*{lem739}. Therefore uniform convergence in probability is obtained.
\end{proof}

\subsection{Asymptotic normality}

\begin{theorem}\label{thm751}
	Under (A1)-(A5), (AH) and $k_n\Delta_n^2\to0$,
	\begin{align*}
		\crotchet{\begin{matrix}
			\sqrt{n}D_n\\
			\sqrt{k_n}\crotchet{\bar{Q}_n\parens{A_\kappa(\cdot)}
				-\frac{2}{3}\bar{M}_n\parens{\ip{A_\kappa(\cdot)}{c_n^{\tau}\parens{\cdot,\alpha^{\star},\Lambda_{\star}}}}}_{\kappa}\\
			\sqrt{k_n\Delta_n}\crotchet{\bar{D}_n\parens{f_\lambda(\cdot)}}_\lambda
			\end{matrix}}\cl N(\mathbf{0},W^{(\tau)}(\tuborg{A_{\kappa}},\tuborg{f_{\lambda}})),
	\end{align*}
	where
	\begin{align*}
		c_n^{\tau}(x,\alpha,\Lambda)&:=c(x,\alpha)+3\Delta_n^{\frac{2-\tau}{\tau-1}}\Lambda.
	\end{align*}
\end{theorem}

\begin{proof}\textbf{(Step 1): } As the proof for Lemma 3.1.1, we can decompose
	\begin{align*}
		\hat{\Lambda}_n-\Lambda_{\star}
		&=\frac{1}{2n}\sum_{i=0}^{n-1}\parens{Y_{(i+1)h_n}-Y_{ih_n}}^{\otimes 2}-\Lambda_{\star}\\
		&=\frac{1}{2n}\sum_{i=0}^{n-1}\parens{X_{(i+1)h_n}+\Lambda_{\star}^{1/2}\epsilon_{(i+1)h_n}
			-X_{ih_n}-\Lambda_{\star}^{1/2}\epsilon_{ih_n}}^{\otimes 2}
		-\Lambda_{\star}\\
		&=\frac{1}{2n}\sum_{i=0}^{n-1}\parens{X_{(i+1)h_n}-X_{ih_n}}^{\otimes 2}
		+\frac{1}{2n}\sum_{i=0}^{n-1}\Lambda_{\star}^{1/2}\parens{\parens{\epsilon_{(i+1)h_n}-\epsilon_{ih_n}}^{\otimes 2}-2I_d}\Lambda_{\star}^{1/2}\\
		&\qquad+\frac{1}{2n}\sum_{i=0}^{n-1}\parens{X_{(i+1)h_n}-X_{ih_n}}\parens{\epsilon_{(i+1)h_n}-\epsilon_{ih_n}}^T\Lambda_{\star}^{1/2}\\
		&\qquad+\frac{1}{2n}\sum_{i=0}^{n-1}\Lambda_{\star}^{1/2}\parens{\epsilon_{(i+1)h_n}-\epsilon_{ih_n}}\parens{X_{(i+1)h_n}-X_{ih_n}}^T.
	\end{align*}
	Hence
	\begin{align*}
		\sqrt{n}\parens{\hat{\Lambda}_n-\Lambda_{\star}}&=\frac{1}{2\sqrt{n}}\sum_{i=0}^{n-1}\parens{X_{(i+1)h_n}-X_{ih_n}}^{\otimes 2}
		\\
		&\qquad+\frac{1}{2\sqrt{n}}\sum_{i=0}^{n-1}\Lambda_{\star}^{1/2}\parens{\parens{\epsilon_{(i+1)h_n}-\epsilon_{ih_n}}^{\otimes 2}-I_d}\Lambda_{\star}^{1/2}\\
		&\qquad+\frac{1}{2\sqrt{n}}\sum_{i=0}^{n-1}\parens{X_{(i+1)h_n}-X_{ih_n}}\parens{\epsilon_{(i+1)h_n}-\epsilon_{ih_n}}^T\Lambda_{\star}^{1/2}\\
		&\qquad+\frac{1}{2\sqrt{n}}\sum_{i=0}^{n-1}\Lambda_{\star}^{1/2}\parens{\epsilon_{(i+1)h_n}-\epsilon_{ih_n}}\parens{X_{(i+1)h_n}-X_{ih_n}}^T.
	\end{align*}
	The first term of right hand side shows the following moment convergence:
	\begin{align*}
		\E{\norm{\frac{1}{2\sqrt{n}}\sum_{i=0}^{n-1}\parens{X_{(i+1)h_n}-X_{ih_n}}^{\otimes 2}}}
		\to0.
	\end{align*}
	The conditional moment evaluation for the third term can be given as follows: for all $j_1$ and $j_2$ in $\tuborg{1,\cdots,d}$,
	\begin{align*}
		\E{\abs{\frac{1}{2\sqrt{n}}\sum_{i=0}^{n-1}\CE{\parens{X_{(i+1)h_n}-X_{ih_n}}^{j_1}\parens{\epsilon_{(i+1)h_n}-\epsilon_{ih_n}}^{j_2}}{\mathcal{H}_{ih_n}^n}}}\to 0
	\end{align*}
	and
	\begin{align*}
		\E{\abs{\frac{1}{4n}\sum_{i=0}^{n-1}\CE{\abs{\parens{X_{(i+1)h_n}-X_{ih_n}}^{j_1}\parens{\epsilon_{(i+1)h_n}-\epsilon_{ih_n}}^{j_2}}^2}{\mathcal{H}_{ih_n}^n}}}\to 0.
	\end{align*}
	Therefore Lemma 9 in \citep{GeJ93} leads to the third and the fourth term are $o_P(1)$. Then we obtain
	\begin{align*}
		\sqrt{n}\parens{\hat{\Lambda}_n-\Lambda_{\star}}&=\frac{1}{2\sqrt{n}}\sum_{i=0}^{n-1}\Lambda_{\star}^{1/2}\parens{\parens{\epsilon_{(i+1)h_n}-\epsilon_{ih_n}}^{\otimes 2}-2I_d}\Lambda_{\star}^{1/2}+o_P(1)
	\end{align*}
	and 
	\begin{align*}
		\sqrt{n}D_n&=\frac{1}{2\sqrt{n}}\sum_{i=0}^{n-1}\vech\parens{\Lambda_{\star}^{1/2}\parens{\parens{\epsilon_{(i+1)h_n}-\epsilon_{ih_n}}^{\otimes 2}-2I_d}\Lambda_{\star}^{1/2}}+o_P(1).
	\end{align*}
	We can rewrite the summation as
	\begin{align*}
		&\frac{1}{2\sqrt{n}}\sum_{i=0}^{n-1}\Lambda_{\star}^{1/2}\parens{\parens{\epsilon_{(i+1)h_n}-\epsilon_{ih_n}}^{\otimes 2}-2I_d}\Lambda_{\star}^{1/2}\\
		&=\frac{1}{2\sqrt{n}}\sum_{i=0}^{n-1}\Lambda_{\star}^{1/2}\parens{\parens{\epsilon_{(i+1)h_n}}^{\otimes 2}+\parens{\epsilon_{ih_n}}^{\otimes 2}+\parens{\epsilon_{(i+1)h_n}}\parens{\epsilon_{ih_n}}^T+\parens{\epsilon_{ih_n}}\parens{\epsilon_{(i+1)h_n}}^T-2I_d}\Lambda_{\star}^{1/2}\\
		&=\frac{1}{2\sqrt{n}}\sum_{i=1}^{n-1}\Lambda_{\star}^{1/2}\parens{2\parens{\epsilon_{ih_n}}^{\otimes 2}+\parens{\epsilon_{ih_n}}\parens{\epsilon_{(i-1)h_n}}^T+\parens{\epsilon_{(i-1)h_n}}\parens{\epsilon_{ih_n}}^T-2I_d}\Lambda_{\star}^{1/2}+o_P(1).
	\end{align*}
	Since
	\begin{align*}
		\E{\norm{\frac{1}{\sqrt{n}}\sum_{i=1}^{p_n}\CE{2\parens{\epsilon_{ih_n}}^{\otimes 2}+\parens{\epsilon_{ih_n}}\parens{\epsilon_{(i-1)h_n}}^T+\parens{\epsilon_{(i-1)h_n}}\parens{\epsilon_{ih_n}}^T-2I_d}{\mathcal{H}_{(i-1)h_n}}}}
		&=0
	\end{align*}
	and
	\begin{align*}
		\E{\frac{1}{n}\sum_{i=1}^{p_n}\CE{\norm{2\parens{\epsilon_{ih_n}}^{\otimes 2}+\parens{\epsilon_{ih_n}}\parens{\epsilon_{(i-1)h_n}}^T+\parens{\epsilon_{(i-1)h_n}}\parens{\epsilon_{ih_n}}^T-2I_d}^2}{\mathcal{H}_{(i-1)h_n}}}\to 0,
	\end{align*}
	we have
	\begin{align*}
		&\frac{1}{2\sqrt{n}}\sum_{i=0}^{n-1}\Lambda_{\star}^{1/2}\parens{\parens{\epsilon_{(i+1)h_n}-\epsilon_{ih_n}}^{\otimes 2}-2I_d}\Lambda_{\star}^{1/2}\\
		&=\frac{1}{2\sqrt{n}}\sum_{i=p_n}^{n-p_n-1}\Lambda_{\star}^{1/2}\parens{2\parens{\epsilon_{ih_n}}^{\otimes 2}+\parens{\epsilon_{ih_n}}\parens{\epsilon_{(i-1)h_n}}^T+\parens{\epsilon_{(i-1)h_n}}\parens{\epsilon_{ih_n}}^T-2I_d}\Lambda_{\star}^{1/2}+o_P(1)
	\end{align*}
	and
	\begin{align*}
		\sqrt{n}D_n&=\frac{\sqrt{n}}{k_n}\sum_{j=1}^{k_n-2}D_{j,n}'+o_P(1),
	\end{align*}
	where
	\begin{align*}
		D_{j,n}'&=\frac{1}{2p_n}\sum_{i=0}^{p_n-1}\vech\lparens{\Lambda_{\star}^{1/2}\lparens{2\parens{\epsilon_{j\Delta_n+ih_n}}^{\otimes 2}+\parens{\epsilon_{j\Delta_n+ih_n}}\parens{\epsilon_{j\Delta_n+(i-1)h_n}}^T}}\\
			&\qquad\qquad\qquad\qquad\qquad\rparens{\rparens{+\parens{\epsilon_{j\Delta_n+(i-1)h_n}}\parens{\epsilon_{j\Delta_n+ih_n}}^T-2I_d}\Lambda_{\star}^{1/2}},\\
		\parens{D_{j,n}'}^{l_1,l_2}&=\frac{1}{2p_n}\sum_{i=0}^{p_n-1}\lparens{\Lambda_{\star}^{1/2}\lparens{2\parens{\epsilon_{j\Delta_n+ih_n}}^{\otimes 2}+\parens{\epsilon_{j\Delta_n+ih_n}}\parens{\epsilon_{j\Delta_n+(i-1)h_n}}^T}}\\
		&\qquad\qquad\qquad\qquad\qquad\rparens{\rparens{+\parens{\epsilon_{j\Delta_n+(i-1)h_n}}\parens{\epsilon_{j\Delta_n+ih_n}}^T-2I_d}\Lambda_{\star}^{1/2}}^{l_1,l_2}.
	\end{align*}
	The conditional moment of $u_{i}$ is given as
	\begin{align*}
	\CE{D_{j,n}'}{\mathcal{H}_{j}^n}=\mathbf{0}.
	\end{align*}
	Note that
	\begin{align*}
		&\parens{\Lambda_{\star}^{1/2}\parens{2\parens{\epsilon_{ih_n}}^{\otimes 2}+\parens{\epsilon_{ih_n}}\parens{\epsilon_{(i-1)h_n}}^T+\parens{\epsilon_{(i-1)h_n}}\parens{\epsilon_{ih_n}}^T-2I_d}\Lambda_{\star}^{1/2}}^{l_1,l_2}\\
		&=2\parens{\sum_{k_1=1}^{d}\parens{\Lambda_{\star}^{1/2}}^{l_1,k_1}\parens{\epsilon_{ih_n}^{k_1}}}
		\parens{\sum_{k_2=1}^{d}\parens{\epsilon_{ih_n}^{k_2}}\parens{\Lambda_{\star}^{1/2}}^{k_2,l_2}}\\
		&\qquad+\parens{\sum_{k_1=1}^{d}\parens{\Lambda_{\star}^{1/2}}^{l_1,k_1}\parens{\epsilon_{ih_n}^{k_1}}}
		\parens{\sum_{k_2=1}^{d}\parens{\epsilon_{(i-1)h_n}^{k_2}}\parens{\Lambda_{\star}^{1/2}}^{k_2,l_2}}\\
		&\qquad+\parens{\sum_{k_1=1}^{d}\parens{\Lambda_{\star}^{1/2}}^{l_1,k_1}\parens{\epsilon_{(i-1)h_n}^{k_1}}}
		\parens{\sum_{k_2=1}^{d}\parens{\epsilon_{ih_n}^{k_2}}\parens{\Lambda_{\star}^{1/2}}^{k_2,l_2}}\\
		&\qquad-2\Lambda_{\star}^{l_1,l_2}
	\end{align*}
	and hence
	\begin{align*}
		&\tilde{D}_{ih_n,n}\parens{(l_1,l_2),(l_3,l_4)}\\
		&:=\mathbf{E}\lcrotchet{\parens{\Lambda_{\star}^{1/2}\parens{2\parens{\epsilon_{ih_n}}^{\otimes 2}+\parens{\epsilon_{ih_n}}\parens{\epsilon_{(i-1)h_n}}^T+\parens{\epsilon_{(i-1)h_n}}\parens{\epsilon_{ih_n}}^T-2I_d}\Lambda_{\star}^{1/2}}^{l_1,l_2}}\\
		&\qquad\times\rcrotchet{\left.\parens{\Lambda_{\star}^{1/2}\parens{2\parens{\epsilon_{ih_n}}^{\otimes 2}+\parens{\epsilon_{ih_n}}\parens{\epsilon_{(i-1)h_n}}^T+\parens{\epsilon_{(i-1)h_n}}\parens{\epsilon_{ih_n}}^T-2I_d}\Lambda_{\star}^{1/2}}^{l_3,l_4}\right|\mathcal{H}_{(i-1)h_n}^n}\\
		&=4\sum_{k=1}^{d}\parens{\Lambda_{\star}^{1/2}}^{l_1,k}\parens{\Lambda_{\star}^{1/2}}^{l_2,k}\parens{\Lambda_{\star}^{1/2}}^{l_3,k}\parens{\Lambda_{\star}^{1/2}}^{l_4,k}
		\parens{\E{\abs{\epsilon_0^k}^4}-3}\\
		&\quad+4\parens{\Lambda_{\star}^{l_1,l_3}\Lambda_{\star}^{l_2,l_4}+\Lambda_{\star}^{l_1,l_4}\Lambda_{\star}^{l_2,l_3}}\\
		&\quad+\Lambda_{\star}^{l_1,l_3}\parens{\sum_{k_2=1}^{d}\parens{\epsilon_{(i-1)h_n}^{k_2}}\parens{\Lambda_{\star}^{1/2}}^{k_2,l_2}}\parens{\sum_{k_2=1}^{d}\parens{\epsilon_{(i-1)h_n}^{k_2}}\parens{\Lambda_{\star}^{1/2}}^{k_2,l_4}}\\
		&\quad+\Lambda_{\star}^{l_1,l_4}\parens{\sum_{k_2=1}^{d}\parens{\epsilon_{(i-1)h_n}^{k_2}}\parens{\Lambda_{\star}^{1/2}}^{k_2,l_2}}
		\parens{\sum_{k_1=1}^{d}\parens{\Lambda_{\star}^{1/2}}^{l_3,k_1}\parens{\epsilon_{(i-1)h_n}^{k_1}}}\\
		&\quad+\Lambda_{\star}^{l_2,l_3}\parens{\sum_{k_1=1}^{d}\parens{\Lambda_{\star}^{1/2}}^{l_1,k_1}\parens{\epsilon_{(i-1)h_n}^{k_1}}}\parens{\sum_{k_2=1}^{d}\parens{\epsilon_{(i-1)h_n}^{k_2}}\parens{\Lambda_{\star}^{1/2}}^{k_2,l_4}}\\
		&\quad+\Lambda_{\star}^{l_2,l_4}\parens{\sum_{k_1=1}^{d}\parens{\Lambda_{\star}^{1/2}}^{l_1,k_1}\parens{\epsilon_{(i-1)h_n}^{k_1}}}\parens{\sum_{k_1=1}^{d}\parens{\Lambda_{\star}^{1/2}}^{l_3,k_1}\parens{\epsilon_{(i-1)h_n}^{k_1}}}
	\end{align*}
	and
	\begin{align*}
		\CE{\parens{\sum_{k_2=1}^{d}\parens{\epsilon_{(i-1)h_n}^{k_2}}\parens{\Lambda_{\star}^{1/2}}^{k_2,l_2}}\parens{\sum_{k_2=1}^{d}\parens{\epsilon_{(i-1)h_n}^{k_2}}\parens{\Lambda_{\star}^{1/2}}^{k_2,l_4}}}{\mathcal{H}_{(i-2)h_n}^n}
		= \Lambda_{\star}^{l_2,l_4}.
	\end{align*}
	These lead to
	\begin{align*}
		\frac{n}{k_n^2}\sum_{j=1}^{k_n-2}\CE{\parens{D_{j,n}'}^{l_1,l_2}\parens{D_{j,n}'}^{l_3,l_4}}{\mathcal{H}_j^n}
		&=\frac{n}{4n^2}\sum_{j=1}^{k_n-2}\CE{\sum_{i=0}^{p_n-1}\tilde{D}_{j\Delta_n+ih_n,n}\parens{(l_1,l_2),(l_3,l_4)}}{\mathcal{H}_j^n}\\
		&\cp W_1^{(l_1,l_2),(l_3,l_4)}.
	\end{align*}
	Then
	\begin{align*}
		\frac{n}{k_n^2}\sum_{j=1}^{k_n-2}\CE{\parens{D_{j,n}'}^{\otimes2}}{\mathcal{H}_{j}^n}\cp W_1.
	\end{align*}
	Finally we check
	\begin{align*}
		\E{\abs{\frac{n^2}{k_n^4}\sum_{j=1}^{k_n-2}\CE{\norm{D_{j,n}'}^4}{\mathcal{H}_{j}^n}}}\to 0.
	\end{align*}
	Note that $\tuborg{\epsilon_{ih_n}}$ are i.i.d. and when we denote
	\begin{align*}
		M_i=2\parens{\epsilon_{ih_n}}^{\otimes 2}+\parens{\epsilon_{ih_n}}\parens{\epsilon_{(i-1)h_n}}^T+\parens{\epsilon_{(i-1)h_n}}\parens{\epsilon_{ih_n}}^T-2I_d,
	\end{align*}
	then
	\begin{align*}
		\E{\norm{\sum_{i=1}^{p_n}M_i}^4}&=\E{\sum_{i_1}\sum_{i_2}\parens{\tr\parens{M_{i_1}M_{i_2}}}^2+\sum_{i_1}\sum_{i_2}\sum_{i_3}\sum_{i_4}\tr\parens{M_{i_1}M_{i_2}M_{i_3}M_{i_4}}}\\
		&\le Cp_n^2;
	\end{align*}
	They verify the result.\\
	
	\noindent\textbf{(Step 2): } Corollary \ref*{cor738} gives
	\begin{align*}
		\parens{\lm{Y}{j+1}-\lm{Y}{j}}^{\otimes 2}
		&=\parens{\Delta_nb(X_{j\Delta_n})+e_{j,n}}^{\otimes 2}
		+\parens{a(X_{j\Delta_n})\parens{\zeta_{j+1,n}+\zeta_{j+2,n}'}}^{\otimes 2}\\
		&\qquad+\parens{\Lambda_{\star}^{1/2}\parens{\lm{\epsilon}{j+1}-\lm{\epsilon}{j}}}^{\otimes2}\\
		&\qquad+\parens{\Delta_nb(X_{j\Delta_n})+e_{j,n}}\parens{a(X_{j\Delta_n})\parens{\zeta_{j+1,n}+\zeta_{j+2,n}'}}^T\\
		&\qquad+\parens{a(X_{j\Delta_n})\parens{\zeta_{j+1,n}+\zeta_{j+2,n}'}}\parens{\Delta_nb(X_{j\Delta_n})+e_{j,n}}^T\\
		&\qquad+\parens{\Delta_nb(X_{j\Delta_n})+e_{j,n}}
		\parens{\Lambda_{\star}^{1/2}\parens{\lm{\epsilon}{j+1}-\lm{\epsilon}{j}}}^T\\
		&\qquad+\parens{\Lambda_{\star}^{1/2}\parens{\lm{\epsilon}{j+1}-\lm{\epsilon}{j}}}
		\parens{\Delta_nb(X_{j\Delta_n})+e_{j,n}}^T\\
		&\qquad+\parens{a(X_{j\Delta_n})\parens{\zeta_{j+1,n}+\zeta_{j+2,n}'}}
		\parens{\Lambda_{\star}^{1/2}\parens{\lm{\epsilon}{j+1}-\lm{\epsilon}{j}}}^T\\
		&\qquad+\parens{\Lambda_{\star}^{1/2}\parens{\lm{\epsilon}{j+1}-\lm{\epsilon}{j}}}
		\parens{a(X_{j\Delta_n})\parens{\zeta_{j+1,n}+\zeta_{j+2,n}'}}^T.
	\end{align*}
	We define the random variable $U_n(\lambda)$ such that
	\begin{align*}
		U_n(\kappa)&:=\sqrt{k_n}\parens{\bar{Q}_n(A_{\kappa}(\cdot))-\frac{2}{3}\bar{M}_n\parens{\ip{A_\kappa(\cdot)}{
			c_n^{\tau}(\cdot,\alpha^{\star},\Lambda_{\star})}}}\\
		&=\frac{1}{\sqrt{k_n}\Delta_n}\sum_{j=1}^{k_n-2}\ip{A_{\kappa}(\lm{Y}{j-1})}{\parens{\lm{Y}{j+1}-\lm{Y}{j}}^{\otimes 2}}
		-\frac{2}{3\sqrt{k_n}}\sum_{j=1}^{k_n-2}\ip{A_{\kappa}(\lm{Y}{j-1})}{\parens{c_n^{\tau}(\lm{Y}{j-1},\alpha^{\star},\Lambda_{\star})}}
\\
		&=\frac{1}{\sqrt{k_n}\Delta_n}\sum_{j=1}^{k_n-2}
		\ip{A_{\kappa}(\lm{Y}{j-1})}{\parens{\Delta_nb(X_{j\Delta_n})+e_{j,n}}^{\otimes 2}}\\
		&\qquad+\frac{1}{\sqrt{k_n}\Delta_n}\sum_{j=1}^{k_n-2}
		\ip{A_{\kappa}(\lm{Y}{j-1})}{\parens{a(X_{j\Delta_n})\parens{\zeta_{j+1,n}+\zeta_{j+2,n}'}}^{\otimes 2}}\\
		&\qquad+\frac{1}{\sqrt{k_n}\Delta_n}\sum_{j=1}^{k_n-2}
		\ip{A_{\kappa}(\lm{Y}{j-1})}{\parens{\Lambda_{\star}^{1/2}\parens{\lm{\epsilon}{j+1}-\lm{\epsilon}{j}}}^{\otimes2}}\\
		&\qquad+\frac{2}{\sqrt{k_n}\Delta_n}\sum_{j=1}^{k_n-2}
		\ip{\bar{A}_{\kappa}(\lm{Y}{j-1})}{
			\parens{\Delta_nb(X_{j\Delta_n})+e_{j,n}}\parens{a(X_{j\Delta_n})\parens{\zeta_{j+1,n}+\zeta_{j+2,n}'}}^T}\\
		&\qquad+\frac{2}{\sqrt{k_n}\Delta_n}\sum_{j=1}^{k_n-2}
		\ip{\bar{A}_{\kappa}(\lm{Y}{j-1})}{
			\parens{\Delta_nb(X_{j\Delta_n})+e_{j,n}}
			\parens{\Lambda_{\star}^{1/2}\parens{\lm{\epsilon}{j+1}-\lm{\epsilon}{j}}}^T}\\
		&\qquad+\frac{2}{\sqrt{k_n}\Delta_n}\sum_{j=1}^{k_n-2}
		\ip{\bar{A}_{\kappa}(\lm{Y}{j-1})}{
			\parens{a(X_{j\Delta_n})\parens{\zeta_{j+1,n}+\zeta_{j+2,n}'}}
			\parens{\Lambda_{\star}^{1/2}\parens{\lm{\epsilon}{j+1}-\lm{\epsilon}{j}}}^T}\\
		&\qquad-\frac{2}{3\sqrt{k_n}}\sum_{j=1}^{k_n-2}\ip{A_\kappa(\lm{Y}{j-1})}
			{c(\lm{Y}{j-1},\alpha^{\star})+3\Delta_n^{\frac{2-\tau}{\tau-1}}\Lambda_{\star}},
	\end{align*}
	where $\bar{A}_{\kappa}:=\frac{1}{2}\parens{A_{\kappa}+A_{\kappa}^T}$. Related to this decomposition, let us $u_{j,n}^{(l)}(\kappa),\ l=1,\cdots,7$ such that
	\begin{align*}
		u_{j,n}^{(1)}(\kappa)
		&:=\ip{A_{\kappa}(\lm{Y}{j-1})}{a(X_{j\Delta_n})\parens{\frac{1}{\Delta_n}\parens{\zeta_{j+1,n}+\zeta_{j+2,n}'}^{\otimes 2}
				-\frac{2}{3}I_r}a(X_{j\Delta_n})^{T}},\\
		u_{j,n}^{(2)}(\kappa)
		&:=\frac{2}{3}\ip{A_{\kappa}(\lm{Y}{j-1})}{c(X_{j\Delta_n})-c(\lm{Y}{j-1})},\\
		u_{j,n}^{(3)}(\kappa)
		&:=\ip{A_{\kappa}(\lm{Y}{j-1})}{\Lambda_{\star}^{1/2}
			\parens{\frac{1}{\Delta_n}\parens{\lm{\epsilon}{j+1}-\lm{\epsilon}{j}}^{\otimes 2}-2\Delta_n^{\frac{2-\tau}{\tau-1}}I_d}\Lambda_{\star}^{1/2}}\\
		u_{j,n}^{(4)}(\kappa)
		&:=\frac{2}{\Delta_n}\ip{\bar{A}_{\kappa}(\lm{Y}{j-1})}{
			\parens{a(X_{j\Delta_n})\parens{\zeta_{j+1,n}+\zeta_{j+2,n}'}}
			\parens{\Lambda_{\star}^{1/2}\parens{\lm{\epsilon}{j+1}-\lm{\epsilon}{j}}}^T},\\
		u_{j,n}^{(5)}(\kappa)
		&:=\frac{1}{\Delta_n}\ip{A_{\kappa}(\lm{Y}{j-1})}{\parens{\Delta_nb(X_{j\Delta_n})+e_{j,n}}^{\otimes 2}},\\
		u_{j,n}^{(6)}(\kappa)
		&:=\frac{2}{\Delta_n}\ip{\bar{A}_{\kappa}(\lm{Y}{j-1})}{
			\parens{\Delta_nb(X_{j\Delta_n})+e_{j,n}}\parens{a(X_{j\Delta_n})\parens{\zeta_{j+1,n}+\zeta_{j+2,n}'}}^T},\\
		u_{j,n}^{(7)}(\kappa)
		&:=\frac{2}{\Delta_n}
		\ip{\bar{A}_{\kappa}(\lm{Y}{j-1})}{
			\parens{\Delta_nb(X_{j\Delta_n})+e_{j,n}}
			\parens{\Lambda_{\star}^{1/2}\parens{\lm{\epsilon}{j+1}-\lm{\epsilon}{j}}}^T}.
	\end{align*}
	Then we obtain
	\begin{align*}
		U_n(\kappa)=\sum_{l=1}^{7}U_n^{(l)}(\kappa),\ 
		U_n^{(l)}(\kappa)=\frac{1}{\sqrt{k_n}}\sum_{j=1}^{k_n-2}u_{j,n}^{(l)}(\kappa).
	\end{align*}
	With respect to $U_n^{(1)}(\kappa)$, we have
	\begin{align*}
		U_n^{(1)}(\kappa)=\frac{1}{\sqrt{k_n}}\sum_{j=1}^{k_n-2}u_{j,n}^{(1)}(\kappa)
		=\frac{1}{\sqrt{k_n}}\sum_{j=2}^{k_n-2}s_{j,n}^{(1)}(\kappa)+
		\frac{1}{\sqrt{k_n}}\sum_{j=2}^{k_n-2}\tilde{s}_{j,n}^{(1)}(\kappa)+o_P(1),
	\end{align*}
	where
	\begin{align*}
		s_{j,n}^{(1)}(\kappa)&=\ip{A_{\kappa}(\lm{Y}{j-1})}{
			a(X_{j\Delta_n})\parens{\frac{1}{\Delta_n}\parens{\zeta_{j+1,n}}^{\otimes 2}-m_nI_r}a(X_{j\Delta_n})^{T}}\\
		&\qquad+\ip{A_{\kappa}(\lm{Y}{j-2})}{
			a(X_{(j-1)\Delta_n})\parens{\frac{1}{\Delta_n}\parens{\zeta_{j+1,n}'}^{\otimes 2}-m_n'I_r}a(X_{(j-1)\Delta_n})^{T}}\\
		&\qquad+2\ip{\bar{A}_{\kappa}(\lm{Y}{j-2})}{a(X_{(j-1)\Delta_n})
			\parens{\frac{1}{\Delta_n}\parens{\zeta_{j,n}\parens{\zeta_{j+1,n}'}^T}}a(X_{(j-1)\Delta_n})^{T}},\\
		\tilde{s}_{j,n}^{(1)}(\kappa)&=\parens{\frac{1}{2p_n}+\frac{1}{6p_n^2}}
		\ip{A_{\kappa}(\lm{Y}{j-1})}{c(X_{j\Delta_n})}\\
		&\qquad+\parens{-\frac{1}{2p_n}+\frac{1}{6p_n^2}}\ip{A_{\kappa}(\lm{Y}{j-2})}{c(X_{(j-1)\Delta_n})},\\
		m_n&=\frac{1}{3}+\frac{1}{2p_n}+\frac{1}{6p_n^2},\\
		m_n'&=\frac{1}{3}-\frac{1}{2p_n}+\frac{1}{6p_n^2}.
	\end{align*}
	Note the following $L^1$ convergence
	\begin{align*}
		\E{\abs{\frac{1}{\sqrt{k_n}}\sum_{j=2}^{k_n-2}\tilde{s}_{j,n}^{(1)}(\kappa)}}
		\to0.
	\end{align*}
	Hence
	\begin{align*}
		U_n^{(1)}(\kappa)=\frac{1}{\sqrt{k_n}}\sum_{j=2}^{k_n-2}s_{j,n}^{(1)}(\kappa)+o_P(1)
	\end{align*}
	and it is enough to examine the first term of the right hand side.
	Firstly, Lemma \ref*{lem732} leads to
	\begin{align*}
		\CE{s_{j,n}^{(1)}(\kappa)}{\mathcal{H}_j^n}=0.
	\end{align*}
	Note the fact that for $\Re^r$-valued random vectors $\mathbf{x}$ and $\mathbf{y}$ such that
	\begin{align*}
		\crotchet{\begin{matrix}
			\mathbf{x}\\
			\mathbf{y}
			\end{matrix}}\sim N\parens{\mathbf{0}, \crotchet{\begin{matrix}
				\sigma_{11}I_r & \sigma_{12}I_r\\
				\sigma_{12}I_r & \sigma_{22}I_r
				\end{matrix}}},
	\end{align*}
	where $\sigma_{11}>0$, $\sigma_{22}>0$ and $\abs{\sigma_{12}}^2\le \sigma_{11}\sigma_{22}$, it holds for any $\Re^r\times\Re^r$-valued matrix $M$,
	\begin{align*}
		\E{\mathbf{y}\mathbf{x}^TM\mathbf{x}\mathbf{y}^T}=\sigma_{12}^2\parens{M+M^T}+\sigma_{11}\sigma_{22}\tr\parens{M}I_r
	\end{align*}
	and also the fact that for any square matrices $A$ and $B$ whose dimensions coincide,
	\begin{align*}
		\tr\parens{AB}+\tr\parens{AB^T}=2\tr\parens{\bar{A}\bar{B}},
	\end{align*}
	where $\bar{A}=\parens{A+A^T}/2$ and $\bar{B}=\parens{B+B^T}/2$.
	For all $\kappa_1,\ \kappa_2$
	\begin{align*}
		&\mathbf{E}\left[\ip{A_{\kappa_1}(\lm{Y}{j-1})}{
				a(X_{j\Delta_n})\parens{\frac{1}{\Delta_n}\parens{\zeta_{j+1,n}}^{\otimes 2}-m_nI_r}a(X_{j\Delta_n})^{T}}\right.\\
		&\qquad\left.\left.\times\ip{A_{\kappa_2}(\lm{Y}{j-1})}{
				a(X_{j\Delta_n})\parens{\frac{1}{\Delta_n}\parens{\zeta_{j+1,n}}^{\otimes 2}-m_nI_r}a(X_{j\Delta_n})^{T}}\right|\mathcal{H}_j^n\right]\\
		&=m_n^2\tr\parens{A_{\kappa_1}(\lm{Y}{j-1})c(X_{j\Delta_n})A_{\kappa_2}(\lm{Y}{j-1})c(X_{j\Delta_n})}\\
		&\qquad+m_n^2\tr\parens{A_{\kappa_1}(\lm{Y}{j-1})^Tc(X_{j\Delta_n})A_{\kappa_2}(\lm{Y}{j-1})c(X_{j\Delta_n})}
	\end{align*}
	and
	\begin{align*}
		&\mathbf{E}\left[\ip{A_{\kappa_1}(\lm{Y}{j-2})}{
			a(X_{(j-1)\Delta_n})\parens{\frac{1}{\Delta_n}\parens{\zeta_{j+1,n}'}^{\otimes 2}-m_n'I_r}a(X_{(j-1)\Delta_n})^{T}}\right.\\
		&\qquad\left.\left.\times\ip{A_{\kappa_2}(\lm{Y}{j-2})}{
			a(X_{(j-1)\Delta_n})\parens{\frac{1}{\Delta_n}\parens{\zeta_{j+1,n}'}^{\otimes 2}-m_n'I_r}a(X_{(j-1)\Delta_n})^{T}}\right|\mathcal{H}_j^n\right]\\
		&=\parens{m_n'}^2\tr\parens{A_{\kappa_1}(\lm{Y}{j-2})c(X_{(j-1)\Delta_n})A_{\kappa_2}(\lm{Y}{j-2})c(X_{(j-1)\Delta_n})}\\
		&\qquad+\parens{m_n'}^2\tr\parens{A_{\kappa_1}(\lm{Y}{j-2})^Tc(X_{(j-1)\Delta_n})A_{\kappa_2}(\lm{Y}{j-2})c(X_{(j-1)\Delta_n})}
	\end{align*}
	and
	\begin{align*}
		&\mathbf{E}\left[2\ip{\bar{A}_{\kappa_1}(\lm{Y}{j-2})}{a(X_{(j-1)\Delta_n})
			\parens{\frac{1}{\Delta_n}\parens{\zeta_{j,n}\parens{\zeta_{j+1,n}'}^T}}a(X_{(j-1)\Delta_n})^{T}}\right.\\
		&\qquad\left.\left.\times2\ip{\bar{A}_{\kappa_2}(\lm{Y}{j-2})}{a(X_{(j-1)\Delta_n})
			\parens{\frac{1}{\Delta_n}\parens{\zeta_{j,n}\parens{\zeta_{j+1,n}'}^T}}a(X_{(j-1)\Delta_n})^{T}}\right|\mathcal{H}_j^n\right]\\
		&=\frac{4m_n'}{\Delta_n}\parens{\zeta_{j,n}}^Ta(X_{(j-1)\Delta_n})^{T}\bar{A}_{\kappa_1}(\lm{Y}{j-2})c(X_{(j-1)\Delta_n})
		\bar{A}_{\kappa_2}(\lm{Y}{j-2})a(X_{(j-1)\Delta_n})
		\parens{\zeta_{j,n}}
	\end{align*}
	and
	\begin{align*}
		&\mathbf{E}\left[\ip{A_{\kappa_1}(\lm{Y}{j-1})}{
			a(X_{j\Delta_n})\parens{\frac{1}{\Delta_n}\parens{\zeta_{j+1,n}}^{\otimes 2}-m_nI_r}a(X_{j\Delta_n})^{T}}\right.\\
		&\qquad\left.\left.\times\ip{A_{\kappa_2}(\lm{Y}{j-2})}{
			a(X_{(j-1)\Delta_n})\parens{\frac{1}{\Delta_n}\parens{\zeta_{j+1,n}'}^{\otimes 2}-m_n'I_r}a(X_{(j-1)\Delta_n})^{T}}\right|\mathcal{H}_j^n\right]\\
		&=\chi_n^2 \tr\left\{a(X_{j\Delta_n})^{T}A_{\kappa_1}(\lm{Y}{j-1})
		a(X_{j\Delta_n})a(X_{(j-1)\Delta_n})^{T}A_{\kappa_2}(\lm{Y}{j-2})a(X_{(j-1)\Delta_n})\right\}\\
		&\qquad+\chi_n^2 \tr\left\{a(X_{j\Delta_n})^{T}A_{\kappa_1}(\lm{Y}{j-1})^T
		a(X_{j\Delta_n})a(X_{(j-1)\Delta_n})^{T}A_{\kappa_2}(\lm{Y}{j-2})a(X_{(j-1)\Delta_n})\right\}
	\end{align*}
	and
	\begin{align*}
		&\mathbf{E}\left[\ip{A_{\kappa_1}(\lm{Y}{j-1})}{
			a(X_{j\Delta_n})\parens{\frac{1}{\Delta_n}\parens{\zeta_{j+1,n}}^{\otimes 2}-m_nI_r}a(X_{j\Delta_n})^{T}}\right.\\
		&\qquad\times\left.\left.2\ip{\bar{A}_{\kappa_2}(\lm{Y}{j-2})}{a(X_{(j-1)\Delta_n})
			\parens{\frac{1}{\Delta_n}\parens{\zeta_{j,n}\parens{\zeta_{j+1,n}'}^T}}a(X_{(j-1)\Delta_n})^{T}}\right|\mathcal{H}_j^n\right]\\
		&=0
	\end{align*}
	and
	\begin{align*}
		&\mathbf{E}\left[\ip{A_{\kappa_1}(\lm{Y}{j-2})}{
			a(X_{(j-1)\Delta_n})\parens{\frac{1}{\Delta_n}\parens{\zeta_{j+1,n}'}^{\otimes 2}-m_n'I_r}a(X_{(j-1)\Delta_n})^{T}}\right.\\
		&\qquad\times\left.\left.2\ip{\bar{A}_{\kappa_2}(\lm{Y}{j-2})}{a(X_{(j-1)\Delta_n})
			\parens{\frac{1}{\Delta_n}\parens{\zeta_{j,n}\parens{\zeta_{j+1,n}'}^T}}a(X_{(j-1)\Delta_n})^{T}}\right|\mathcal{H}_j^n\right]\\
		&=0.
	\end{align*}
	Hence we obtain
	\begin{align*}
		&\CE{\parens{s_{j,n}^{(1)}(\kappa_1)}\parens{s_{j,n}^{(1)}(\kappa_2)}}{\mathcal{H}_j^n}\\
		&=m_n^2\tr\parens{A_{\kappa_1}(\lm{Y}{j-1})c(X_{j\Delta_n})A_{\kappa_2}(\lm{Y}{j-1})c(X_{j\Delta_n})}\\
		&\qquad+m_n^2\tr\parens{A_{\kappa_1}(\lm{Y}{j-1})^Tc(X_{j\Delta_n})A_{\kappa_2}(\lm{Y}{j-1})c(X_{j\Delta_n})}\\
		&\qquad+\parens{m_n'}^2\tr\parens{A_{\kappa_1}(\lm{Y}{j-2})c(X_{(j-1)\Delta_n})A_{\kappa_2}(\lm{Y}{j-2})c(X_{(j-1)\Delta_n})}\\
		&\qquad+\parens{m_n'}^2\tr\parens{A_{\kappa_1}(\lm{Y}{j-2})^Tc(X_{(j-1)\Delta_n})A_{\kappa_2}(\lm{Y}{j-2})c(X_{(j-1)\Delta_n})}\\
		&\qquad+\frac{4m_n'}{\Delta_n}\parens{\zeta_{j,n}}^Ta(X_{(j-1)\Delta_n})^{T}\bar{A}_{\kappa_1}(\lm{Y}{j-2})c(X_{(j-1)\Delta_n})
		\bar{A}_{\kappa_2}(\lm{Y}{j-2})a(X_{(j-1)\Delta_n})
		\parens{\zeta_{j,n}}\\
		&\qquad+\chi_n^2 \tr\left\{a(X_{j\Delta_n})^{T}A_{\kappa_1}(\lm{Y}{j-1})
		a(X_{j\Delta_n})a(X_{(j-1)\Delta_n})^{T}A_{\kappa_2}(\lm{Y}{j-2})a(X_{(j-1)\Delta_n})\right\}\\
		&\qquad+\chi_n^2 \tr\left\{a(X_{j\Delta_n})^{T}A_{\kappa_1}(\lm{Y}{j-1})^T
		a(X_{j\Delta_n})a(X_{(j-1)\Delta_n})^{T}A_{\kappa_2}(\lm{Y}{j-2})a(X_{(j-1)\Delta_n})\right\}\\
		&\qquad+\chi_n^2 \tr\left\{a(X_{j\Delta_n})^{T}A_{\kappa_2}(\lm{Y}{j-1})
		a(X_{j\Delta_n})a(X_{(j-1)\Delta_n})^{T}A_{\kappa_1}(\lm{Y}{j-2})a(X_{(j-1)\Delta_n})\right\}\\
		&\qquad+\chi_n^2 \tr\left\{a(X_{j\Delta_n})^{T}A_{\kappa_2}(\lm{Y}{j-1})^T
		a(X_{j\Delta_n})a(X_{(j-1)\Delta_n})^{T}A_{\kappa_1}(\lm{Y}{j-2})a(X_{(j-1)\Delta_n})\right\}.
	\end{align*}
	We have the following evaluation
	\begin{align*}
		&\frac{1}{k_n}\sum_{j=2}^{k_n}\CE{\frac{4m_n'}{\Delta_n}\parens{\zeta_{j,n}}^Ta(X_{(j-1)\Delta_n})^{T}\bar{A}_{\kappa_1}(\lm{Y}{j-2})c(X_{(j-1)\Delta_n})
			\bar{A}_{\kappa_2}(\lm{Y}{j-2})a(X_{(j-1)\Delta_n})
			\parens{\zeta_{j,n}}}{\mathcal{H}_j^n}\\
		&\cp \frac{4}{9}\nu_0\parens{\tr\parens{\bar{A}_{\kappa_1}(\cdot)c(\cdot)\bar{A}_{\kappa_2}(\cdot)c(\cdot)}}
	\end{align*}
	and
	\begin{align*}
		&\frac{1}{k_n^2}\sum_{j=2}^{k_n}\CE{\abs{\frac{4m_n'}{\Delta_n}\parens{\zeta_{j,n}}^Ta(X_{(j-1)\Delta_n})^{T}\bar{A}_{\kappa_1}(\lm{Y}{j-2})c(X_{(j-1)\Delta_n})
			\bar{A}_{\kappa_2}(\lm{Y}{j-2})a(X_{(j-1)\Delta_n})
			\parens{\zeta_{j,n}}}^2}{\mathcal{H}_j^n}\\
		&=o_P(1).
	\end{align*}
	Then we have
	\begin{align*}
		\frac{1}{k_n}\sum_{j=2}^{k_n-2}\CE{\parens{s_{j,n}^{(1)}(\kappa_1)}\parens{s_{j,n}^{(1)}(\kappa_2)}}{\mathcal{H}_j^n}
		\cp\nu_0\parens{\tr\parens{\bar{A}_{\kappa_1}(\cdot)c(\cdot)\bar{A}_{\kappa_2}(\cdot)c(\cdot)}}.
	\end{align*}
	Now let us consider the fourth conditional expectation. It can be evaluated such that
	\begin{align*}
		\CE{\parens{s_{j,n}^{(1)}(\kappa)}^4}{\mathcal{H}_j^n}
		&\le C\norm{a(X_{j\Delta_n})^{T}A_{\kappa}(\lm{Y}{j-1})a(X_{j\Delta_n})}^4\\
		&\qquad+C\norm{a(X_{(j-1)\Delta_n})^{T}A_{\kappa}(\lm{Y}{j-2})a(X_{(j-1)\Delta_n})}^4\\
		&\qquad +C\norm{a(X_{(j-1)\Delta_n})^{T}\bar{A}_{\kappa}(\lm{Y}{j-2})a(X_{(j-1)\Delta_n})}^4\frac{1}{\Delta_n^2}\norm{\zeta_{j,n}}^4
	\end{align*}
	and then
	\begin{align*}
		\E{\abs{\CE{\parens{s_{j,n}^{(1)}(\kappa)}^4}{\mathcal{H}_j^n}}}\le C.
	\end{align*}
	Therefore,
	\begin{align*}
		\E{\abs{\frac{1}{k_n^2}\sum_{j=2}^{k_n-2}\CE{\parens{s_{j,n}^{(1)}(\kappa)}^4}{\mathcal{H}_j^n}}}\to0.
	\end{align*}
	Next we consider $U_n^{(2)}(\kappa)=o_P(1)$. Because of Corollary \ref*{cor736},
	\begin{align*}
		\E{\abs{\frac{1}{\sqrt{k_n}}\sum_{j=1}^{k_n-2}\CE{u_{j,k}^{(2)}(\kappa)}{\mathcal{H}_j^n}}}
		\to0
	\end{align*}
	and
	\begin{align*}
		\E{\abs{\frac{1}{k_n}\sum_{j=1}^{k_n-2}\CE{\abs{u_{j,k}^{(2)}(\kappa)}^2}{\mathcal{H}_j^n}}}\to0.
	\end{align*}
	Hence $U_n^{(2)}(\kappa)=o_P(1)$ because of Lemma 9 in \citep{GeJ93}.We will see the asymptotic behaviour of $U_n^{(3)}(\kappa)$ in the next place. As $u_n^{(1)}(\kappa)$, $u_{j,n}^{(3)}(\kappa)$ contains $\mathcal{H}_{j+1}^{n}$-measurable $\lm{\epsilon}{j}$
	 and $\mathcal{H}_{j+2}^{n}$-measurable $\lm{\epsilon}{j+1}$. Hence we rewrite the summation as follows:
	\begin{align*}
		U_n^{(3)}(\kappa)=\frac{1}{\sqrt{k_n}}\sum_{j=1}^{k_n-2}u_{j,n}^{(3)}(\kappa)
		=\frac{1}{\sqrt{k_n}}\sum_{j=2}^{k_n-2}s_{j,n}^{(3)}(\kappa)+o_P(1),
	\end{align*}
	where
	\begin{align*}
		s_{j,n}^{(3)}(\kappa)&:=\ip{\parens{A_{\kappa}(\lm{Y}{j-2})+A_{\kappa}(\lm{Y}{j-1})}}{\Lambda_{\star}^{1/2}
			\parens{\frac{1}{\Delta_n}\parens{\lm{\epsilon}{j}}^{\otimes 2}-\Delta_n^{\frac{2-\tau}{\tau-1}}I_d}\Lambda_{\star}^{1/2}}\\
		&\qquad+\ip{A_{\kappa}(\lm{Y}{j-2})}{\Lambda_{\star}^{1/2}
			\parens{-\frac{2}{\Delta_n}\parens{\lm{\epsilon}{j-1}}\parens{\lm{\epsilon}{j}}^T}\Lambda_{\star}^{1/2}}.
	\end{align*}
	We can obtain
	\begin{align*}
		\CE{s_{j,n}^{(3)}(\kappa)}{\mathcal{H}_j^n}
		=0.
	\end{align*}
	For all $\kappa_1$ and $\kappa_2$, 
	\begin{align*}
		&s_{j,n}^{(3)}(\kappa_1)s_{j,n}^{(3)}(\kappa_2)\\
		&=\ip{\parens{A_{\kappa_1}(\lm{Y}{j-2})+A_{\kappa_1}(\lm{Y}{j-1})}}{\Lambda_{\star}^{1/2}
			\parens{\frac{1}{\Delta_n}\parens{\lm{\epsilon}{j}}^{\otimes 2}}\Lambda_{\star}^{1/2}}\\
		&\qquad\times\ip{\parens{A_{\kappa_2}(\lm{Y}{j-2})+A_{\kappa_2}(\lm{Y}{j-1})}}{\Lambda_{\star}^{1/2}
			\parens{\frac{1}{\Delta_n}\parens{\lm{\epsilon}{j}}^{\otimes 2}}\Lambda_{\star}^{1/2}}\\
		&\quad-\Delta_n^{\frac{2-\tau}{\tau-1}}\ip{\parens{A_{\kappa_1}(\lm{Y}{j-2})+A_{\kappa_1}(\lm{Y}{j-1})}}{\Lambda_{\star}^{1/2}
			\parens{\frac{1}{\Delta_n}\parens{\lm{\epsilon}{j}}^{\otimes 2}}\Lambda_{\star}^{1/2}}
		\ip{\parens{A_{\kappa_2}(\lm{Y}{j-2})+A_{\kappa_2}(\lm{Y}{j-1})}}{\Lambda_{\star}}\\
		&\quad-\Delta_n^{\frac{2-\tau}{\tau-1}}\ip{\parens{A_{\kappa_1}(\lm{Y}{j-2})+A_{\kappa_1}(\lm{Y}{j-1})}}{\Lambda_{\star}}\\
		&\qquad\times\ip{\parens{A_{\kappa_2}(\lm{Y}{j-2})+A_{\kappa_2}(\lm{Y}{j-1})}}{\Lambda_{\star}^{1/2}
			\parens{\frac{1}{\Delta_n}\parens{\lm{\epsilon}{j}}^{\otimes 2}-\Delta_n^{\frac{2-\tau}{\tau-1}}I_d}\Lambda_{\star}^{1/2}}\\
		&\quad+\ip{\parens{A_{\kappa_1}(\lm{Y}{j-2})+A_{\kappa_1}(\lm{Y}{j-1})}}{\Lambda_{\star}^{1/2}
			\parens{\frac{1}{\Delta_n}\parens{\lm{\epsilon}{j}}^{\otimes 2}}\Lambda_{\star}^{1/2}}\\
		&\quad\qquad\times\ip{A_{\kappa_2}(\lm{Y}{j-2})}{\Lambda_{\star}^{1/2}
			\parens{-\frac{2}{\Delta_n}\parens{\lm{\epsilon}{j-1}}\parens{\lm{\epsilon}{j}}^T}\Lambda_{\star}^{1/2}}\\
		&\quad-\Delta_n^{\frac{2-\tau}{\tau-1}}\ip{\parens{A_{\kappa_1}(\lm{Y}{j-2})+A_{\kappa_1}(\lm{Y}{j-1})}}{\Lambda_{\star}}
		\ip{A_{\kappa_2}(\lm{Y}{j-2})}{\Lambda_{\star}^{1/2}
			\parens{-\frac{2}{\Delta_n}\parens{\lm{\epsilon}{j-1}}\parens{\lm{\epsilon}{j}}^T}\Lambda_{\star}^{1/2}}\\
		&\quad+\ip{A_{\kappa_1}(\lm{Y}{j-2})}{\Lambda_{\star}^{1/2}
			\parens{-\frac{2}{\Delta_n}\parens{\lm{\epsilon}{j-1}}\parens{\lm{\epsilon}{j}}^T}\Lambda_{\star}^{1/2}}\\
		&\quad\qquad\times\ip{\parens{A_{\kappa_2}(\lm{Y}{j-2})+A_{\kappa_2}(\lm{Y}{j-1})}}{\Lambda_{\star}^{1/2}
			\parens{\frac{1}{\Delta_n}\parens{\lm{\epsilon}{j}}^{\otimes 2}}\Lambda_{\star}^{1/2}}\\
		&\quad-\Delta_n^{\frac{2-\tau}{\tau-1}}\ip{A_{\kappa_1}(\lm{Y}{j-2})}{\Lambda_{\star}^{1/2}
			\parens{-\frac{2}{\Delta_n}\parens{\lm{\epsilon}{j-1}}\parens{\lm{\epsilon}{j}}^T}\Lambda_{\star}^{1/2}}
		\ip{\parens{A_{\kappa_2}(\lm{Y}{j-2})+A_{\kappa_2}(\lm{Y}{j-1})}}{\Lambda_{\star}}\\
		&\quad+\frac{4}{\Delta_n^2}\parens{\lm{\epsilon}{j-1}}^T\Lambda_{\star}^{1/2}A_{\kappa_1}(\lm{Y}{j-2})\Lambda_{\star}^{1/2}\parens{\lm{\epsilon}{j}}
		\parens{\lm{\epsilon}{j}}^T\Lambda_{\star}^{1/2}\parens{A_{\kappa_2}(\lm{Y}{j-2})}^T\Lambda_{\star}^{1/2}\parens{\lm{\epsilon}{j-1}}.
	\end{align*}
	Hence we can evaluate
	\begin{align*}
		&\CE{s_{j,n}^{(3)}(\kappa_1)s_{j,n}^{(3)}(\kappa_2)}{\mathcal{H}_j^n}\\
		&=\frac{1}{\Delta_n^2}\mathbf{E}\left[\parens{\lm{\epsilon}{j}}^T\Lambda_{\star}^{1/2}\parens{A_{\kappa_1}(\lm{Y}{j-2})+A_{\kappa_1}(\lm{Y}{j-1})}\Lambda_{\star}^{1/2}
			\parens{\lm{\epsilon}{j}}^{\otimes 2}\right.\\
		&\qquad\qquad\left.\left.\times\Lambda_{\star}^{1/2}\parens{A_{\kappa_2}(\lm{Y}{j-2})+A_{\kappa_2}(\lm{Y}{j-1})}\Lambda_{\star}^{1/2}
		\parens{\lm{\epsilon}{j}}\right|\mathcal{H}_j^n\right]\\
		&\qquad-\parens{\Delta_n^{\frac{2-\tau}{\tau-1}}}^2\ip{\parens{A_{\kappa_1}(\lm{Y}{j-2})+A_{\kappa_1}(\lm{Y}{j-1})}}{\Lambda_{\star}}
		\ip{\parens{A_{\kappa_2}(\lm{Y}{j-2})+A_{\kappa_2}(\lm{Y}{j-1})}}{\Lambda_{\star}}\\
		&\qquad+\frac{4\Delta_n^{\frac{2-\tau}{\tau-1}}}{\Delta_n}\parens{\lm{\epsilon}{j-1}}^T\Lambda_{\star}^{1/2}A_{\kappa_1}(\lm{Y}{j-2})\Lambda_{\star}\parens{A_{\kappa_2}(\lm{Y}{j-2})}^T\Lambda_{\star}^{1/2}\parens{\lm{\epsilon}{j-1}}.
	\end{align*}
	Note the fact that for any $\Re^d\times\Re^d$-valued matrix $A$
	\begin{align*}
		\E{\parens{\lm{\epsilon}{j}}^{\otimes2}A\parens{\lm{\epsilon}{j}}^{\otimes2}}
		&=\frac{1}{p_n^2}\parens{2\bar{A}}+\frac{1}{p_n^2}\tr\parens{A}I_d
		+\crotchet{\parens{\frac{\E{\parens{\epsilon_{0}^i}^4}-3}{p_n^3}}A^{i,i}}_{i,i},
	\end{align*}
	where $\bar{A}=\parens{A+A^T}/2$. Therefore,
	\begin{align*}
		&\mathbf{E}\left[\parens{\lm{\epsilon}{j}}^T\Lambda_{\star}^{1/2}\parens{A_{\kappa_1}(\lm{Y}{j-2})+A_{\kappa_1}(\lm{Y}{j-1})}\Lambda_{\star}^{1/2}
		\parens{\lm{\epsilon}{j}}^{\otimes 2}\right.\\
		&\qquad\qquad\left.\left.\times\Lambda_{\star}^{1/2}\parens{A_{\kappa_2}(\lm{Y}{j-2})+A_{\kappa_2}(\lm{Y}{j-1})}\Lambda_{\star}^{1/2}
		\parens{\lm{\epsilon}{j}}\right|\mathcal{H}_j^n\right]\\
		&=\frac{2}{p_n^2}\tr\tuborg{\Lambda_{\star}^{1/2}\parens{\bar{A}_{\kappa_1}(\lm{Y}{j-2})+\bar{A}_{\kappa_1}(\lm{Y}{j-1})}\Lambda_{\star}^{1/2}\Lambda_{\star}^{1/2}\parens{A_{\kappa_2}(\lm{Y}{j-2})+A_{\kappa_2}(\lm{Y}{j-1})}\Lambda_{\star}^{1/2}}\\
		&\qquad+\frac{1}{p_n^2}\tr\tuborg{\Lambda_{\star}^{1/2}\parens{A_{\kappa_1}(\lm{Y}{j-2})+A_{\kappa_1}(\lm{Y}{j-1})}\Lambda_{\star}^{1/2}}
		\tr\tuborg{\Lambda_{\star}^{1/2}\parens{A_{\kappa_2}(\lm{Y}{j-2})+A_{\kappa_2}(\lm{Y}{j-1})}\Lambda_{\star}^{1/2}}\\
		&\qquad+\sum_{i=1}^{d}\parens{\frac{\E{\parens{\epsilon_{0}^i}^4}-3}{p_n^3}}\parens{\Lambda_{\star}^{1/2}\parens{A_{\kappa_1}(\lm{Y}{j-2})+A_{\kappa_1}(\lm{Y}{j-1})}\Lambda_{\star}^{1/2}}^{i,i}\\
		&\qquad\qquad\qquad\times\parens{\Lambda_{\star}^{1/2}\parens{A_{\kappa_2}(\lm{Y}{j-2})+A_{\kappa_2}(\lm{Y}{j-1})}\Lambda_{\star}^{1/2}}^{i,i}.
	\end{align*}
	Hence
	\begin{align*}
		&\frac{1}{k_n}\sum_{j=2}^{k_n-2}\CE{s_{j,n}^{(3)}(\kappa_1)s_{j,n}^{(3)}(\kappa_2)}{\mathcal{H}_j^n}\\
		&=\frac{1}{k_n}\sum_{j=2}^{k_n-2}\frac{2}{p_n^2\Delta_n^2}\tr\tuborg{\parens{\bar{A}_{\kappa_1}(\lm{Y}{j-2})+\bar{A}_{\kappa_1}(\lm{Y}{j-1})}\Lambda_{\star}\parens{A_{\kappa_2}(\lm{Y}{j-2})+A_{\kappa_2}(\lm{Y}{j-1})}\Lambda_{\star}}\\
		&\qquad+\frac{1}{k_n}\sum_{j=2}^{k_n-2}\frac{4\Delta_n^{\frac{2-\tau}{\tau-1}}}{\Delta_n}\parens{\lm{\epsilon}{j-1}}^T\Lambda_{\star}^{1/2}A_{\kappa_1}(\lm{Y}{j-2})\Lambda_{\star}\parens{A_{\kappa_1}(\lm{Y}{j-2})}^T\Lambda_{\star}^{1/2}\parens{\lm{\epsilon}{j-1}}\\
		&\qquad+o_P(1).
	\end{align*}
	We have
	\begin{align*}
		&\frac{1}{k_n}\sum_{j=2}^{k_n-2}\CE{\frac{4\Delta_n^{\frac{2-\tau}{\tau-1}}}{\Delta_n}\parens{\lm{\epsilon}{j-1}}^T\Lambda_{\star}^{1/2}A_{\kappa_1}(\lm{Y}{j-2})\Lambda_{\star}\parens{A_{\kappa_2}(\lm{Y}{j-2})}^T\Lambda_{\star}^{1/2}\parens{\lm{\epsilon}{j-1}}}
		{\mathcal{H}_j^n}\\
		&\cp\begin{cases}
		0 & \text{ if }\tau\in(1,2)\\
		4\tr\tuborg{A_{\kappa_1}(\cdot)\Lambda_{\star}A_{\kappa_2}(\cdot)\Lambda_{\star}} &\text{ if }\tau=2
		\end{cases}
	\end{align*}
	and
	\begin{align*}
		\frac{1}{k_n^2}\sum_{j=2}^{k_n-2}\frac{16\parens{\Delta_n^{\frac{2-\tau}{\tau-1}}}^2}{\Delta_n^2}\CE{\parens{\parens{\lm{\epsilon}{j-1}}^T\Lambda_{\star}^{1/2}A_{\kappa_1}(\lm{Y}{j-2})\Lambda_{\star}\parens{A_{\kappa_2}(\lm{Y}{j-2})}^T\Lambda_{\star}^{1/2}\parens{\lm{\epsilon}{j-1}}}^2}
		{\mathcal{H}_j^n}=o_P(1).
	\end{align*}
	To sum up if $\tau\in(1,2)$ we obtain
	\begin{align*}
		&\frac{1}{k_n}\sum_{j=2}^{k_n-2}\CE{s_{j,n}^{(3)}(\kappa_1)s_{j,n}^{(3)}(\kappa_2)}{\mathcal{H}_j^n}\cp 0
	\end{align*}
	and if $\tau=2$
	\begin{align*}
		&\frac{1}{k_n}\sum_{j=2}^{k_n-2}\CE{s_{j,n}^{(3)}(\kappa_1)s_{j,n}^{(3)}(\kappa_2)}{\mathcal{H}_j^n}\\
		&\cp {2}\nu_0\parens{\tr\tuborg{\parens{\bar{A}_{\kappa_1}(\cdot)+\bar{A}_{\kappa_1}(\cdot)}\Lambda_{\star}\parens{A_{\kappa_2}(\cdot)
					+A_{\kappa_2}(\cdot)}\Lambda_{\star}}}\\
		&\qquad+4\nu_0\parens{\tr\tuborg{A_{\kappa_1}(\cdot)\Lambda_{\star}A_{\kappa_2}(\cdot)\Lambda_{\star}}}\\
		&=12\nu_0\parens{\tr\tuborg{\bar{A}_{\kappa_1}(\cdot)\Lambda_{\star}\bar{A}_{\kappa_2}(\cdot)\Lambda_{\star}}}.
	\end{align*}
	Therefore, $U_n^{(3)}(\kappa)=o_P(1)$ if $\tau\in(1,2)$. The conditional fourth expectation of $s_{j,n}^{(3)}$ can be evaluated as
	\begin{align*}
		\CE{\parens{s_{j,n}^{(3)}(\kappa)}^4}{\mathcal{H}_j^n}
		\le C\norm{\parens{A_{\kappa}(\lm{Y}{j-2})+A_{\kappa}(\lm{Y}{j-1})}}^4+\frac{C}{\Delta_n^2}\norm{A_{\kappa}(\lm{Y}{j-2})}^4\norm{\lm{\epsilon}{j-1}}^4
	\end{align*}
	and hence
	\begin{align*}
		\E{\abs{\frac{1}{k_n^2}\sum_{j=2}^{k_n-2}\CE{\parens{s_{j,n}^{(3)}(\kappa)}^4}{\mathcal{H}_j^n}}}
		\to0.
	\end{align*}
	Next, we see the asymptotic behaviour of $U_{n}^{(4)}(\kappa)$. We again rewrite the summation as follows:
	\begin{align*}
		U_{n}^{(4)}(\kappa)&:=\frac{1}{\sqrt{k_n}}\sum_{j=1}^{k_n-2}\frac{2}{\Delta_n}\ip{\bar{A}_{\kappa}(\lm{Y}{j-1})}{
			a(X_{j\Delta_n})\parens{\zeta_{j+1,n}+\zeta_{j+2,n}'}
			\parens{\lm{\epsilon}{j+1}-\lm{\epsilon}{j}}^T\Lambda_{\star}^{1/2}}\\
		&=\frac{1}{\sqrt{k_n}}\sum_{j=2}^{k_n-2}s_{j,n}^{(4)}(\kappa)+o_P(1),
	\end{align*}
	where
	\begin{align*}
		s_{j,n}^{(4)}(\kappa)&:=\frac{2}{\Delta_n}\ip{\bar{A}_{\kappa}(\lm{Y}{j-2})}{
		a(X_{(j-1)\Delta_n})\parens{\zeta_{j,n}}
		\parens{\lm{\epsilon}{j}}^T\Lambda_{\star}^{1/2}}\\
		&\qquad-\frac{2}{\Delta_n}\ip{\bar{A}_{\kappa}(\lm{Y}{j-1})}{
			a(X_{j\Delta_n})\parens{\zeta_{j+1,n}}
			\parens{\lm{\epsilon}{j}}^T\Lambda_{\star}^{1/2}}\\
		&\qquad+\frac{2}{\Delta_n}\ip{\bar{A}_{\kappa}(\lm{Y}{j-2})}{
			a(X_{(j-1)\Delta_n})\parens{\zeta_{j+1,n}'}
			\parens{\lm{\epsilon}{j}}^T\Lambda_{\star}^{1/2}}\\
		&\qquad-\frac{2}{\Delta_n}\ip{\bar{A}_{\kappa}(\lm{Y}{j-2})}{
			a(X_{(j-1)\Delta_n})\parens{\zeta_{j+1,n}'}
			\parens{\lm{\epsilon}{j-1}}^T\Lambda_{\star}^{1/2}}.
	\end{align*}
	Hence it is enough to examine $\frac{1}{\sqrt{k_n}}\sum_{j=2}^{k_n-2}s_{j,n}^{(4)}(\kappa)$. It is obvious that
	\begin{align*}
		\CE{s_{j,n}^{(4)}(\kappa)}{\mathcal{H}_j^n}=0.
	\end{align*}
	For all $\kappa_1$ and $\kappa_2$,
	\begin{align*}
		&\parens{\frac{2}{\Delta_n}}^{-2}\CE{s_{j,n}^{(4)}(\kappa_1)s_{j,n}^{(4)}(\kappa_2)}{\mathcal{H}_j^n}\\
		&=\frac{1}{p_n}\parens{\zeta_{j,n}}^Ta(X_{(j-1)\Delta_n})^T\bar{A}_{\kappa_1}(\lm{Y}{j-2})\Lambda_{\star}\bar{A}_{\kappa_2}(\lm{Y}{j-2})a(X_{(j-1)\Delta_n})\parens{\zeta_{j,n}}\\
		&\qquad+\frac{m_n\Delta_n}{p_n}\tr\tuborg{\bar{A}_{\kappa_1}(\lm{Y}{j-1})\Lambda_{\star}\bar{A}_{\kappa_2}(\lm{Y}{j-1})c(X_{j\Delta_n})}\\
		&\qquad+\frac{m_n'\Delta_n}{p_n}\tr\tuborg{\bar{A}_{\kappa_1}(\lm{Y}{j-2})\Lambda_{\star}\bar{A}_{\kappa_2}(\lm{Y}{j-2})c(X_{(j-1)\Delta_n})}\\
		&\qquad+m_n'\Delta_n\parens{\lm{\epsilon}{j-1}}^T\Lambda_{\star}^{1/2}\bar{A}_{\kappa_1}(\lm{Y}{j-2})
		c(X_{(j-1)\Delta_n})\bar{A}_{\kappa_2}(\lm{Y}{j-2})\Lambda_{\star}^{1/2}\parens{\lm{\epsilon}{j-1}}\\
		&\qquad-\frac{\chi_n\Delta_n}{p_n}\tr\tuborg{\bar{A}_{\kappa_1}(\lm{Y}{j-1})\Lambda_{\star}\bar{A}_{\kappa_2}(\lm{Y}{j-2})a(X_{(j-1)\Delta_n})a(X_{j\Delta_n})^T}\\
		&\qquad-\frac{\chi_n\Delta_n}{p_n}\tr\tuborg{\bar{A}_{\kappa_2}(\lm{Y}{j-1})\Lambda_{\star}\bar{A}_{\kappa_1}(\lm{Y}{j-2})a(X_{(j-1)\Delta_n})a(X_{j\Delta_n})^T}
	\end{align*}
	and then
	\begin{align*}
		&\frac{1}{k_n}\sum_{j=2}^{k_n-2}\CE{s_{j,n}^{(4)}(\kappa_1)s_{j,n}^{(4)}(\kappa_2)}{\mathcal{H}_j^n}\\
		&=\frac{1}{k_n}\sum_{j=2}^{k_n-2}\parens{\frac{2}{\Delta_n}}^{2}\frac{1}{p_n}\parens{\zeta_{j,n}}^Ta(X_{(j-1)\Delta_n})^T\bar{A}_{\kappa_1}(\lm{Y}{j-2})\Lambda_{\star}\bar{A}_{\kappa_2}(\lm{Y}{j-2})a(X_{(j-1)\Delta_n})\parens{\zeta_{j,n}}\\
		&\qquad+\frac{1}{k_n}\sum_{j=2}^{k_n-2}\parens{\frac{2}{\Delta_n}}^{2}\frac{m_n\Delta_n}{p_n}\tr\tuborg{\bar{A}_{\kappa_1}(\lm{Y}{j-1})\Lambda_{\star}\bar{A}_{\kappa_2}(\lm{Y}{j-1})c(X_{j\Delta_n})}\\
		&\qquad+\frac{1}{k_n}\sum_{j=2}^{k_n-2}\parens{\frac{2}{\Delta_n}}^{2}\frac{m_n'\Delta_n}{p_n}\tr\tuborg{\bar{A}_{\kappa_1}(\lm{Y}{j-2})\Lambda_{\star}\bar{A}_{\kappa_2}(\lm{Y}{j-2})c(X_{(j-1)\Delta_n})}\\
		&\qquad+\frac{1}{k_n}\sum_{j=2}^{k_n-2}\parens{\frac{2}{\Delta_n}}^{2}m_n'\Delta_n\parens{\lm{\epsilon}{j-1}}^T\Lambda_{\star}^{1/2}\bar{A}_{\kappa_1}(\lm{Y}{j-2})
		c(X_{(j-1)\Delta_n})\bar{A}_{\kappa_2}(\lm{Y}{j-2})\Lambda_{\star}^{1/2}\parens{\lm{\epsilon}{j-1}}\\
		&\qquad-\frac{1}{k_n}\sum_{j=2}^{k_n-2}\parens{\frac{2}{\Delta_n}}^{2}\frac{\chi_n\Delta_n}{p_n}\tr\tuborg{\bar{A}_{\kappa_1}(\lm{Y}{j-1})\Lambda_{\star}\bar{A}_{\kappa_2}(\lm{Y}{j-2})a(X_{(j-1)\Delta_n})a(X_{j\Delta_n})^T}\\
		&\qquad-\frac{1}{k_n}\sum_{j=2}^{k_n-2}\parens{\frac{2}{\Delta_n}}^{2}\frac{\chi_n\Delta_n}{p_n}\tr\tuborg{\bar{A}_{\kappa_2}(\lm{Y}{j-1})\Lambda_{\star}\bar{A}_{\kappa_1}(\lm{Y}{j-2})a(X_{(j-1)\Delta_n})a(X_{j\Delta_n})^T}.
	\end{align*}
	We examine the terms in right hand side respectively. With respect to the first term,
	\begin{align*}
		&\frac{1}{k_n}\sum_{j=2}^{k_n-2}\parens{\frac{2}{\Delta_n}}^{2}\frac{1}{p_n}\CE{\parens{\zeta_{j,n}}^Ta(X_{(j-1)\Delta_n})^T\bar{A}_{\kappa_1}(\lm{Y}{j-2})\Lambda_{\star}\bar{A}_{\kappa_2}(\lm{Y}{j-2})a(X_{(j-1)\Delta_n})\parens{\zeta_{j,n}}}{\mathcal{H}_{j-1}^n}\\
		&\cp \begin{cases}
		0 & \text{ if }\tau\in(1,2)\\
		\frac{4}{3}\nu_0\parens{\tr\tuborg{\bar{A}_{\kappa_1}(\cdot)\Lambda_{\star}\bar{A}_{\kappa_2}(\cdot)c(\cdot)}} & \text{ if }\tau=2,
		\end{cases}
	\end{align*}
	and
	\begin{align*}
		&\frac{1}{k_n^2}\sum_{j=2}^{k_n-2}\parens{\frac{2}{\Delta_n}}^{4}\frac{1}{p_n^2}\CE{\abs{\parens{\zeta_{j,n}}^Ta(X_{(j-1)\Delta_n})^T\bar{A}_{\kappa_1}(\lm{Y}{j-2})\Lambda_{\star}\bar{A}_{\kappa_2}(\lm{Y}{j-2})a(X_{(j-1)\Delta_n})\parens{\zeta_{j,n}}}^2}{\mathcal{H}_{j-1}^n}\\
		&\cp0;
	\end{align*}
	therefore
	\begin{align*}
		&\frac{1}{k_n}\sum_{j=2}^{k_n-2}\parens{\frac{2}{\Delta_n}}^{2}\frac{1}{p_n}\parens{\zeta_{j,n}}^Ta(X_{(j-1)\Delta_n})^T\bar{A}_{\kappa_1}(\lm{Y}{j-2})\Lambda_{\star}\bar{A}_{\kappa_2}(\lm{Y}{j-2})a(X_{(j-1)\Delta_n})\parens{\zeta_{j,n}}\\
		&\cp \begin{cases}
		0 & \text{ if }\tau\in(1,2)\\
		\frac{4}{3}\nu_0\parens{\tr\tuborg{\bar{A}_{\kappa_1}(\cdot)\Lambda_{\star}\bar{A}_{\kappa_2}(\cdot)c(\cdot)}} & \text{ if }\tau=2
		\end{cases}
	\end{align*}
	because of Lemma 9 in \citep{GeJ93}. The fourth term can be evaluated as follows:
	\begin{align*}
		&\frac{1}{k_n}\sum_{j=2}^{k_n-2}\CE{\parens{\frac{2}{\Delta_n}}^{2}m_n'\Delta_n\parens{\lm{\epsilon}{j-1}}^T\Lambda_{\star}^{1/2}\bar{A}_{\kappa_1}(\lm{Y}{j-2})
		c(X_{(j-1)\Delta_n})\bar{A}_{\kappa_2}(\lm{Y}{j-2})\Lambda_{\star}^{1/2}\parens{\lm{\epsilon}{j-1}}}{\mathcal{H}_{j-1}^n}\\
		&\cp \begin{cases}
		0 & \text{ if }\tau\in(1,2)\\
		\frac{4}{3}\nu_0\parens{\tr\tuborg{\bar{A}_{\kappa_1}(\cdot)\Lambda_{\star}\bar{A}_{\kappa_2}(\cdot)c(\cdot)}} & \text{ if }\tau=2,
		\end{cases}
	\end{align*}
	and
	\begin{align*}
		&\frac{1}{k_n^2}\sum_{j=2}^{k_n-2}\CE{\parens{\frac{2}{\Delta_n}}^{4}\parens{m_n'\Delta_n}^2\abs{\parens{\lm{\epsilon}{j-1}}^T\Lambda_{\star}^{1/2}\bar{A}_{\kappa_1}(\lm{Y}{j-2})
			c(X_{(j-1)\Delta_n})\bar{A}_{\kappa_2}(\lm{Y}{j-2})\Lambda_{\star}^{1/2}\parens{\lm{\epsilon}{j-1}}}^2}{\mathcal{H}_{j-1}^n}\\
		&\cp 0;
	\end{align*}
	then as the first term we obtain
	\begin{align*}
		&\frac{1}{k_n}\sum_{j=2}^{k_n-2}\parens{\frac{2}{\Delta_n}}^{2}m_n'\Delta_n\parens{\lm{\epsilon}{j-1}}^T\Lambda_{\star}^{1/2}\bar{A}_{\kappa_1}(\lm{Y}{j-2})
		c(X_{(j-1)\Delta_n})\bar{A}_{\kappa_2}(\lm{Y}{j-2})\Lambda_{\star}^{1/2}\parens{\lm{\epsilon}{j-1}}\\
		&\cp \begin{cases}
		0 & \text{ if }\tau\in(1,2)\\
		\frac{4}{3}\nu_0\parens{\tr\tuborg{\bar{A}_{\kappa_1}(\cdot)\Lambda_{\star}\bar{A}_{\kappa_2}(\cdot)c(\cdot)}} & \text{ if }\tau=2.
		\end{cases}
	\end{align*}
	As for the other terms, we have
	\begin{align*}
		&\frac{1}{k_n}\sum_{j=2}^{k_n-2}\parens{\frac{2}{\Delta_n}}^{2}\frac{m_n\Delta_n}{p_n}\tr
		\tuborg{\bar{A}_{\kappa_1}(\lm{Y}{j-1})\Lambda_{\star}\bar{A}_{\kappa_2}(\lm{Y}{j-1})c(X_{j\Delta_n})}\\
		&\quad\cp \begin{cases}
		0 & \text{ if }\tau\in(1,2)\\
		\frac{4}{3}\nu_0\parens{\tr\tuborg{\bar{A}_{\kappa_1}(\cdot)\Lambda_{\star}\bar{A}_{\kappa_2}(\cdot)c(\cdot)}} & \text{ if }\tau=2,
		\end{cases}\\
		&\frac{1}{k_n}\sum_{j=2}^{k_n-2}\parens{\frac{2}{\Delta_n}}^{2}\frac{m_n'\Delta_n}{p_n}\tr
		\tuborg{\bar{A}_{\kappa_1}(\lm{Y}{j-2})\Lambda_{\star}\bar{A}_{\kappa_2}(\lm{Y}{j-2})c(X_{(j-1)\Delta_n})}\\
		&\quad\cp \begin{cases}
		0 & \text{ if }\tau\in(1,2)\\
		\frac{4}{3}\nu_0\parens{\tr\tuborg{\bar{A}_{\kappa_1}(\cdot)\Lambda_{\star}\bar{A}_{\kappa_2}(\cdot)c(\cdot)}} & \text{ if }\tau=2,
		\end{cases}\\
		&\frac{1}{k_n}\sum_{j=2}^{k_n-2}\parens{\frac{2}{\Delta_n}}^{2}\frac{\chi_n\Delta_n}{p_n}\tr\tuborg{\bar{A}_{\kappa_1}(\lm{Y}{j-1})\Lambda_{\star}\bar{A}_{\kappa_2}(\lm{Y}{j-2})a(X_{(j-1)\Delta_n})a(X_{j\Delta_n})^T}\\
		&\quad\cp \begin{cases}
		0 & \text{ if }\tau\in(1,2)\\
		\frac{2}{3}\nu_0\parens{\tr\tuborg{\bar{A}_{\kappa_1}(\cdot)\Lambda_{\star}\bar{A}_{\kappa_2}(\cdot)c(\cdot)}} & \text{ if }\tau=2,
		\end{cases}\\
		&\frac{1}{k_n}\sum_{j=2}^{k_n-2}\parens{\frac{2}{\Delta_n}}^{2}\frac{\chi_n\Delta_n}{p_n}\tr\tuborg{\bar{A}_{\kappa_2}(\lm{Y}{j-1})\Lambda_{\star}\bar{A}_{\kappa_1}(\lm{Y}{j-2})a(X_{(j-1)\Delta_n})a(X_{j\Delta_n})^T}\\
		&\quad\cp \begin{cases}
		0 & \text{ if }\tau\in(1,2)\\
		\frac{2}{3}\nu_0\parens{\tr\tuborg{\bar{A}_{\kappa_2}(\cdot)\Lambda_{\star}\bar{A}_{\kappa_1}(\cdot)c(\cdot)}} & \text{ if }\tau=2.
		\end{cases}
	\end{align*}
	Note the following fact that
	\begin{align*}
		\tr\tuborg{\bar{A}_{\kappa_2}(\cdot)\Lambda_{\star}\bar{A}_{\kappa_1}(\cdot)c(\cdot)}=\tr\tuborg{\bar{A}_{\kappa_1}(\cdot)\Lambda_{\star}\bar{A}_{\kappa_2}(\cdot)c(\cdot)}.
	\end{align*}
	In summary we obtain
	\begin{align*}
		&\frac{1}{k_n}\sum_{j=2}^{k_n-2}\CE{s_{j,n}^{(4)}(\kappa_1)s_{j,n}^{(4)}(\kappa_2)}{\mathcal{H}_j^n}\\
		&\quad\cp \begin{cases}
		0 & \text{ if }\tau\in(1,2)\\
		4\nu_0\parens{\tr\tuborg{\bar{A}_{\kappa_2}(\cdot)\Lambda_{\star}\bar{A}_{\kappa_1}(\cdot)c(\cdot)}} & \text{ if }\tau=2.
		\end{cases}
	\end{align*}
	Hence $U_n^{(4)}(\kappa)=o_P(1)$ if $\tau\in(1,2)$. The conditional fourth moment can be evaluated as
	\begin{align*}
		\E{\abs{\frac{1}{k_n^2}\sum_{j=2}^{k_n-2}\CE{\abs{s_{j,n}^{(4)}(\kappa)}^4}{\mathcal{H}_j^n}}}\to 0.
	\end{align*}
	In the next place, we can see the $L^1$ convergence of $U_n^{(5)}(\kappa)$ such that
	\begin{align*}
		\E{\abs{U_n^{(5)}(\kappa)}}
		&\to0.
	\end{align*}
	To show $U_n^{(6)}(\kappa)=o_P(1)$ and $U_n^{(7)}(\kappa)=o_P(1)$, we use Lemma 9 in \citep{GeJ93}. We have
	\begin{align*}
		\E{\abs{\frac{1}{\sqrt{k_n}}\sum_{j=1}^{k_n-2}\CE{u_{j,n}^{(6)}(\kappa)}{\mathcal{H}_j^n}}}\to 0
	\end{align*}
	because of Proposition \ref*{pro737}, and
	\begin{align*}
		\E{\abs{\frac{1}{k_n}\sum_{j=1}^{k_n-2}\CE{\abs{u_{j,n}^{(6)}(\kappa)}^2}{\mathcal{H}_j^n}}}\to0.
	\end{align*}
	Therefore, $U_n^{(6)}(\kappa)=o_P(1)$. We also obtain
	\begin{align*}
		\CE{u_{j,n}^{(7)}(\kappa)}{\mathcal{H}_j^n}=0
	\end{align*}
	and
	\begin{align*}
		\E{\abs{\frac{1}{k_n}\sum_{j=1}^{k_n-2}\CE{\abs{u_{j,n}^{(7)}(\kappa)}^2}{\mathcal{H}_j^n}}}\to 0. 
	\end{align*}
	Hence $U_n^{(7)}(\kappa)=o_P(1)$.
	
	Finally we see the covariance structure among $U_n^{(1)}$, $U_n^{(3)}$ and $U_n^{(4)}$ when $\tau=2$. Because of the independence of $\tuborg{w_t}$ and $\tuborg{\epsilon_{ih_n}}$, for all $\kappa_1$ and $\kappa_2$,
	\begin{align*}
		\CE{s_{j,n}^{(1)}(\kappa_1)s_{j,n}^{(3)}(\kappa_2)}{\mathcal{H}_j^n}=\CE{s_{j,n}^{(1)}(\kappa_1)}{\mathcal{H}_j^n}\CE{s_{j,n}^{(3)}(\kappa_2)}{\mathcal{H}_j^n}=0.
	\end{align*}
	With respect to the covariance between $U_n^{(1)}$ and $U_n^{(4)}$, the independence of $\tuborg{w_t}$ and $\tuborg{\epsilon_{ih_n}}$ leads to
	\begin{align*}
		\CE{s_{j,n}^{(1)}(\kappa_1)s_{j,n}^{(4)}(\kappa_2)}{\mathcal{H}_j^n}
		=-\frac{4m_n'}{\Delta_n}\parens{\zeta_{j,n}}^{T}a(X_{(j-1)\Delta_n})^T\bar{A}_{\kappa_1}(\lm{Y}{j-2})c(X_{(j-1)\Delta_n})\bar{A}_{\kappa_2}(\lm{Y}{j-2})
		\Lambda_{\star}^{1/2}\parens{\lm{\epsilon}{j-1}}.
	\end{align*}
	Hence
	\begin{align*}
		&\frac{1}{k_n}\sum_{j=2}^{k_n-2}\CE{s_{j,n}^{(1)}(\kappa_1)s_{j,n}^{(4)}(\kappa_2)}{\mathcal{H}_j^n}\\
		&=-\frac{1}{k_n}\sum_{j=2}^{k_n-2}\frac{4m_n'}{\Delta_n}\parens{\zeta_{j,n}}^{T}a(X_{(j-1)\Delta_n})^T\bar{A}_{\kappa_1}(\lm{Y}{j-2})c(X_{(j-1)\Delta_n})\bar{A}_{\kappa_2}(\lm{Y}{j-2})
		\Lambda_{\star}^{1/2}\parens{\lm{\epsilon}{j-1}}.
	\end{align*}
	We have
	\begin{align*}
		\CE{-\frac{4m_n'}{\Delta_n}\parens{\zeta_{j,n}}^{T}a(X_{(j-1)\Delta_n})^T\bar{A}_{\kappa_1}(\lm{Y}{j-2})c(X_{(j-1)\Delta_n})\bar{A}_{\kappa_2}(\lm{Y}{j-2})
		\Lambda_{\star}^{1/2}\parens{\lm{\epsilon}{j-1}}}{\mathcal{H}_{j-1}^n}
		=0
	\end{align*}
	and
	\begin{align*}
		&\CE{\abs{-\frac{4m_n'}{\Delta_n}\parens{\zeta_{j,n}}^{T}a(X_{(j-1)\Delta_n})^T\bar{A}_{\kappa_1}(\lm{Y}{j-2})c(X_{(j-1)\Delta_n})\bar{A}_{\kappa_2}(\lm{Y}{j-2})
			\Lambda_{\star}^{1/2}\parens{\lm{\epsilon}{j-1}}}^2}{\mathcal{H}_{j-1}^n}\\
		&\le C\CE{\norm{a(X_{(j-1)\Delta_n})^T\bar{A}_{\kappa_1}(\lm{Y}{j-2})c(X_{(j-1)\Delta_n})\bar{A}_{\kappa_2}(\lm{Y}{j-2})
				\Lambda_{\star}^{1/2}}^2}{\mathcal{H}_{j-1}^n}.
	\end{align*}
	They verify
	\begin{align*}
		&\frac{1}{k_n^2}\sum_{j=2}^{k_n-2}\CE{\CE{s_{j,n}^{(1)}(\kappa_1)s_{j,n}^{(4)}(\kappa_2)}{\mathcal{H}_j^n}}{\mathcal{H}_{j-1}^n}=0
	\end{align*}
	and
	\begin{align*}
		\E{\abs{\frac{1}{k_n}\sum_{j=2}^{k_n-2}\CE{\abs{\CE{s_{j,n}^{(1)}(\kappa_1)s_{j,n}^{(4)}(\kappa_2)}{\mathcal{H}_j^n}}^2}{
					\mathcal{H}_{j-1}^n}}}
		\to 0.
	\end{align*}
	Therefore,
	\begin{align*}
		\frac{1}{k_n}\sum_{j=2}^{k_n-2}\CE{s_{j,n}^{(1)}(\kappa_1)s_{j,n}^{(4)}(\kappa_2)}{\mathcal{H}_j^n}=o_P(1)
	\end{align*}
	because of Lemma 9 in \citep{GeJ93}.
	Now we examine the covariance structure between $U_n^{(3)}(\kappa_1)$ and $U_n^{(4)}(\kappa_2)$. Again 
	\begin{align*}
		\CE{s_{j,n}^{(3)}(\kappa_1)s_{j,n}^{(4)}(\kappa_2)}{\mathcal{H}_j^n}
		=-\frac{4}{p_n\Delta_n^2}\parens{\lm{\epsilon}{j-1}}^T\Lambda_{\star}^{1/2}A_{\kappa_1}(\lm{Y}{j-2})\Lambda_{\star}\bar{A}_{\kappa_2}(\lm{Y}{j-2})a(X_{(j-1)\Delta_n})\parens{\zeta_{j,n}}.
	\end{align*}
	We also the conditional expectation with respect to $\mathcal{H}_{j-1}^n$. We have
	\begin{align*}
		\CE{\CE{s_{j,n}^{(3)}(\kappa_1)s_{j,n}^{(4)}(\kappa_2)}{\mathcal{H}_j^n}}{\mathcal{H}_{j-1}^n}
		=0
	\end{align*}
	and
	\begin{align*}
	\CE{\abs{\CE{s_{j,n}^{(3)}(\kappa_1)s_{j,n}^{(4)}(\kappa_2)}{\mathcal{H}_j^n}}^2}{\mathcal{H}_{j-1}^n}\le \frac{C}{p_n^3\Delta_n^3}\norm{\Lambda_{\star}^{1/2}A_{\kappa_1}(\lm{Y}{j-2})\Lambda_{\star}\bar{A}_{\kappa_2}(\lm{Y}{j-2})a(X_{(j-1)\Delta_n})}^2.
	\end{align*}
	Hence
	\begin{align*}
		\frac{1}{k_n}\sum_{j=2}^{k_n-2}\CE{\CE{s_{j,n}^{(3)}(\kappa_1)s_{j,n}^{(4)}(\kappa_2)}{\mathcal{H}_j^n}}{\mathcal{H}_{j-1}^n}=0
	\end{align*}
	and
	\begin{align*}
		\E{\abs{\frac{1}{k_n^2}\sum_{j=2}^{k_n-2}\CE{\abs{\CE{s_{j,n}^{(3)}(\kappa_1)s_{j,n}^{(4)}(\kappa_2)}{\mathcal{H}_j^n}}^2
				}{\mathcal{H}_{j-1}^n}}}\to 0.
	\end{align*}
	They lead to
	\begin{align*}
		\frac{1}{k_n}\sum_{j=2}^{k_n-2}\CE{s_{j,n}^{(3)}(\kappa_1)s_{j,n}^{(4)}(\kappa_2)}{\mathcal{H}_j^n}=o_P(1)
	\end{align*}
	by Lemma 9 in \citep{GeJ93}.\\

	\noindent \textbf{(Step 3): }We check the following decomposition
	\begin{align*}
		f_\lambda(\lm{Y}{j-1})\parens{\lm{Y}{j+1}-\lm{Y}{j}-\Delta_nb(\lm{Y}{j-1})}
		&=f_\lambda(\lm{Y}{j-1})\parens{\lm{Y}{j+1}-\lm{Y}{j}-\Delta_nb(X_{j\Delta_n})}\\
		&\qquad+f_\lambda(\lm{Y}{j-1})\Delta_n\parens{b(X_{j\Delta_n})-b(\lm{Y}{j-1})},
	\end{align*}
	and because of Corollary \ref*{cor738} we also have
	\begin{align*}
		f_\lambda(\lm{Y}{j-1})\parens{\lm{Y}{j+1}-\lm{Y}{j}-\Delta_nb(X_{j\Delta_n})}
		&=f_\lambda(\lm{Y}{j-1})a(X_{j\Delta_n})\parens{\zeta_{j+1,n}+\zeta_{j+2,n}'}\\
		&\qquad+f_\lambda(\lm{Y}{j-1})\Lambda_{\star}^{1/2}\parens{\lm{\epsilon}{j+1}-\lm{\epsilon}{j}}+f_\lambda(\lm{Y}{j-1})e_{j,n},
	\end{align*}
	and then
	\begin{align*}
		f_\lambda(\lm{Y}{j-1})\parens{\lm{Y}{j+1}-\lm{Y}{j}-\Delta_nb(\lm{Y}{j-1})}
		&=f_\lambda(\lm{Y}{j-1})a(X_{j\Delta_n})\parens{\zeta_{j+1,n}+\zeta_{j+2,n}'}\\
		&\qquad+f_\lambda(\lm{Y}{j-1})\Lambda_{\star}^{1/2}\parens{\lm{\epsilon}{j+1}-\lm{\epsilon}{j}}+f_\lambda(\lm{Y}{j-1})e_{j,n}\\
		&\qquad+f_\lambda(\lm{Y}{j-1})\Delta_n\parens{b(X_{j\Delta_n})-b(\lm{Y}{j-1})}.
	\end{align*}
	Here we can rewrite $\sqrt{k_n\Delta_n} \bar{D}_n\parens{f_\lambda(\cdot)}$ as
	\begin{align*}
		\sqrt{k_n\Delta_n} \bar{D}_n\parens{f_\lambda(\cdot)}=\bar{R}_n^{(1)}(\lambda)+\bar{R}_n^{(2)}(\lambda)+\bar{R}_n^{(3)}(\lambda)+\bar{R}_n^{(4)}(\lambda),
	\end{align*}
	where
	\begin{align*}
		\bar{R}_n^{(1)}(\lambda)&:=\frac{1}{\sqrt{k_n\Delta_n}}\sum_{j=1}^{k_n-2}f_\lambda(\lm{Y}{j-1})a(X_{j\Delta_n})\parens{\zeta_{j+1,n}+\zeta_{j+2,n}'},\\
		\bar{R}_n^{(2)}(\lambda)&:=\frac{1}{\sqrt{k_n\Delta_n}}\sum_{j=1}^{k_n-2}f_\lambda(\lm{Y}{j-1})\Lambda_{\star}^{1/2}\parens{\lm{\epsilon}{j+1}-\lm{\epsilon}{j}},\\
		\bar{R}_n^{(3)}(\lambda)&:=\frac{1}{\sqrt{k_n\Delta_n}}\sum_{j=1}^{k_n-2}f_\lambda(\lm{Y}{j-1})e_{j,n},\\
		\bar{R}_n^{(4)}(\lambda)&:=\frac{1}{\sqrt{k_n\Delta_n}}\sum_{j=1}^{k_n-2}f_\lambda(\lm{Y}{j-1})
		\Delta_n\parens{b(X_{j\Delta_n})-b(\lm{Y}{j-1})}.
	\end{align*}
	Hence it is enough to see asymptotic behaviour of $\bar{R}$'s and firstly we examine that of $R_{n}^{(1)}$. We define the $\mathcal{H}_{j+1}^n$-measurable random variable 
	\begin{align*}
		r_{j,n}^{(1)}(\lambda)&:=\frac{1}{\sqrt{k_n\Delta_n}}f_\lambda(\lm{Y}{j-1})a(X_{j\Delta_n})\zeta_{j+1,n}+\frac{1}{\sqrt{k_n\Delta_n}}f_\lambda(\lm{Y}{j-2})a(X_{(j-1)\Delta_n})\zeta_{j+1,n}'
	\end{align*}
	and then
	\begin{align*}
		\bar{R}_n^{(1)}(\lambda)&=\frac{1}{\sqrt{k_n\Delta_n}}\sum_{j=1}^{k_n-2}f_\lambda(\lm{Y}{j-1})a(X_{j\Delta_n})\parens{\zeta_{j+1,n}+\zeta_{j+2,n}'}\\
		&=\sum_{j=2}^{k_n-2}r_{j,n}^{(1)}(\lambda)+o_P(1).
	\end{align*}
	Obviously
	\begin{align*}
		\CE{r_{j,n}^{(1)}(\lambda)}{\mathcal{H}_j^n}=0.
	\end{align*}
	With respect to the second moment, for all $\lambda_1,\lambda_2\in\{1,\cdots,m_2\}$, by Lemma \ref*{lem732}, 
	\begin{align*}
		&\CE{\parens{r_{j,n}^{(1)}(\lambda_1)}\parens{r_{j,n}^{(1)}(\lambda_2)}}{\mathcal{H}_j^n}\\
		&=\frac{1}{k_n\Delta_n}\Delta_n\parens{\frac{1}{3}+\frac{1}{2p_n}+\frac{1}{6p_n^2}}f_{\lambda_1}(\lm{Y}{j-1})c(X_{j\Delta_n})\parens{f_{\lambda_2}(\lm{Y}{j-1})}^T\\
		&\qquad+\frac{1}{k_n\Delta_n}\Delta_n\parens{\frac{1}{3}-\frac{1}{2p_n}+\frac{1}{6p_n^2}}f_{\lambda_1}(\lm{Y}{j-2})c(X_{(j-1)\Delta_n})\parens{f_{\lambda_2}(\lm{Y}{j-2})}^T\\
		&\qquad+\frac{2}{k_n\Delta_n}\frac{\Delta_n}{6}\parens{1-\frac{1}{p_n^2}}f_{\lambda_1}(\lm{Y}{j-1})a(X_{j\Delta_n})\parens{a(X_{(j-1)\Delta_n})}^T\parens{f_{\lambda_2}(\lm{Y}{j-2})}^T.
	\end{align*}
	Therefore,
	\begin{align*}
		\sum_{j=2}^{k_n-2}\CE{\parens{r_{j,n}^{(1)}(\lambda_1)}\parens{r_{j,n}^{(1)}(\lambda_2)}}{\mathcal{H}_j^n}
		\cp\nu_0\parens{\parens{f_{\lambda_1}}\parens{c}\parens{f_{\lambda_2}}^T(\cdot)}
	\end{align*}
	because of Lemma \ref*{lem721} and Lemma \ref*{lem739}. Lindeberg condition can be satisfied since
	\begin{align*}
		\CE{\parens{r_{j,n}^{(1)}(\lambda)}^4}{\mathcal{H}_j^n}
		&\le \frac{C}{k_n^2}\norm{f_\lambda(\lm{Y}{j-1})a(X_{j\Delta_n})}^4+\frac{C}{k_n^2}\norm{f_\lambda(\lm{Y}{j-2})a(X_{(j-1)\Delta_n})}
	\end{align*}
	and
	\begin{align*}
		\E{\abs{\sum_{j=2}^{k_n-2}\CE{\parens{r_{j,n}^{(1)}(\lambda)}^4}{\mathcal{H}_j^n}}}\to0
	\end{align*}
	by Lemma \ref*{lem739}.
	Now we show $\bar{R}^{(2)}$ is $o_P(1)$. Let us define the following $\mathcal{H}_{j+1}^{n}$-measurable random variable
	\begin{align*}
		r_{j,n}^{(2)}(\lambda)
		&:=\frac{1}{\sqrt{k_n\Delta_n}}\parens{f_{\lambda}(\lm{Y}{j-2})-f_{\lambda}(\lm{Y}{j-1})}\Lambda_{\star}^{1/2}\lm{\epsilon}{j},
	\end{align*}
	and then we have
	\begin{align*}
		\bar{R}_{n}^{(2)}(\lambda)=\frac{1}{\sqrt{k_n\Delta_n}}\sum_{j=1}^{k_n-2}f_\lambda(\lm{Y}{j-1})\Lambda_{\star}^{1/2}\parens{\lm{\epsilon}{j+1}-\lm{\epsilon}{j}}=\sum_{j=2}^{k_n-2}r_{j,n}^{(2)}(\lambda)+o_P(1).
	\end{align*}
	We prove $\bar{R}^{(2)}=o_P(1)$ with Lemma 9 in \citep{GeJ93}. It is obvious
	\begin{align*}
		\CE{r_{j,n}^{(2)}(\lambda)}{\mathcal{H}_j^n}=0.
	\end{align*}
	For the second moment, with Lemma \ref*{lem739},
	\begin{align*}
		\CE{\parens{r_{j,n}^{(2)}(\lambda)}^2}{\mathcal{H}_j^n}
		\le \frac{C}{k_n}\norm{f_{\lambda}(\lm{Y}{j-2})-f_{\lambda}(\lm{Y}{j-1})}^2
	\end{align*}
	and therefore, by Lemma \ref*{lem739}, 
	\begin{align*}
		\E{\abs{\sum_{j=2}^{k_n-2}\CE{\parens{r_{j,n}^{(2)}(\lambda)}^2}{\mathcal{H}_j^n}}}\to 0.
	\end{align*}
	Hence $\bar{R}_{n}^{(2)}(\lambda)=o_P(1)$. With respect to $\bar{R}_{n}^{(3)}(\lambda)$, we again use Lemma 9 in \citep{GeJ93} to show convergence to zero in probability. To ease notation, we separate the summation into three parts as same as Theorem \ref*{thm742} and Theorem \ref*{thm743} such that
	\begin{align*}
		\bar{R}_{n}^{(3)}(\lambda)&=\bar{R}_{0,n}^{(3)}(\lambda)+\bar{R}_{1,n}^{(3)}(\lambda)+\bar{R}_{2,n}^{(3)}(\lambda),
	\end{align*}
	where for $l=0,1,2$, 
	\begin{align*}
		\bar{R}_{l,n}^{(3)}(\lambda)=\frac{1}{\sqrt{k_n\Delta_n}}\sum_{1\le 3j+l\le k_n-2}f_{\lambda}(\lm{Y}{3j+l-1})e_{3j+l,n},
	\end{align*}
	and it is enough to examine if $\bar{R}_{0,n}^{(3)}(\lambda)=o_P(1)$. Let $r_{3j,n}^{(3)}(\lambda)$ be a random variable defined as
	\begin{align*}
		r_{3j,n}^{(3)}(\lambda)&:=\frac{1}{\sqrt{k_n\Delta_n}}f_{\lambda}(\lm{Y}{3j-1})e_{3j,n}
	\end{align*}
	and then $r_{3j,n}^{(3)}(\lambda)$ is $\mathcal{H}_{3j+2}^{n}$-measurable and $\mathcal{H}_{3(j+1)}^{n}$-measurable. Furthermore,
	\begin{align*}
		\bar{R}_{0,n}^{(3)}(\lambda)=\sum_{1\le 3j\le k_n-2}r_{3j,n}^{(3)}(\lambda).
	\end{align*}
	Therefore, the conditional expectation with respect to $\mathcal{H}_{3j}^{n}$ can be evaluated as
	\begin{align*}
		\CE{r_{3j,n}^{(3)}(\lambda)}{\mathcal{H}_{3j}^n}
		=\frac{1}{\sqrt{k_n\Delta_n}}f_{\lambda}(\lm{Y}{3j-1})\CE{e_{3j,n}}{\mathcal{H}_{3j}^n}
	\end{align*}
	and hence with Proposition \ref*{pro737} and Lemma \ref*{lem739}, 
	\begin{align*}
		\E{\abs{\sum_{1\le 3j\le k_n-2}\CE{r_{3j,n}^{(3)}(\lambda)}{\mathcal{H}_{3j}^n}}}
		\to 0.
	\end{align*}
	With respect to the second moment,
	\begin{align*}
		\CE{\abs{r_{3j,n}^{(3)}(\lambda)}^2}{\mathcal{H}_{3j}^n}
		\le \frac{1}{k_n\Delta_n}\norm{f_{\lambda}(\lm{Y}{3j-1})}^2C\Delta_n^2\parens{1+\norm{X_{3j\Delta_n}}^6}
	\end{align*}
	and hence Lemma \ref*{lem739} leads to
	\begin{align*}
		\E{\abs{\sum_{1\le 3j\le k_n-2}\CE{\abs{r_{3j,n}^{(3)}(\lambda)}^2}{\mathcal{H}_{3j}^n}}}\to0.
	\end{align*}
	As a result we obtained $\bar{R}_{n}^{(3)}(\lambda)=o_P(1)$. We can evaluated $L^1$ norm of $\bar{R}_{n}^{(4)}(\lambda)$ such that
	\begin{align*}
		\E{\abs{\bar{R}_{n}^{(4)}(\lambda)}}
			&\to0
	\end{align*}
	and it verifies $\bar{R}_n^{(4)}(\lambda)=o_P(1)$.\\
	
	\noindent\textbf{(Step 4): } We check the covariance structures among $\sqrt{n}D_{n}$, $U_n^{(1)}$, $U_n^{(3)}$, $U_n^{(4)}$ $\bar{R}_n^{(1)}$ which have not been shown. It is easy to see
	\begin{align*}
		\CE{\parens{D_{j,n}'}^{l_1,l_2}s_{j,n}^{(1)}(\kappa)}{\mathcal{H}_j^n}&=0,\\
		\CE{\parens{D_{j,n}'}^{l_1,l_2}r_{j,n}^{(1)}(\lambda)}{\mathcal{H}_j^n}&=0.
	\end{align*}
	We also have
	\begin{align*}
		&\frac{\sqrt{n}}{k_n^{3/2}}\sum_{j=1}^{k_n-2}\CE{\parens{D_{j,n}'}^{l_1,l_2}s_{j,n}^{(3)}(\kappa)}{\mathcal{H}_j^n}\\
		&=\frac{1}{\sqrt{p_n}k_n}\sum_{j=1}^{k_n-2}\ip{\Lambda_{\star}^{1/2}\parens{A_{\kappa}(\lm{Y}{j-2})+A_{\kappa}(\lm{Y}{j-1})}\Lambda_{\star}^{1/2}}{\frac{1}{\Delta_n}\CE{\parens{D_{j,n}'}^{l_1,l_2}
				\parens{\parens{\lm{\epsilon}{j}}^{\otimes 2}}}{\mathcal{H}_j^n}}\\
		&\qquad+\frac{1}{\sqrt{p_n}k_n}\sum_{j=1}^{k_n-2}\ip{\Lambda_{\star}^{1/2}A_{\kappa}(\lm{Y}{j-2})\Lambda_{\star}^{1/2}}{
			\parens{-\frac{2}{\Delta_n}\parens{\lm{\epsilon}{j-1}}}\CE{\parens{D_{j,n}'}^{l_1,l_2}
				\parens{\lm{\epsilon}{j}}^T}{\mathcal{H}_j^n}}.
	\end{align*}
	Because
	\begin{align*}
		&\CE{\parens{D_{j,n}'}^{l_1,l_2}
			\parens{\parens{\lm{\epsilon}{j}}^{\otimes 2}}}{\mathcal{H}_j^n}\\
		&=\frac{1}{2p_n^3}\sum_{i_1,i_2}\E{\parens{\Lambda_{\star}^{1/2}\parens{2\parens{\epsilon_{j\Delta_n+i_1h_n}}^{\otimes 2}}\Lambda_{\star}^{1/2}}^{l_1,l_2}\parens{\epsilon_{j\Delta_n+i_2h_n}}^{\otimes 2}}\\
		&\qquad+\frac{1}{2p_n^3}\sum_{i_1}\E{\parens{\Lambda_{\star}^{1/2}\parens{\parens{\epsilon_{j\Delta_n+i_1h_n}}\parens{\epsilon_{j\Delta_n+(i_1-1)h_n}}^T}\Lambda_{\star}^{1/2}}^{l_1,l_2}\parens{\epsilon_{j\Delta_n+i_1h_n}}\parens{\epsilon_{j\Delta_n+(i_1-1)h_n}}^T}\\
		&\qquad+\frac{1}{2p_n^3}\sum_{i_1}\E{\parens{\Lambda_{\star}^{1/2}\parens{\parens{\epsilon_{j\Delta_n+(i_1-1)h_n}}\parens{\epsilon_{j\Delta_n+i_1h_n}}^T}\Lambda_{\star}^{1/2}}^{l_1,l_2}\parens{\epsilon_{j\Delta_n+i_1h_n}}\parens{\epsilon_{j\Delta_n+(i_1-1)h_n}}^T}\\
		&\qquad-\frac{1}{p_n}\parens{\Lambda_{\star}}^{l_1,l_2}I_d
	\end{align*}
	and
	\begin{align*}
		\CE{\parens{D_{j,n}'}^{l_1,l_2}\parens{\lm{\epsilon}{j}}^T}{\mathcal{H}_j^n}
		=0,
	\end{align*}
	we obtain
	\begin{align*}
		\E{\abs{\frac{1}{\sqrt{p_n}k_n}\sum_{j=1}^{k_n-2}\ip{\Lambda_{\star}^{1/2}\parens{A_{\kappa}(\lm{Y}{j-2})+A_{\kappa}(\lm{Y}{j-1})}\Lambda_{\star}^{1/2}}{\frac{1}{\Delta_n}\CE{\parens{D_{j,n}'}^{l_1,l_2}
						\parens{\parens{\lm{\epsilon}{j}}^{\otimes 2}}}{\mathcal{H}_j^n}}}}\to 0
	\end{align*}
	and
	\begin{align*}
		\frac{1}{\sqrt{p_n}k_n}\sum_{j=1}^{k_n-2}\ip{\Lambda_{\star}^{1/2}A_{\kappa}(\lm{Y}{j-2})\Lambda_{\star}^{1/2}}{
			\parens{-\frac{2}{\Delta_n}\parens{\lm{\epsilon}{j-1}}}\CE{\parens{D_{j,n}'}^{l_1,l_2}
				\parens{\lm{\epsilon}{j}}^T}{\mathcal{H}_j^n}}=0.
	\end{align*}
	Hence $U_n^{(3)}$ and $\sqrt{n}D_n$ are asymptotically independent. Furthermore, using the evaluation
	\begin{align*}
		\CE{\parens{D_{j,n}'}^{l_1,l_2}\parens{\lm{\epsilon}{j}}^T}{\mathcal{H}_j^n}=0
	\end{align*} we have
	\begin{align*}
		\CE{\parens{D_{j,n}'}^{l_1,l_2}s_{j,n}^{(4)}(\kappa)}{\mathcal{H}_j^n}=0.
	\end{align*}
	Next we evaluate the asymptotics of
	\begin{align*}
		\frac{1}{\sqrt{k_n}}\sum_{j=1}^{k_n-2}\CE{s_{j,n}^{(1)}(\kappa)r_{j,n}^{(1)}(\lambda)}{\mathcal{H}_j^n}.
	\end{align*}
	We have
	\begin{align*}
		&\CE{s_{j,n}^{(1)}(\kappa)r_{j,n}^{(1)}(\lambda)}{\mathcal{H}_j^n}\\
		&=\frac{2\chi_n}{\sqrt{k_n\Delta_n}}\parens{\zeta_{j,n}}^Ta(X_{(j-1)\Delta_n})^T\bar{A}_{\kappa}(\lm{Y}{j-2})a(X_{(j-1)\Delta_n})
			a(X_{j\Delta_n})^Tf_\lambda(\lm{Y}{j-1})^T\\
		&\qquad+\frac{2m_n'}{\sqrt{k_n\Delta_n}}\parens{\zeta_{j,n}}^Ta(X_{(j-1)\Delta_n})^T\bar{A}_{\kappa}(\lm{Y}{j-2})c(X_{(j-1)\Delta_n})f_\lambda(\lm{Y}{j-2})^T,
	\end{align*}
	and then
	\begin{align*}
		&\frac{1}{\sqrt{k_n}}\sum_{j=1}^{k_n-2}\CE{s_{j,n}^{(1)}(\kappa)r_{j,n}^{(1)}(\lambda)}{\mathcal{H}_j^n}\\
		&=\frac{2\chi_n}{k_n\sqrt{\Delta_n}}\sum_{j=1}^{k_n-2}\parens{\zeta_{j,n}}^Ta(X_{(j-1)\Delta_n})^T\bar{A}_{\kappa}(\lm{Y}{j-2})a(X_{(j-1)\Delta_n})
		a(X_{j\Delta_n})^Tf_\lambda(\lm{Y}{j-1})^T\\
		&\qquad+\frac{2m_n'}{k_n\sqrt{\Delta_n}}\sum_{j=1}^{k_n-2}\parens{\zeta_{j,n}}^Ta(X_{(j-1)\Delta_n})^T\bar{A}_{\kappa}(\lm{Y}{j-2})c(X_{(j-1)\Delta_n})f_\lambda(\lm{Y}{j-2})^T.
	\end{align*}
	Note the $L^1$ convergence
	\begin{align*}
		&\mathbf{E}\lcrotchet{\labs{\frac{2\chi_n}{k_n\sqrt{\Delta_n}}\sum_{j=1}^{k_n-2}\parens{\zeta_{j,n}}^Ta(X_{(j-1)\Delta_n})^T\bar{A}_{\kappa}(\lm{Y}{j-2})a(X_{(j-1)\Delta_n})
		a(X_{j\Delta_n})^Tf_\lambda(\lm{Y}{j-1})^T}}\\
		&\qquad\rcrotchet{\rabs{-\frac{2\chi_n}{k_n\sqrt{\Delta_n}}\sum_{j=1}^{k_n-2}\parens{\zeta_{j,n}}^Ta(X_{(j-1)\Delta_n})^T\bar{A}_{\kappa}(\lm{Y}{j-2})a(X_{(j-1)\Delta_n})
				a(X_{(j-1)\Delta_n})^Tf_\lambda(\lm{Y}{j-2})^T}}\\
		&\to 0.
	\end{align*}
	Hence we obtain
	\begin{align*}
		&\frac{1}{\sqrt{k_n}}\sum_{j=1}^{k_n-2}\CE{s_{j,n}^{(1)}(\kappa)r_{j,n}^{(1)}(\lambda)}{\mathcal{H}_j^n}\\
		&=\frac{2\parens{\chi_n+m_n'}}{k_n\sqrt{\Delta_n}}\sum_{j=1}^{k_n-2}\parens{\zeta_{j,n}}^Ta(X_{(j-1)\Delta_n})^T\bar{A}_{\kappa}(\lm{Y}{j-2})c(X_{(j-1)\Delta_n})f_\lambda(\lm{Y}{j-2})^T\\
		&\qquad+o_P(1).
	\end{align*}
	We are also able to evaluate
	\begin{align*}
		&\E{\abs{\frac{2\parens{\chi_n+m_n'}}{k_n\sqrt{\Delta_n}}\sum_{j=1}^{k_n-2}\CE{\parens{\zeta_{j,n}}^Ta(X_{(j-1)\Delta_n})^T\bar{A}_{\kappa}(\lm{Y}{j-2})c(X_{(j-1)\Delta_n})f_\lambda(\lm{Y}{j-2})^T}{\mathcal{H}_{j-1}^n}}}\\
		&=0
	\end{align*}
	and
	\begin{align*}
		&\E{\abs{\frac{4\parens{\chi_n+m_n'}^2}{k_n^2\Delta_n}\sum_{j=1}^{k_n-2}\CE{\abs{\parens{\zeta_{j,n}}^Ta(X_{(j-1)\Delta_n})^T\bar{A}_{\kappa}(\lm{Y}{j-2})c(X_{(j-1)\Delta_n})f_\lambda(\lm{Y}{j-2})^T}^2}{\mathcal{H}_{j-1}^n}}}\\
		&\to 0.
	\end{align*}
	Hence Lemma 9 in \citep{GeJ93} leads to
	\begin{align*}
		\frac{1}{\sqrt{k_n}}\sum_{j=1}^{k_n-2}\CE{s_{j,n}^{(1)}(\kappa)r_{j,n}^{(1)}(\lambda)}{\mathcal{H}_j^n}\cp 0.
	\end{align*}
	Obviously we obtain
	\begin{align*}
		\CE{s_{j,n}^{(3)}(\kappa)r_{j,n}^{(1)}(\lambda)}{\mathcal{H}_j^n}=0.
	\end{align*}
	Finally, we examine the asymptotics of
	\begin{align*}
		\frac{1}{\sqrt{k_n}}\sum_{j=1}^{k_n-2}\CE{s_{j,n}^{(4)}(\kappa)r_{j,n}^{(1)}(\lambda)}{\mathcal{H}_j^n}.
	\end{align*}
	We have
	\begin{align*}
		&\CE{s_{j,n}^{(4)}(\kappa)r_{j,n}^{(1)}(\lambda)}{\mathcal{H}_j^n}\\
		&=-\frac{2\chi_n}{\sqrt{k_n\Delta_n}}\parens{\lm{\epsilon}{j-1}}^T\Lambda_{\star}^{1/2}\bar{A}_{\kappa}(\lm{Y}{j-2})a(X_{(j-1)\Delta_n})a(X_{j\Delta_n})^Tf_\lambda(\lm{Y}{j-1})^T\\
		&\qquad-\frac{2m_n'}{\sqrt{k_n\Delta_n}}\parens{\lm{\epsilon}{j-1}}^T\Lambda_{\star}^{1/2}\bar{A}_{\kappa}(\lm{Y}{j-2})c(X_{(j-1)\Delta_n})f_\lambda(\lm{Y}{j-2})^T
	\end{align*}
	and the $L^1$ convergence
	\begin{align*}
		&\mathbf{E}\lcrotchet{\labs{\frac{2\chi_n}{k_n\sqrt{\Delta_n}}\sum_{j=1}^{k_n-2}\parens{\lm{\epsilon}{j-1}}^T\Lambda_{\star}^{1/2}\bar{A}_{\kappa}(\lm{Y}{j-2})a(X_{(j-1)\Delta_n})a(X_{j\Delta_n})^Tf_\lambda(\lm{Y}{j-1})^T}}\\
		&\qquad\rcrotchet{\rabs{-\frac{2\chi_n}{k_n\sqrt{\Delta_n}}\sum_{j=1}^{k_n-2}\parens{\lm{\epsilon}{j-1}}^T\Lambda_{\star}^{1/2}\bar{A}_{\kappa}(\lm{Y}{j-2})a(X_{(j-1)\Delta_n})a(X_{(j-1)\Delta_n})^Tf_\lambda(\lm{Y}{j-2})^T}}\\
		&\to 0.
	\end{align*}
	Hence
	\begin{align*}
		&\frac{1}{\sqrt{k_n}}\CE{s_{j,n}^{(4)}(\kappa)r_{j,n}^{(1)}(\lambda)}{\mathcal{H}_j^n}\\
		&=-\frac{2\parens{\chi_n+m_n'}}{k_n\sqrt{\Delta_n}}\sum_{j=1}^{k_n-2}\parens{\lm{\epsilon}{j-1}}^T\Lambda_{\star}^{1/2}\bar{A}_{\kappa}(\lm{Y}{j-2})c(X_{(j-1)\Delta_n})f_\lambda(\lm{Y}{j-2})^T\\
		&\qquad+o_P(1).
	\end{align*}
	Again we can evaluate
	\begin{align*}
		\E{\abs{\frac{2\parens{\chi_n+m_n'}}{k_n\sqrt{\Delta_n}}\sum_{j=1}^{k_n-2}\CE{\parens{\lm{\epsilon}{j-1}}^T\Lambda_{\star}^{1/2}\bar{A}_{\kappa}(\lm{Y}{j-2})c(X_{(j-1)\Delta_n})f_\lambda(\lm{Y}{j-2})^T}{\mathcal{H}_{j-1}^n}}}
		=0
	\end{align*}
	and
	\begin{align*}
		\E{\abs{\frac{4\parens{\chi_n+m_n'}^2}{k_n^2\Delta_n}\sum_{j=1}^{k_n-2}\CE{\abs{\parens{\lm{\epsilon}{j-1}}^T\Lambda_{\star}^{1/2}\bar{A}_{\kappa}(\lm{Y}{j-2})c(X_{(j-1)\Delta_n})f_\lambda(\lm{Y}{j-2})^T}^2}{\mathcal{H}_{j-1}^n}}}\to 0.
	\end{align*}
	Therefore,
	\begin{align*}
		\frac{1}{\sqrt{k_n}}\CE{s_{j,n}^{(4)}(\kappa)r_{j,n}^{(1)}(\lambda)}{\mathcal{H}_j^n}\cp 0.
	\end{align*}
	Then we obtain the proof.
\end{proof}

\begin{corollary}\label{cor752}
	With the same assumption as Theorem \ref*{thm751}, we have
	\begin{align*}
	\crotchet{\begin{matrix}
		\sqrt{n}D_n\\
		\sqrt{k_n}\crotchet{\bar{Q}_n\parens{A_{\kappa,n}(\cdot)}
			-\frac{2}{3}\bar{M}_n\parens{\ip{A_{\kappa,n}(\cdot)}{c_n^{\tau}\parens{\cdot,\alpha^{\star},\Lambda_{\star}}}}}_{\kappa}\\
		\sqrt{k_n\Delta_n}\crotchet{\bar{D}_n\parens{f_{\lambda,n}(\cdot)}}_\lambda
		\end{matrix}}\cl N(\mathbf{0},W(\tuborg{A_{\kappa}},\tuborg{f_{\lambda}})).
	\end{align*}
\end{corollary}

\begin{proof}
	It is enough to check
	\begin{align*}
		&\crotchet{\begin{matrix}
			\sqrt{k_n}\crotchet{\bar{Q}_n\parens{\parens{A_{\kappa,n}-A_{\kappa}}(\cdot)}
				-\frac{2}{3}\bar{M}_n\parens{\ip{\parens{A_{\kappa,n}-A_{\kappa}}(\cdot)}{c_n^{\tau}\parens{\cdot,\alpha^{\star},\Lambda_{\star}}}}}_{\kappa}\\
			\sqrt{k_n\Delta_n}\crotchet{\bar{D}_n\parens{\parens{f_{\lambda,n}-f_{\lambda}}(\cdot)}}_\lambda
			\end{matrix}}=o_P(1).
	\end{align*}
	We prepare the following notation such that
	\begin{align*}
		u_{j,n}^{(1)}((\kappa,n),\kappa)
		&:=\ip{\parens{A_{\kappa,n}-A_{\kappa}}(\lm{Y}{j-1})}{a(X_{j\Delta_n})\parens{\frac{1}{\Delta_n}\parens{\zeta_{j+1,n}+\zeta_{j+2,n}'}^{\otimes 2}
				-\frac{2}{3}I_r}a(X_{j\Delta_n})^{T}},\\
		u_{j,n}^{(2)}((\kappa,n),\kappa)
		&:=\frac{2}{3}\ip{\parens{A_{\kappa,n}-A_{\kappa}}(\lm{Y}{j-1})}{c(X_{j\Delta_n})-c(\lm{Y}{j-1})},\\
		u_{j,n}^{(3)}((\kappa,n),\kappa)
		&:=\ip{\parens{A_{\kappa,n}-A_{\kappa}}(\lm{Y}{j-1})}{\Lambda_{\star}^{1/2}
			\parens{\frac{1}{\Delta_n}\parens{\lm{\epsilon}{j+1}-\lm{\epsilon}{j}}^{\otimes 2}
				-2\Delta_n^{\frac{2-\tau}{\tau-1}}I_d}\Lambda_{\star}^{1/2}},\\
		u_{j,n}^{(4)}((\kappa,n),\kappa)
		&:=\frac{2}{\Delta_n}\ip{\parens{\bar{A}_{\kappa,n}-\bar{A}_{\kappa}}(\lm{Y}{j-1})}{
			\parens{a(X_{j\Delta_n})\parens{\zeta_{j+1,n}+\zeta_{j+2,n}'}}
			\parens{\Lambda_{\star}^{1/2}\parens{\lm{\epsilon}{j+1}-\lm{\epsilon}{j}}}^T},\\
		u_{j,n}^{(5)}((\kappa,n),\kappa)
		&:=\frac{1}{\Delta_n}\ip{\parens{\bar{A}_{\kappa,n}-\bar{A}_{\kappa}}(\lm{Y}{j-1})}{\parens{\Delta_nb(X_{j\Delta_n})+e_{j,n}}^{\otimes 2}},\\
		u_{j,n}^{(6)}((\kappa,n),\kappa)
		&:=\frac{2}{\Delta_n}\ip{\parens{\bar{A}_{\kappa,n}-\bar{A}_{\kappa}}(\lm{Y}{j-1})}{
			\parens{\Delta_nb(X_{j\Delta_n})+e_{j,n}}\parens{a(X_{j\Delta_n})\parens{\zeta_{j+1,n}+\zeta_{j+2,n}'}}^T},\\
		u_{j,n}^{(7)}((\kappa,n),\kappa)
		&:=\frac{2}{\Delta_n}
		\ip{\parens{\bar{A}_{\kappa,n}-\bar{A}_{\kappa}}(\lm{Y}{j-1})}{
			\parens{\Delta_nb(X_{j\Delta_n})+e_{j,n}}
			\parens{\Lambda_{\star}^{1/2}\parens{\lm{\epsilon}{j+1}-\lm{\epsilon}{j}}}^T},\\
		s_{j,n}^{(1)}((\kappa,n),\kappa)&:=\ip{\parens{A_{\kappa,n}-A_{\kappa}}(\lm{Y}{j-1})}{
			a(X_{j\Delta_n})\parens{\frac{1}{\Delta_n}\parens{\zeta_{j+1,n}}^{\otimes 2}-m_nI_r}a(X_{j\Delta_n})^{T}}\\
		&\qquad+\ip{\parens{A_{\kappa,n}-A_{\kappa}}(\lm{Y}{j-2})}{
		a(X_{(j-1)\Delta_n})\parens{\frac{1}{\Delta_n}\parens{\zeta_{j+1,n}'}^{\otimes 2}-m_n'I_r}a(X_{(j-1)\Delta_n})^{T}}\\
		&\qquad+2\ip{\parens{\bar{A}_{\kappa,n}-\bar{A}_{\kappa}}(\lm{Y}{j-2})}{a(X_{(j-1)\Delta_n})
		\parens{\frac{1}{\Delta_n}\parens{\zeta_{j,n}\parens{\zeta_{j+1,n}'}^T}}a(X_{(j-1)\Delta_n})^{T}},\\
		s_{j,n}^{(3)}((\kappa,n),\kappa)&:=\ip{\parens{A_{\kappa,n}-A_{\kappa}}(\lm{Y}{j-2})}{\Lambda_{\star}^{1/2}
		\parens{\frac{1}{\Delta_n}\parens{\lm{\epsilon}{j}}^{\otimes 2}-\Delta_n^{\frac{2-\tau}{\tau-1}}I_d}\Lambda_{\star}^{1/2}}\\
		&\qquad+\ip{\parens{A_{\kappa,n}-A_{\kappa}}(\lm{Y}{j-1})}{\Lambda_{\star}^{1/2}
			\parens{\frac{1}{\Delta_n}\parens{\lm{\epsilon}{j}}^{\otimes 2}-\Delta_n^{\frac{2-\tau}{\tau-1}}I_d}\Lambda_{\star}^{1/2}}\\
		&\qquad+\ip{\parens{A_{\kappa,n}-A_{\kappa}}(\lm{Y}{j-2})}{\Lambda_{\star}^{1/2}
			\parens{-\frac{2}{\Delta_n}\parens{\lm{\epsilon}{j-1}}\parens{\lm{\epsilon}{j}}^T}\Lambda_{\star}^{1/2}},\\
		s_{j,n}^{(4)}((\kappa,n),\kappa)&:=\frac{2}{\Delta_n}\ip{\parens{\bar{A}_{\kappa,n}-\bar{A}_{\kappa}}(\lm{Y}{j-2})}{
			a(X_{(j-1)\Delta_n})\parens{\zeta_{j,n}}
			\parens{\lm{\epsilon}{j}}^T\Lambda_{\star}^{1/2}}\\
		&\qquad-\frac{2}{\Delta_n}\ip{\parens{\bar{A}_{\kappa,n}-\bar{A}_{\kappa}}(\lm{Y}{j-1})}{
			a(X_{j\Delta_n})\parens{\zeta_{j+1,n}}
			\parens{\lm{\epsilon}{j}}^T\Lambda_{\star}^{1/2}}\\
		&\qquad+\frac{2}{\Delta_n}\ip{\parens{\bar{A}_{\kappa,n}-\bar{A}_{\kappa}}(\lm{Y}{j-2})}{
			a(X_{(j-1)\Delta_n})\parens{\zeta_{j+1,n}'}
			\parens{\lm{\epsilon}{j}}^T\Lambda_{\star}^{1/2}}\\
		&\qquad-\frac{2}{\Delta_n}\ip{\parens{\bar{A}_{\kappa,n}-\bar{A}_{\kappa}}(\lm{Y}{j-2})}{
			a(X_{(j-1)\Delta_n})\parens{\zeta_{j+1,n}'}
			\parens{\lm{\epsilon}{j-1}}^T\Lambda_{\star}^{1/2}},\\
		r_{j,n}^{(1)}((\lambda,n),\lambda)&:=\frac{1}{\sqrt{k_n\Delta_n}}\parens{f_{\lambda,n}-f_{\lambda}}(\lm{Y}{j-1})a(X_{j\Delta_n})\zeta_{j+1,n}\\
		&\qquad+\frac{1}{\sqrt{k_n\Delta_n}}\parens{f_{\lambda,n}-f_{\lambda}}(\lm{Y}{j-2})a(X_{(j-1)\Delta_n})\zeta_{j+1,n}',
	\end{align*}
	and
	\begin{align*}
	U_n^{(l)}((\kappa,n),\kappa)&:=\frac{1}{\sqrt{k_n}}\sum_{j=1}^{k_n-2}u_{j,n}^{(l)}((\kappa,n),\kappa),\\
	\bar{R}_n^{(1)}((\lambda,n),\lambda)&:=\frac{1}{\sqrt{k_n\Delta_n}}\sum_{j=1}^{k_n-2}\parens{f_{\lambda,n}-f_{\lambda}}(\lm{Y}{j-1})a(X_{j\Delta_n})\parens{\zeta_{j+1,n}+\zeta_{j+2,n}'},\\
	\bar{R}_n^{(2)}((\lambda,n),\lambda)&:=\frac{1}{\sqrt{k_n\Delta_n}}\sum_{j=1}^{k_n-2}\parens{f_{\lambda,n}-f_{\lambda}}(\lm{Y}{j-1})\Lambda_{\star}^{1/2}\parens{\lm{\epsilon}{j+1}-\lm{\epsilon}{j}},\\
	\bar{R}_n^{(3)}((\lambda,n),\lambda)&:=\frac{1}{\sqrt{k_n\Delta_n}}\sum_{j=1}^{k_n-2}\parens{f_{\lambda,n}-f_{\lambda}}(\lm{Y}{j-1})e_{j,n},\\
	\bar{R}_n^{(4)}((\lambda,n),\lambda)&:=\frac{1}{\sqrt{k_n\Delta_n}}\sum_{j=1}^{k_n-2}\parens{f_{\lambda,n}-f_{\lambda}}(\lm{Y}{j-1})
	\Delta_n\parens{b(X_{j\Delta_n})-b(\lm{Y}{j-1})}.
	\end{align*} We can obtain
	\begin{align*}
		U_n^{(l_1)}((\kappa,n),\kappa)&=o_P(1)\ l_1=2,5,6,7,\\
		\bar{R}_n^{(l_2)}((\lambda,n),\lambda)&=o_P(1),\ l_2=2,3,4
	\end{align*}
	with the same way as proof of Theorem \ref*{thm751}.
	Thus it is sufficient to check the asymptotics of $U_{n}^{(1)}((\kappa,n),\kappa)$, $U_{n}^{(3)}((\kappa,n),\kappa)$, $U_{n}^{(4)}((\kappa,n),\kappa)$, $\bar{R}_{n}^{(1)}((\kappa,n),\kappa)$. Using the evaluation of the conditional first moments for $s_{j,n}^{(l_1)}((\kappa,n),\kappa)$ for $l_1=1,3,4$ and $r_{j,n}^{(1)}((\lambda,n),\lambda)$, it is enough to check the conditional second moments of them. We obtain
	\begin{align*}
		\E{\abs{\frac{1}{k_n}\sum_{j=1}^{k_n-2}\CE{\abs{s_{j,n}^{(1)}((\kappa,n),\kappa)}^2}{\mathcal{H}_j^n}}}
		\to 0.
	\end{align*}
	Hence $U_n^{(1)}((\kappa,n),\kappa)=o_P(1)$ by Lemma 9 in \citep{GeJ93}. Similarly,
	\begin{align*}
		\E{\abs{\frac{1}{k_n}\sum_{j=1}^{k_n-2}\CE{\abs{s_{j,n}^{(3)}((\kappa,n),\kappa)}^2}{\mathcal{H}_j^n}}}
		&\to 0,\\
		\E{\abs{\frac{1}{k_n}\sum_{j=1}^{k_n-2}\CE{\abs{s_{j,n}^{(4)}((\kappa,n),\kappa)}^2}{\mathcal{H}_j^n}}}
		&\to 0.
	\end{align*}
	Finally,
	\begin{align*}
		\E{\abs{\sum_{j=1}^{k_n-2}\CE{\abs{r_{j,n}^{(1)}((\kappa,n),\kappa)}^2}{\mathcal{H}_j^n}}}
		\to 0.
	\end{align*}
	Here we obtain the proof.
\end{proof}

\subsection{Proofs of results in Section 3.1}
\begin{proof}[Proof of \ref*{lem311}]
	We obtain
	\begin{align*}
		\hat{\Lambda}_n-\Lambda_{\star}
		&=\frac{1}{2n}\sum_{i=0}^{n-1}\parens{X_{(i+1)h_n}-X_{ih_n}}^{\otimes 2}
		+\frac{1}{2n}\sum_{i=0}^{n-1}\Lambda_{\star}^{1/2}\parens{\parens{\epsilon_{(i+1)h_n}-\epsilon_{ih_n}}^{\otimes 2}-2I_d}\Lambda_{\star}^{1/2}\\
		&\qquad+\frac{1}{2n}\sum_{i=0}^{n-1}\parens{X_{(i+1)h_n}-X_{ih_n}}\parens{\epsilon_{(i+1)h_n}-\epsilon_{ih_n}}^T\Lambda_{\star}^{1/2}\\
		&\qquad+\frac{1}{2n}\sum_{i=0}^{n-1}\Lambda_{\star}^{1/2}\parens{\epsilon_{(i+1)h_n}-\epsilon_{ih_n}}\parens{X_{(i+1)h_n}-X_{ih_n}}^T.
	\end{align*}
	Note the following evaluation: the first term in the right hand side can be evaluated with Lemma \ref*{lem739},
	\begin{align*}
		\E{\norm{\frac{1}{2n}\sum_{i=0}^{n-1}\parens{X_{(i+1)h_n}-X_{ih_n}}^{\otimes 2}}}
		\to0.
	\end{align*}
	For the second term, we obtain
	\begin{align*}
		&\frac{1}{2n}\sum_{i=0}^{n-1}\Lambda_{\star}^{1/2}\parens{\parens{\epsilon_{(i+1)h_n}}^{\otimes 2}+\parens{\epsilon_{ih_n}}^{\otimes 2}-2I_d}\Lambda_{\star}^{1/2}\\
		&\qquad+\frac{1}{2n}\sum_{i=0}^{n-1}\Lambda_{\star}^{1/2}\parens{\parens{\epsilon_{(i+1)h_n}}\parens{\epsilon_{ih_n}}^T+\parens{\epsilon_{ih_n}}\parens{\epsilon_{(i+1)h_n}}^T}
		\Lambda_{\star}^{1/2}\\
		&=\Lambda_{\star}^{1/2}\parens{\frac{1}{2n}\sum_{i=0}^{n-1}\parens{\epsilon_{(i+1)h_n}}\parens{\epsilon_{ih_n}}^T+\frac{1}{2n}\sum_{i=0}^{n-1}\parens{\epsilon_{ih_n}}\parens{\epsilon_{(i+1)h_n}}^T
		}\Lambda_{\star}^{1/2}+o_P(1)
	\end{align*}
	because of law of large number and 
	\begin{align*}
		\E{\norm{\frac{1}{2n}\sum_{i=0}^{n-1}\parens{\epsilon_{(i+1)h_n}}\parens{\epsilon_{ih_n}}^T}^2}
		\to0
	\end{align*}
	hence the second term is $o_P(1)$.
	For the third term,
	\begin{align*}
		\E{\norm{\frac{1}{2n}\sum_{i=0}^{n-1}\parens{X_{(i+1)h_n}-X_{ih_n}}\parens{\epsilon_{(i+1)h_n}-\epsilon_{ih_n}}^T\Lambda_{\star}^{1/2}}}
		\to0
	\end{align*}
	because of Lemma \ref*{lem739}; and the fourth term can be evaluated as same as the third term. Hence we obtain the consistency of $\hat{\Lambda}_n$.
\end{proof}

Let us define the following quasi-likelihood functions such that
	\begin{align*}
	\ell_{1,n}(\alpha)&:=-\frac{1}{2}\sum_{j=1}^{k_n-2}
	\parens{\ip{\parens{\frac{2}{3}\Delta_n c(\lm{Y}{j-1},\alpha)}^{-1}}{\parens{\lm{Y}{j+1}-\lm{Y}{j}}^{\otimes 2}}
		+\log\det \parens{ c(\lm{Y}{j-1},\alpha)}},\\
	\tilde{\ell}_{1,n}(\alpha|\Lambda)&:=-\frac{1}{2}\sum_{j=1}^{k_n-2}
	\parens{\ip{\parens{\frac{2}{3}\Delta_n c^\dagger(\lm{Y}{j-1},\alpha,\Lambda)}^{-1}}{\parens{\lm{Y}{j+1}-\lm{Y}{j}}^{\otimes 2}}
		+\log\det \parens{ c^\dagger(\lm{Y}{j-1},\alpha,\Lambda)}},\\
	\ell_{2,n}(\beta|\alpha)&:=-\frac{1}{2}\sum_{j=1}^{k_n-2}
	\ip{\parens{\Delta_nc(\lm{Y}{j-1},\alpha)}^{-1}}{\parens{\lm{Y}{j+1}-\lm{Y}{j}-\Delta_nb(\lm{Y}{j-1},\beta)}^{\otimes 2}},\\
	\tilde{\ell}_{2,n}(\beta|\Lambda,\alpha)
	&:=-\frac{1}{2}\sum_{j=1}^{k_n-2}
	\ip{\parens{\Delta_nc^\dagger(\lm{Y}{j-1},\alpha,\Lambda)}^{-1}}{\parens{\lm{Y}{j+1}-\lm{Y}{j}-\Delta_nb(\lm{Y}{j-1},\beta)}^{\otimes 2}},
	\end{align*}
	and corresponding adaptive ML-type estimators $\check{\alpha}_n$, $\tilde{\alpha}_n$, $\check{\beta}_n$ and $\tilde{\beta}_n$, where
	\begin{align*}
	\ell_{1,n}(\check{\alpha}_n)&=\sup_{\alpha\in\Theta_1}\ell_{1,n}(\alpha),\\
	\tilde{\ell}_{1,n}(\tilde{\alpha}_n|\hat{\Lambda}_n)&=\sup_{\alpha\in\Theta_1}\tilde{\ell}_{1,n}(\alpha|\hat{\Lambda}_n),\\
	\ell_{2,n}(\check{\beta}_n|\check{\alpha}_n)&=\sup_{\beta\in\Theta_2}\ell_{2,n}(\beta|\check{\alpha}_n),\\
	\tilde{\ell}_{2,n}(\tilde{\beta}_n|\hat{\Lambda}_n,\tilde{\alpha}_n)&=\sup_{\beta\in\Theta_2}\tilde{\ell}_{2,n}(\beta|\hat{\Lambda}_n,\tilde{\alpha}_n).
	\end{align*}

\begin{lemma}\label{lem761}
	Under (A1)-(A7) and (AH), $\check{\alpha}_n$ and $\check{\beta}_n$ are consistent when $\tau\in(1,2)$, and $\tilde{\alpha}_n$ and $\tilde{\beta}_n$ are consistent when $\tau=2$. 
\end{lemma}

\begin{proof}
	First of all, we consider the case $\tau\in(1,2)$. Then we can evaluate $\ell_{1,n}$ as
	\begin{align*}
		\frac{1}{k_n}\ell_{1,n}\parens{\alpha}
		&\cp -\frac{1}{2}\nu_0\parens{\parens{c(\cdot,\alpha)}^{-1}c(\cdot,\alpha^\star)+\log \det c(\cdot,\alpha)}\text{ uniformly in }\theta,
	\end{align*}
	and $\ell_{2,n}$ as
	\begin{align*}
		\frac{1}{k_n\Delta_n}\parens{\ell_{2,n}(\beta|\alpha)-\ell_{2,n}(\beta^\star|\alpha)}
		\cp -\frac{1}{2}\nu_0\parens{\ip{\parens{c(\cdot,\alpha)}^{-1}}{b(\cdot,\beta)-b(\cdot,\beta^\star)}}\text{ uniformly in }\theta.
	\end{align*}
	We define 
	\begin{align*}
		\mathbb{Y}_2(\beta|\alpha):=-\frac{1}{2}\nu_0\parens{\ip{\parens{c(\cdot,\alpha)}^{-1}}{b(\cdot,\beta)-b(\cdot,\beta^\star)}}
	\end{align*}
	and obviously $\mathbb{Y}_2(\beta|\alpha^\star)=\mathbb{Y}_2(\beta)$.
	
	Note that convergence in probability is equivalent to the existence of the subsequence $\{n^{(1)}_k\}\subset\{n_k\}$ converging almost surely 
	 for any subsequence $\{n_k\}\subset \mathbf{N}$. Using this fact, for any subsequence $\{n_k\}\subset \mathbf{N}$,
	  there exists a subsequence $\{n_k^{(1)}\}$ such that
	 \begin{align*}
	  	&P\parens{A_{1}\cap B_{1}}=1,\\
	  	& \text{ where }A_{1}:=\left\{\omega\in \Omega\left|\sup_{\alpha\in\Theta_1}\abs{\parens{\frac{1}{k_{n_k^{(1)}}}\ell_{1,n_k^{(1)}}(\alpha)-d-\nu_0(\log\det c(\cdot,\alpha^\star))}-\mathbb{Y}_1(\alpha)}(\omega)\to 0\right.\right\},\\
	  	&\qquad B_1:=\left\{\omega\in\Omega\left|\sup_{\theta\in\Theta}
	  	\abs{\frac{1}{k_{n_k^{(1)}}\Delta_{n_k^{(1)}}}\parens{\ell_{2,{n_k^{(1)}}}(\beta|\alpha)-\ell_{2,{n_k^{(1)}}}(\beta^\star|\alpha)}-\mathbb{Y}_2(\beta|\alpha)}(\omega)\to0\right.\right\}.
	 \end{align*}
	We take $\omega\in A_1\cap B_1$ and then for compactness of $\Theta_1$, we have a subsequence $n_k^{(2)}\subset n_k^{(1)}$ and an element $\alpha_\infty\in\Theta_1$ such that 
	 $\check{\alpha}_{n_k^{(2)}}(\omega)\to\alpha_\infty$. The continuity of $\mathbb{Y}_1(\alpha)$ leads to
	\begin{align*}
	 	\parens{\frac{1}{k_{n_k^{(2)}}}\ell_{1,n_k^{(2)}}\parens{\check{\alpha}_{n_k^{(2)}}}(\omega)
	 		-d-\nu_0(\log\det c(\cdot,\alpha^\star))}\to \mathbb{Y}_1(\alpha_\infty).
	\end{align*}
	The definition of $\check{\alpha}_n$ also leads to
	\begin{align*}
		\frac{1}{k_{n_k^{(2)}}}\ell_{1,n_k^{(2)}}\parens{\check{\alpha}_{n_k^{(2)}}}(\omega)\ge \frac{1}{k_{n_k^{(2)}}}\ell_{1,n_k^{(2)}}\parens{\alpha^\star}(\omega).
	\end{align*}
	Then we obtain
	\begin{align*}
		\mathbb{Y}_1(\alpha_\infty)\ge \mathbb{Y}_1(\alpha^\star),
	\end{align*}
	and the assumption (A6) leads to $\alpha_\infty=\alpha^\star$. Hence $\check{\alpha}_n\cp\alpha^\star$. With respect to $\beta$, for the compactness of $\Theta_2$, we have a subsequence $n_k^{(3)}\subset n_{k}^{(2)}$ and $\beta_\infty\in\Theta_2$ such that $\check{\beta}_{n_k^{(3)}}(\omega)\to\beta_\infty$. The continuity of $\mathbb{Y}_2(\beta|\alpha)$ and the convergence of $\hat{\alpha}_{n_k^{(2)}}(\omega)$ leads to
	\begin{align*}
		\sup_{\theta\in\Theta}
		\abs{\frac{1}{k_{n_k^{(3)}}\Delta_{n_k^{(3)}}}\parens{\ell_{2,{n_k^{(3)}}}(\beta|\check{\alpha}_{n_k^{(3)}})-\ell_{2,{n_k^{(3)}}}(\beta^\star|\check{\alpha}_{n_k^{(3)}})}-\mathbb{Y}_2(\beta)}(\omega)\to0
	\end{align*}
	and
	\begin{align*}
		\parens{\frac{1}{k_{n_k^{(3)}}\Delta_{n_k^{(3)}}}\parens{\ell_{2,{n_k^{(3)}}}\parens{\check{\beta}_{n_k^{(3)}}|\check{\alpha}_{n_k^{(3)}}}
			-\ell_{2,{n_k^{(3)}}}\parens{\beta^\star|\check{\alpha}_{n_k^{(3)}}}}}(\omega)\to \mathbb{Y}_2(\beta_\infty).
	\end{align*}
	The definition of $\check{\beta}_n$ leads to
	\begin{align*}
		\parens{\frac{1}{k_{n_k^{(3)}}\Delta_{n_k^{(3)}}}\parens{\ell_{2,{n_k^{(3)}}}\parens{\check{\beta}_{n_k^{(3)}}|\check{\alpha}_{n_k^{(3)}}}
			-\ell_{2,{n_k^{(3)}}}\parens{\beta^\star|\check{\alpha}_{n_k^{(3)}}}}}(\omega)\ge 0
	\end{align*}
	and then the assumption (A6) verifies $\beta_\infty=\beta^\star$. Here we obtain the result.\\
	
	Secondly, for $\tau=2$, we evaluate $\tilde{\ell}_{1,n}$ as
	\begin{align*}
		\frac{1}{k_n}\tilde{\ell}_{1,n}\parens{\alpha|\Lambda}
		\cp -\frac{1}{2}\nu_0\parens{\parens{c^\dagger(\cdot,\alpha,\Lambda)}^{-1}c^\dagger(\cdot,\alpha^\star,\Lambda^\star)+\log \det c^\dagger(\cdot,\alpha,\Lambda)}\text{ uniformly in }\vartheta,
	\end{align*}
	and $\tilde{\ell}_{2,n}$ as
	\begin{align*}
		&\frac{1}{k_n\Delta_n}\parens{\tilde{\ell}_{2,n}(\beta|\Lambda,\alpha)-\tilde{\ell}_{2,n}(\beta^\star|\Lambda,\alpha)}\\
		&\cp -\frac{1}{2}\nu_0\parens{\ip{\parens{c^\dagger(\cdot,\alpha,\Lambda)}^{-1}}{b(\cdot,\beta)-b(\cdot,\beta^\star)}}
		\text{ uniformly in }\vartheta.
	\end{align*}
	As above we define
	\begin{align*}
		\tilde{\mathbb{Y}}_1(\alpha|\Lambda)&:=-\frac{1}{2}\nu_0\parens{\mathrm{tr}\parens{
				\parens{c^\dagger(\cdot,\alpha,\Lambda)}^{-1}c^\dagger(\cdot,\alpha,\Lambda_{\star})-I_d}
			+\log\frac{\det c^\dagger(\cdot,\alpha,\Lambda)}{\det c^\dagger(\cdot,\alpha^\star,\Lambda_{\star})}}\\
		\tilde{\mathbb{Y}}_2(\beta|\Lambda,\alpha)&:=-\frac{1}{2}\nu_0\parens{\ip{\parens{c^\dagger(\cdot,\alpha,\Lambda)}^{-1}}{b(\cdot,\beta)-b(\cdot,\beta^\star)}}.
	\end{align*}
	The way of proof is almost identical to the previous one. For any subsequence $\{n_k\}\subset\mathbf{N}$, there exists a subsequence $\{n_k^{(1)}\}$ such that
	\begin{align*}
		&P\parens{S_2\cap A_2\cap B_2}=1,\text{ where }\\
		&S_2:=\left\{\omega\in\Omega\left|\hat{\Lambda}_{n_k^{(1)}}(\omega)\to \Lambda\right.\right\},\\
		&A_2:=\left\{\omega\in\Omega\left|\sup_{\vartheta\in\Xi}
		\abs{\parens{\frac{1}{k_{n_k^{(1)}}}\tilde{\ell}_{1,n_k^{(1)}}(\alpha|\Lambda)-d-\nu_0(\log\det c^\dagger(\cdot,\alpha^\star,\Lambda_{\star}))}-\tilde{\mathbb{Y}}_1(\alpha|\Lambda)}(\omega)\to 0\right.\right\},\\
		&B_2:=\left\{\omega\in\Omega\left|\sup_{\vartheta\in\Xi}
		\abs{\frac{1}{k_{n_k^{(1)}}\Delta_{n_k^{(1)}}}\parens{\tilde{\ell}_{2,{n_k^{(1)}}}(\beta|\Lambda,\alpha)-\tilde{\ell}_{2,{n_k^{(1)}}}(\beta^\star|\Lambda,\alpha)}-\tilde{\mathbb{Y}}_2(\beta|\Lambda,\alpha)}(\omega)\to0\right.\right\}.
	\end{align*}
	We take $\omega\in S_2\cap A_2\cap B_2$. The compactness of $\Theta_1$ leads to the existence of a subsequence 
	$\{n_k^{(2)}\}\subset\{n_k^{(1)}\}$ and an element $\alpha_\infty'\in \Theta_1$ such that $\tilde{\alpha}_{n_k^{(1)}}(\omega)\to \alpha_\infty'$. The continuity of $\tilde{\mathbb{Y}}_1(\alpha|\Lambda)$ verifies
	\begin{align*}
		\parens{\frac{1}{k_{n_k^{(2)}}}\tilde{\ell}_{1,n_k^{(2)}}(\tilde{\alpha}_{n_k^{(2)}}|\hat{\Lambda}_{n_k^{(2)}})
			-d-\nu_0(\log\det c^\dagger(\cdot,\alpha^\star,\Lambda_{\star}))}(\omega)\to
		\tilde{\mathbb{Y}}_1(\alpha_\infty'|\Lambda_\star)
	\end{align*}
	and the definition $\tilde{\alpha}_n$ leads to
	\begin{align*}
		\frac{1}{k_{n_k^{(2)}}}\tilde{\ell}_{1,n_k^{(2)}}(\tilde{\alpha}_{n_k^{(2)}}|\hat{\Lambda}_{n_k^{(2)}})(\omega)\ge \frac{1}{k_{n_k^{(2)}}}\tilde{\ell}_{1,n_k^{(2)}}(\alpha^\star|\hat{\Lambda}_{n_k^{(2)}})
	\end{align*}
	and (A6) gives $\alpha_\infty'=\alpha^\star$. Furthermore, the compactness of $\Theta_2$ leads to the existence of a subsequence 
	$n_k^{(3)}\subset n_k^{(2)}$ and an element $\beta_\infty'\in\Theta_2$ such that $\tilde{\beta}_{n_k^{(3)}}(\omega)\to\beta_\infty'$. 
	The continuity of $\tilde{\mathbb{Y}}_2(\beta|\Lambda,\alpha)$ gives
	\begin{align*}
		\frac{1}{k_{n_k^{(3)}}\Delta_{n_k^{(3)}}}\parens{\tilde{\ell}_{2,{n_k^{(3)}}}(\tilde{\beta}_{n_k^{(3)}}(\omega)|\hat{\Lambda}_{n_k^{(3)}},\tilde{\alpha}_{n_k^{(3)}})
			-\tilde{\ell}_{2,{n_k^{(3)}}}(\beta^\star|\hat{\Lambda}_{n_k^{(3)}},\tilde{\alpha}_{n_k^{(3)}})}(\omega)
		&\to\tilde{\mathbb{Y}}_2(\beta_\infty'|\Lambda_{\star},\alpha^\star)\\
		&=\tilde{\mathbb{Y}}_2(\beta_\infty')
	\end{align*}
	and the definition $\tilde{\beta}_n$ verifies
	\begin{align*}
		\frac{1}{k_{n_k^{(3)}}\Delta_{n_k^{(3)}}}\parens{\tilde{\ell}_{2,{n_k^{(3)}}}(\tilde{\beta}_{n_k^{(3)}}(\omega)|\hat{\Lambda}_{n_k^{(3)}},\tilde{\alpha}_{n_k^{(3)}})
			-\tilde{\ell}_{2,{n_k^{(3)}}}(\beta^\star|\hat{\Lambda}_{n_k^{(3)}},\tilde{\alpha}_{n_k^{(3)}})}(\omega)\ge 0;
	\end{align*}
	therefore, $\beta_\infty'=\beta^\star$ for (A6). Hence we obtain the result.
\end{proof}

\begin{proof}[Proof of Theorem \ref*{thm312}]
	It is enough to consider the case $\tau\in(1,2)$ and show 
	\begin{align*}
		&\frac{1}{k_n}\sum_{j=1}^{k_n-2}
		\parens{\ip{\parens{\frac{2}{3}\Delta_n c_n^\tau(\lm{Y}{j-1},\alpha,\hat{\Lambda}_n)}^{-1}}{\parens{\lm{Y}{j+1}-\lm{Y}{j}}^{\otimes 2}}
			+\log\det \parens{ c_n^\tau(\lm{Y}{j-1},\alpha,\hat{\Lambda}_n)}}\\
		&\quad\qquad\qquad-\frac{1}{k_n}\sum_{j=1}^{k_n-2}
		\parens{\ip{\parens{\frac{2}{3}\Delta_n c(\lm{Y}{j-1},\alpha)}^{-1}}{\parens{\lm{Y}{j+1}-\lm{Y}{j}}^{\otimes 2}}
		+\log\det \parens{ c(\lm{Y}{j-1},\alpha)}}\\
		&\cp 0 \text{ uniformly in }\theta
	\end{align*}
	and
	\begin{align*}
		&\frac{1}{k_n\Delta_n}\sum_{j=1}^{k_n-2}
			\ip{\parens{\Delta_nc_n^\tau(\lm{Y}{j-1},\alpha,\hat{\Lambda}_n)}^{-1}}{\parens{\lm{Y}{j+1}-\lm{Y}{j}-\Delta_nb(\lm{Y}{j-1},\beta)}^{\otimes 2}}\\
		&\qquad-\frac{1}{k_n\Delta_n}\sum_{j=1}^{k_n-2}
		\ip{\parens{\Delta_nc_n^\tau(\lm{Y}{j-1},\alpha,\hat{\Lambda}_n)}^{-1}}{\parens{\lm{Y}{j+1}-\lm{Y}{j}-\Delta_nb(\lm{Y}{j-1},\beta^\star)}^{\otimes 2}}\\
		&\qquad-\frac{1}{k_n\Delta_n}\sum_{j=1}^{k_n-2}
			\ip{\parens{\Delta_nc(\lm{Y}{j-1},\alpha)}^{-1}}{\parens{\lm{Y}{j+1}-\lm{Y}{j}-\Delta_nb(\lm{Y}{j-1},\beta)}^{\otimes 2}}\\
		&\qquad+\frac{1}{k_n\Delta_n}\sum_{j=1}^{k_n-2}
			\ip{\parens{\Delta_nc(\lm{Y}{j-1},\alpha)}^{-1}}{\parens{\lm{Y}{j+1}-\lm{Y}{j}-\Delta_nb(\lm{Y}{j-1},\beta^{\star})}^{\otimes 2}}\\
		&\cp 0 \text{ uniformly in }\theta
	\end{align*}
	because of Lemma \ref*{lem761}. We have
	\begin{align*}
		&\frac{1}{k_n\Delta_n}\sum_{j=1}^{k_n-2}
		\ip{\parens{\Delta_nc_n^\tau(\lm{Y}{j-1},\alpha,\hat{\Lambda}_n)}^{-1}}{\parens{\lm{Y}{j+1}-\lm{Y}{j}-\Delta_nb(\lm{Y}{j-1},\beta)}^{\otimes 2}}\\
		&\qquad-\frac{1}{k_n\Delta_n}\sum_{j=1}^{k_n-2}
		\ip{\parens{\Delta_nc_n^\tau(\lm{Y}{j-1},\alpha,\hat{\Lambda}_n)}^{-1}}{\parens{\lm{Y}{j+1}-\lm{Y}{j}-\Delta_nb(\lm{Y}{j-1},\beta^\star)}^{\otimes 2}}\\
		&\qquad-\frac{1}{k_n\Delta_n}\sum_{j=1}^{k_n-2}
		\ip{\parens{\Delta_nc(\lm{Y}{j-1},\alpha)}^{-1}}{\parens{\lm{Y}{j+1}-\lm{Y}{j}-\Delta_nb(\lm{Y}{j-1},\beta)}^{\otimes 2}}\\
		&\qquad+\frac{1}{k_n\Delta_n}\sum_{j=1}^{k_n-2}
		\ip{\parens{\Delta_nc(\lm{Y}{j-1},\alpha)}^{-1}}{\parens{\lm{Y}{j+1}-\lm{Y}{j}-\Delta_nb(\lm{Y}{j-1},\beta^{\star})}^{\otimes 2}}\\
		&=\frac{1}{k_n\Delta_n}\sum_{j=1}^{k_n-2}
		\ip{\parens{\parens{c_n^\tau(\lm{Y}{j-1},\alpha,\hat{\Lambda}_n)}^{-1}-\parens{c(\lm{Y}{j-1},\alpha)}^{-1}}}{\parens{\lm{Y}{j+1}-\lm{Y}{j}}\parens{-b(\lm{Y}{j-1},\beta)}^T}\\
		&\qquad+\frac{1}{k_n}\sum_{j=1}^{k_n-2}
		\ip{\parens{\parens{c_n^\tau(\lm{Y}{j-1},\alpha,\hat{\Lambda}_n)}^{-1}-\parens{c(\lm{Y}{j-1},\alpha)}^{-1}}}{\parens{b(\lm{Y}{j-1},\beta)}^{\otimes 2}}\\
		&\qquad-\frac{1}{k_n\Delta_n}\sum_{j=1}^{k_n-2}
		\ip{\parens{\parens{c_n^\tau(\lm{Y}{j-1},\alpha,\hat{\Lambda}_n)}^{-1}-\parens{c(\lm{Y}{j-1},\alpha)}^{-1}}}{\parens{\lm{Y}{j+1}-\lm{Y}{j}}\parens{-b(\lm{Y}{j-1},\beta^\star)}^T}\\
		&\qquad-\frac{1}{k_n}\sum_{j=1}^{k_n-2}
		\ip{\parens{\parens{c_n^\tau(\lm{Y}{j-1},\alpha,\hat{\Lambda}_n)}^{-1}-\parens{c(\lm{Y}{j-1},\alpha)}^{-1}}}{\parens{b(\lm{Y}{j-1},\beta^\star)}^{\otimes 2}}.
	\end{align*}
	Hence it is sufficient to check
	\begin{align*}
		&\frac{1}{k_n}\sum_{j=1}^{k_n-2}
			\ip{\parens{\frac{2}{3}\Delta_n}^{-1}\parens{\parens{c_n^\tau(\lm{Y}{j-1},\alpha,\hat{\Lambda}_n)}^{-1}-\parens{c(\lm{Y}{j-1},\alpha)}^{-1}}}
			{\parens{\lm{Y}{j+1}-\lm{Y}{j}}^{\otimes 2}}\\
		&\qquad\cp 0 \text{ uniformly in }\theta,\\
		&\frac{1}{k_n}\sum_{j=1}^{k_n-2}\parens{\log\det \parens{ c_n^\tau(\lm{Y}{j-1},\alpha,\hat{\Lambda}_n)}-\log\det \parens{ c(\lm{Y}{j-1},\alpha)}}\\
		&\qquad\cp 0 \text{ uniformly in }\theta,\\
		&\frac{1}{k_n\Delta_n}\sum_{j=1}^{k_n-2}
		\ip{\parens{\parens{c_n^\tau(\lm{Y}{j-1},\alpha,\hat{\Lambda}_n)}^{-1}-\parens{c(\lm{Y}{j-1},\alpha)}^{-1}}}{\parens{\lm{Y}{j+1}-\lm{Y}{j}}\parens{-b(\lm{Y}{j-1},\beta)}^T}\\
		&\qquad\cp 0 \text{ uniformly in }\theta\\
		&\frac{1}{k_n}\sum_{j=1}^{k_n-2}
		\ip{\parens{\parens{c_n^\tau(\lm{Y}{j-1},\alpha,\hat{\Lambda}_n)}^{-1}-\parens{c(\lm{Y}{j-1},\alpha)}^{-1}}}{\parens{b(\lm{Y}{j-1},\beta)}^{\otimes 2}}\\
		&\qquad\cp 0 \text{ uniformly in }\theta.
	\end{align*}
	Note that
	\begin{align*}
		&\sup_{\theta\in\Theta}\norm{\parens{c_n^\tau(x,\alpha,\hat{\Lambda}_n)}^{-1}-\parens{c(x,\alpha)}^{-1}}\le C\Delta_n^{\frac{2-\tau}{\tau-1}}\parens{1+\norm{x}^C}
	\end{align*}
	and
	\begin{align*}
		&\sup_{\theta\in\Theta}\abs{\log\det\parens{c_n^\tau(x,\alpha,\hat{\Lambda}_n)}-\log\det\parens{c(x,\alpha)}}\le C\Delta_n^{\frac{2-\tau}{\tau-1}}\parens{1+\norm{x}^C}.
	\end{align*}
	Hence with respect to $\mathbb{L}_{1,n}$,
	\begin{align*}
		&\E{\sup_{\theta\in\Theta}\abs{\frac{1}{k_n}\sum_{j=1}^{k_n-2}
				\ip{\parens{\frac{2}{3}\Delta_n}^{-1}\parens{\parens{c_n^\tau(\lm{Y}{j-1},\alpha,\hat{\Lambda}_n)}^{-1}-\parens{c(\lm{Y}{j-1},\alpha)}^{-1}}}
				{\parens{\lm{Y}{j+1}-\lm{Y}{j}}^{\otimes 2}}}}\\
		&\to 0
	\end{align*}
	and
	\begin{align*}
		\E{\sup_{\theta\in\Theta}\abs{\frac{1}{k_n}\sum_{j=1}^{k_n-2}\parens{\log\det \parens{ c_n^\tau(\lm{Y}{j-1},\alpha,\hat{\Lambda}_n)}-\log\det \parens{ c(\lm{Y}{j-1},\alpha)}}}}
		\to 0.
	\end{align*}
	Next we consider $\mathbb{L}_{2,n}$. Let 
	\begin{align*}
		G_{j,n}=\parens{b(\lm{Y}{j-1},\beta)}^T\parens{\parens{c_n^\tau(\lm{Y}{j-1},\alpha,\hat{\Lambda}_n)}^{-1}-\parens{c(\lm{Y}{j-1},\alpha)}^{-1}}
			\parens{\lm{Y}{j+1}-\lm{Y}{j}}.
	\end{align*}
	We have
	\begin{align*}
		\E{\sup_{\theta\in\Theta}\abs{\frac{1}{k_n\Delta_n}\sum_{j=1}^{k_n-2}
		\CE{G_{j,n}}{\mathcal{H}_j^n}}}\to 0
	\end{align*}
	and 
	\begin{align*}
		\E{\sup_{\theta\in\Theta}\frac{1}{k_n^2\Delta_n^2}\sum_{j=1}^{k_n-2}\CE{\abs{G_{j,n}}^2}{\mathcal{H}_j^n}
				}\to 0.
	\end{align*}
	Hence we obtain
	\begin{align*}
		&\frac{1}{k_n\Delta_n}\sum_{j=1}^{k_n-2}
		\ip{\parens{\parens{c_n^\tau(\lm{Y}{j-1},\alpha,\hat{\Lambda}_n)}^{-1}-\parens{c(\lm{Y}{j-1},\alpha)}^{-1}}}{\parens{\lm{Y}{j+1}-\lm{Y}{j}}\parens{-b(\lm{Y}{j-1},\beta)}^T}\\
		&\qquad\cp 0 \text{ uniformly in }\theta
	\end{align*}
	because of Lemma 9 in \citep{GeJ93}. Finally
	\begin{align*}
		\E{\sup_{\theta\in\Theta}\abs{\frac{1}{k_n}\sum_{j=1}^{k_n-2}
				\ip{\parens{\parens{c_n^\tau(\lm{Y}{j-1},\alpha,\hat{\Lambda}_n)}^{-1}-\parens{c(\lm{Y}{j-1},\alpha)}^{-1}}}{\parens{b(\lm{Y}{j-1},\beta)}^{\otimes 2}}}}
		\to 0.
	\end{align*}
	Hence we obtain the result.
\end{proof}

\begin{proof}[Proof of Theorem \ref*{thm313}]
	Firstly we prepare the notation
	\begin{align*}
		\hat{J}_n^{(1,2)}(\hat{\vartheta}_n)&:=-\int_{0}^{1}\frac{1}{\sqrt{nk_n}}\partial_{\theta_\epsilon}\partial_{\alpha}\mathbb{L}_{1,n}(\hat{\alpha}_n|\theta_{\epsilon}^{\star}+u(\hat{\theta}_{\epsilon,n}-\theta_{\epsilon}^{\star}))\dop u,\\
		\hat{J}_n^{(2,2)}(\hat{\vartheta}_n)&:=-\int_{0}^{1}\frac{1}{k_n}\partial_{\alpha}^2\mathbb{L}_{1,n}(\alpha^\star+u(\hat{\alpha}_n-\alpha^\star)|\theta_{\epsilon}^{\star})\dop u,\\
		\hat{J}_n^{(1,3)}(\hat{\vartheta}_n)&:=-\int_{0}^{1}\frac{1}{\sqrt{nk_n\Delta_n}}\partial_{\theta_{\epsilon}}\partial_{\beta}
		\mathbb{L}_{2,n}(\hat{\beta}_n|\theta_{\epsilon}^{\star}+u(\hat{\theta}_{\epsilon,n}-\theta_{\epsilon}^{\star}),\hat{\alpha}_n))\dop u,\\
		\hat{J}_n^{(2,3)}(\hat{\vartheta}_n)&:=-\int_{0}^{1}\frac{1}{k_n\sqrt{\Delta_n}}\partial_{\alpha}\partial_{\beta}
		\mathbb{L}_{2,n}(\hat{\beta}_n|\theta_{\epsilon}^{\star},\alpha^\star+u(\hat{\alpha}_n-\alpha^\star))\dop u,\\
		\hat{J}_n^{(3,3)}(\hat{\vartheta}_n)&:=-\int_{0}^{1}\frac{1}{k_n\Delta_n}\partial_{\beta}^2
		\mathbb{L}_{2,n}(\beta^\star+u(\hat{\beta}_n-\beta^\star)|\theta_{\epsilon}^{\star},\alpha^\star)\dop u,\\
		\hat{J}_n(\hat{\vartheta}_n)&:=\crotchet{\begin{matrix}
			I & O & O \\
			\hat{J}_n^{(1,2)}(\hat{\vartheta}_n) & \hat{J}_n^{(2,2)}(\hat{\vartheta}_n) & O\\
			\hat{J}_n^{(1,3)}(\hat{\vartheta}_n) & \hat{J}_n^{(2,3)}(\hat{\vartheta}_n) & \hat{J}_n^{(3,3)}(\hat{\vartheta}_n)
			\end{matrix}}.
	\end{align*}
	Taylor's theorem gives
	\begin{align*}
		&\parens{\frac{1}{\sqrt{k_n}}\partial_{\alpha}\mathbb{L}_{1,n}(\hat{\alpha}_n|\hat{\theta}_{\epsilon,n})-\frac{1}{\sqrt{k_n}}\partial_{\alpha}\mathbb{L}_{1,n}(\alpha^\star|\theta_{\epsilon}^{\star})}^T\\
		&=\parens{\int_{0}^{1}\frac{1}{\sqrt{nk_n}}\partial_{\theta_\epsilon}\partial_{\alpha}\mathbb{L}_{1,n}(\hat{\alpha}_n|\theta_{\epsilon}^{\star}+u(\hat{\theta}_{\epsilon,n}-\theta_{\epsilon}^{\star}))\dop u }\sqrt{n}\parens{\hat{\theta}_{\epsilon,n}-\theta_{\epsilon}^{\star}}\\
		&\qquad+\parens{\int_{0}^{1}\frac{1}{k_n}\partial_{\alpha}^2\mathbb{L}_{1,n}(\alpha^\star+u(\hat{\alpha}_n-\alpha^\star)|\theta_{\epsilon}^{\star})\dop u 
		}\sqrt{k_n}\parens{\hat{\alpha}_n-\alpha^\star}
	\end{align*}
	and the definition of $\hat{\alpha}_n$ leads to
	\begin{align*}
		\parens{-\frac{1}{\sqrt{k_n}}\partial_{\alpha}\mathbb{L}_{1,n}(\alpha^\star|\theta_{\epsilon}^{\star})}^T
		&=-\hat{J}_n^{(1,2)}(\hat{\vartheta}_n)\sqrt{n}\parens{\hat{\theta}_{\epsilon,n}-\theta_{\epsilon}^{\star}}-\hat{J}_n^{(2,2)}(\hat{\vartheta}_n)
		\sqrt{k_n}\parens{\hat{\alpha}_n-\alpha^\star}.
	\end{align*}
	Similarly we have
	\begin{align*}
		&\frac{1}{\sqrt{k_n\Delta_n}}
		\parens{\partial_{\beta}\mathbb{L}_{2,n}(\hat{\beta}_n|\hat{\theta}_{\epsilon,n},\hat{\alpha}_n)
			-\partial_{\beta}\mathbb{L}_{2,n}(\beta^\star|\theta_{\epsilon}^{\star},\alpha^\star)}^T\\
		&=\parens{\int_{0}^{1}\frac{1}{\sqrt{nk_n\Delta_n}}\partial_{\theta_{\epsilon}}\partial_{\beta}
			\mathbb{L}_{2,n}(\hat{\beta}_n|\theta_{\epsilon}^{\star}+u(\hat{\theta}_{\epsilon,n}-\theta_{\epsilon}^{\star}),\hat{\alpha}_n))\dop u}\sqrt{n}
		\parens{\hat{\theta}_{\epsilon,n}-\theta_{\epsilon}^{\star}}\\
		&\qquad+\parens{\int_{0}^{1}\frac{1}{k_n\sqrt{\Delta_n}}\partial_{\alpha}\partial_{\beta}
			\mathbb{L}_{2,n}(\hat{\beta}_n|\theta_{\epsilon}^{\star},\alpha^\star+u(\hat{\alpha}_n-\alpha^\star))\dop u}\sqrt{k_n}
		\parens{\hat{\alpha}_n-\alpha^\star}\\
		&\qquad+\parens{\int_{0}^{1}\frac{1}{k_n\Delta_n}\partial_{\beta}^2
			\mathbb{L}_{2,n}(\beta^\star+u(\hat{\beta}_n-\beta^\star)|\theta_{\epsilon}^{\star},\alpha^\star)\dop u}\sqrt{k_n\Delta_n}
		\parens{\hat{\beta}_n-\beta^\star}
	\end{align*}
	and hence
	\begin{align*}
		\parens{-\frac{1}{\sqrt{k_n\Delta_n}}\partial_{\beta}\mathbb{L}_{2,n}(\beta^\star|\theta_{\epsilon}^{\star},\alpha^\star)}^T
		&=-\hat{J}_n^{(1,3)}(\hat{\vartheta}_n)\sqrt{n}
		\parens{\hat{\theta}_{\epsilon,n}-\theta_{\epsilon}^{\star}}-\hat{J}_n^{(2,3)}(\hat{\vartheta}_n)\sqrt{k_n}
		\parens{\hat{\alpha}_n-\alpha^\star}\\
		&\qquad-\hat{J}_n^{(3,3)}(\hat{\vartheta}_n)\sqrt{k_n\Delta_n}
		\parens{\hat{\beta}_n-\beta^\star}.
	\end{align*}
	Here we obtain
	\begin{align*}
		&\crotchet{\begin{matrix}
			-\sqrt{n}D_n\\
			-\frac{1}{\sqrt{k_n}}\crotchet{\partial_{\alpha^{i_1}}\mathbb{L}_{1,n}(\alpha^{\star}|\theta_{\epsilon}^{\star})}_{i_1=1,\cdots,m_1}\\
			-\frac{1}{\sqrt{k_n\Delta_n}}\crotchet{\partial_{\beta^{i_2}}\mathbb{L}_{2,n}(\beta^{\star}|\theta_{\epsilon}^{\star},\alpha^{\star})}_{i_2=1,\cdots,m_2}
			\end{matrix}}
			=-\hat{J}_n(\hat{\vartheta}_n)\crotchet{\begin{matrix}
			\sqrt{n}
			\parens{\hat{\theta}_{\epsilon,n}-\theta_{\epsilon}^{\star}}\\
			\sqrt{k_n}
			\parens{\hat{\alpha}_n-\alpha^\star}\\
			\sqrt{k_n\Delta_n}
			\parens{\hat{\beta}_n-\beta^\star}
			\end{matrix}}
	\end{align*}
	and we check the asymptotics of the left hand side and the right one.
	
	\noindent\textbf{(Step 1): } For $i=1,\cdots,m_1$, we can evaluate
	\begin{align*}
		&-\frac{1}{\sqrt{k_n}}\partial_{\alpha^i}\mathbb{L}_{1,n}(\alpha^\star|\theta_{\epsilon}^{\star})\\
		&=-\frac{\sqrt{k_n}}{2}\bar{Q}_n\parens{\parens{\frac{2}{3}}^{-1}\parens{ c_{n}^{\tau}(\cdot,\alpha^\star,\Lambda_{\star})}^{-1}
			\partial_{\alpha^i}c(\cdot,\alpha^\star)\parens{ c_{n}^{\tau}(\cdot,\alpha^\star,\Lambda_{\star})}^{-1}}\\
		&\qquad+\frac{\sqrt{k_n}}{2}\bar{M}_n\parens{\ip{\parens{\parens{c_{n}^{\tau}(\cdot,\alpha^\star,\Lambda_{\star})}^{-1}
			\partial_{\alpha^i}c(\cdot,\alpha^\star)\parens{ c_{n}^{\tau}(\cdot,\alpha^\star,\Lambda_{\star})}^{-1}}}{c_{n}^{\tau}(\cdot,\alpha^\star,\Lambda_{\star})}}.
	\end{align*}
	For $i=1,\cdots,m_2$, we have
	\begin{align*}
		&-\frac{1}{\sqrt{k_n\Delta_n}}\partial_{\beta^i}\mathbb{L}_{2,n}(\beta^\star|\theta_{\epsilon}^{\star},\alpha^\star)\\
		&=\sqrt{k_n\Delta_n}\bar{D}_n\parens{\parens{\partial_{\beta^i}b(\cdot,\beta^\star)}^T\parens{c_{n}^{\tau}(\cdot,\alpha^\star,\Lambda_{\star})}^{-1}}.
	\end{align*}
	As shown in the proof of Theorem 3.1.2, if $\tau\in(1,2)$
	\begin{align*}
		\norm{\parens{c_n^\tau(x,\alpha^{\star},\Lambda_{\star})}^{-1}-\parens{c(x,\alpha^{\star})}^{-1}}\le C\Delta_n^{\frac{2-\tau}{\tau-1}}\parens{1+\norm{x}^C}
	\end{align*}
	and if $\tau=2$,
	\begin{align*}
		c_n^\tau(x,\alpha^{\star},\Lambda_{\star})=c(x,\alpha^{\star})+3\Lambda_{\star}=c^{\dagger}(x,\alpha^{\star},\Lambda_{\star}).
	\end{align*}
	Therefore, by Theorem \ref*{thm751} and Corollary \ref*{cor752}, we obtain
	\begin{align*}
		\crotchet{\begin{matrix}
			-\sqrt{n}D_n\\
			-\frac{1}{\sqrt{k_n}}\crotchet{\partial_{\alpha^{i_1}}\mathbb{L}_{1,n}(\alpha^{\star}|\theta_{\epsilon}^{\star})}_{i_1=1,\cdots,m_1}\\
			-\frac{1}{\sqrt{k_n\Delta_n}}\crotchet{\partial_{\beta^{i_2}}\mathbb{L}_{2,n}(\beta^{\star}|\theta_{\epsilon}^{\star},\alpha^{\star})}_{i_2=1,\cdots,m_2}
			\end{matrix}}\cl N(\mathbf{0},I^{\tau}(\vartheta^{\star})).
	\end{align*}
	
	\noindent\textbf{(Step 2): } We can compute 
	\begin{align*}
		\E{\norm{\int_{0}^{1}\frac{1}{\sqrt{nk_n}}\partial_{\theta_\epsilon}\partial_{\alpha}\mathbb{L}_{1,n}(\hat{\alpha}_n|\theta_{\epsilon}^{\star}+u(\hat{\theta}_{\epsilon,n}-\theta_{\epsilon}^{\star}))\dop u}} 
		&\to 0,
	\end{align*}
	and
	\begin{align*}
		\frac{1}{\sqrt{nk_n\Delta_n}}\partial_{\theta_{\epsilon}}\partial_{\beta}\mathbb{L}_{2,n}(\beta|\theta_{\epsilon},\alpha)
		&\cp 0 \text{ uniformly in }\vartheta,\\
		\frac{1}{k_n\sqrt{\Delta_n}}\partial_{\alpha}\partial_{\beta}
		\mathbb{L}_{2,n}(\beta|\theta_{\epsilon},\alpha)&\cp 0 \text{ uniformly in }\vartheta.
	\end{align*}
	We also have for $i_1,i_2\in\tuborg{1,\cdots,m_1}$
	\begin{align*}
		&-\frac{1}{k_n}\partial_{\alpha^{i_1}}\partial_{\alpha^{i_2}}\mathbb{L}_{1,n}(\alpha|\theta_{\epsilon}^{\star})\\
		&\cp \begin{cases}
		\crotchet{\frac{1}{2}\nu_0\parens{\tr\tuborg{\parens{c(\cdot,\alpha)}^{-1}\partial_{\alpha^{i_1}}c(\cdot,\alpha)\parens{c(\cdot,\alpha)}^{-1}\partial_{\alpha^{i_2}}c(\cdot,\alpha)}}}_{i_1,i_2} & \text{ if }\tau\in(1,2)\\
		\crotchet{\frac{1}{2}\nu_0\parens{\tr\tuborg{\parens{c^{\dagger}(\cdot,\alpha,\Lambda_{\star})}^{-1}\partial_{\alpha^{i_1}}c(\cdot,\alpha)\parens{c^{\dagger}(\cdot,\alpha,\Lambda_{\star})}^{-1}\partial_{\alpha^{i_2}}c(\cdot,\alpha)}}}_{i_1,i_2} & \text{ if }\tau = 2
		\end{cases}
	\end{align*}
	uniformly in $\alpha$ because of Proposition \ref*{pro741}, Theorem \ref*{thm743} and
	\begin{align*}
		\norm{\parens{c_n^\tau(x,\alpha^{\star},\Lambda_{\star})}^{-1}-\parens{c(x,\alpha^{\star})}^{-1}}\le C\Delta_n^{\frac{2-\tau}{\tau-1}}\parens{1+\norm{x}^C}
	\end{align*}
	for $\tau\in(1,2)$. Similarly, for $j_1,j_2\in\tuborg{1,\cdots,m_2}$
	\begin{align*}
		&-\frac{1}{k_n\Delta_n}\partial_{\beta^{j_1}}\partial_{\beta^{j_2}}
		\mathbb{L}_{2,n}(\beta|\theta_{\epsilon}^{\star},\alpha^\star)\\
		&\cp \begin{cases}
		\crotchet{\nu_0\parens{\ip{\parens{c(\cdot,\alpha^\star)}^{-1}}{\parens{\partial_{\beta^{j_1}}\partial_{\beta^{j_2}}b(\cdot,\beta)}\parens{b(\cdot,\beta)-b(\cdot,\beta^{\star})}^T
				+\parens{\partial_{\beta^{j_1}}b}\parens{\partial_{\beta^{j_2}}b}^T(\cdot,\beta)}}}_{j_1,j_2}\\
			\qquad\text{ if }\tau\in(1,2)\\
		\crotchet{\nu_0\parens{\ip{\parens{c^{\dagger}(\cdot,\alpha^\star,\Lambda_{\star})}^{-1}}{\parens{\partial_{\beta^{j_1}}\partial_{\beta^{j_2}}b(\cdot,\beta)}\parens{b(\cdot,\beta)-b(\cdot,\beta^{\star})}^T
					+\parens{\partial_{\beta^{j_1}}b}\parens{\partial_{\beta^{j_2}}b}^T(\cdot,\beta)}}}_{j_1,j_2}\\
		\qquad\text{ if }\tau=2.\\
		\end{cases}
	\end{align*}
	Hence
	\begin{align*}
		-\int_{0}^{1}\frac{1}{k_n}\partial_{\alpha}^2\mathbb{L}_{1,n}(\alpha^\star+u(\hat{\alpha}_n-\alpha^\star)|\theta_{\epsilon}^{\star})\dop u&\cp J^{(2,2),\tau}(\vartheta^{\star}),\\
		-\int_{0}^{1}\frac{1}{k_n\Delta_n}\partial_{\beta}^2
		\mathbb{L}_{2,n}(\beta^\star+u(\hat{\beta}_n-\beta^\star)|\theta_{\epsilon}^{\star},\alpha^\star)\dop u&\cp J^{(3,3),\tau}(\vartheta^{\star}),
	\end{align*}
	and $\hat{J}_n(\hat{\vartheta}_n)\cp J^{\tau}(\vartheta^{\star})$
\end{proof}

\subsection{Proofs of results in Section 3.2}
First of all, we define
\begin{align*}
	c_S(x):=\sum_{\ell_1=1}^{d}\sum_{\ell_2=1}^{d}c^{\ell_1,\ell_2}(x).
\end{align*}
\begin{proposition}\label{pro771} Under (A1)-(A4) and $nh_n^2\to0$,
	\begin{align*}
	\sqrt{n}\parens{\frac{1}{nh_n}\sum_{i=0}^{n-1}\parens{S_{(i+1)h_n}-S_{ih_n}}^2
		-\frac{1}{nh_n}\sum_{0\le 2i\le n-2}\parens{S_{(2i+2)h_n}-S_{2ih_n}}^2}\cl N\parens{0,2\nu_0\parens{c_S^2(\cdot)}}.
	\end{align*}
\end{proposition}

\begin{proof}
	We have
	\begin{align*}
	\sqrt{n}\parens{\frac{1}{nh_n}\sum_{i=0}^{n-1}\parens{S_{(i+1)h_n}-S_{ih_n}}^2
		-\frac{1}{nh_n}\sum_{0\le 2i\le n-2}\parens{S_{(2i+2)h_n}-S_{2ih_n}}^2}
	&=\sum_{0\le 2i\le n-2}A_{2i}^n+o_P(1),
	\end{align*}
	where
	\begin{align*}
	A_{2i}^n=\frac{\parens{-2}}{\sqrt{n}h_n}\parens{S_{(2i+2)h_n}-S_{(2i+1)h_n}}\parens{S_{(2i+1)h_n}-S_{2ih_n}}.
	\end{align*}
	With Ito-Taylor expansion and $g_1^{l}\parens{x,y}=x^{l}-y^{l}$ for all $x$ and $y$ in $\Re^d$,
	\begin{align*}
	\CE{S_{(2i+2)h_n}-S_{(2i+1)h_n}}{\mathcal{G}_{(2i+1)h_n}}=\sum_{l=1}^{d}h_nb^l(X_{(2i+1)h_n})+R(\theta,h_n^2,X_{(2i+1)h_n})
	\end{align*}
	and 
	\begin{align*}
	\CE{\parens{S_{(2i+2)h_n}-S_{(2i+1)h_n}}\parens{S_{(2i+1)h_n}-S_{2ih_n}}}{\mathcal{G}_{2ih_n}}=R(\theta,h_n^2,X_{2ih_n}).
	\end{align*}
	We obtain
	\begin{align*}
	\E{\abs{\sum_{0\le 2i\le n-2}\CE{A_{2i}^n}{\mathcal{G}_{2ih_n}^n}}}
	\to 0.
	\end{align*}
	In next, we check
	\begin{align*}
	&\CE{\parens{S_{(2i+2)h_n}-S_{(2i+1)h_n}}^2\parens{S_{(2i+1)h_n}-S_{2ih_n}}^2}{\mathcal{G}_{2ih_n}^n}\\
	&=\sum_{\ell_1, \ell_2, \ell_3, \ell_4}
	\CE{\parens{g_1^{\ell_1}\cdot g_1^{\ell_2}}(X_{(2i+2)h_n},X_{(2i+1)h_n})
		\parens{g_1^{\ell_3}\cdot g_1^{\ell_4}}(X_{(2i+1)h_n},X_{2ih_n})}{\mathcal{G}_{2ih_n}^n}.
	\end{align*}
	For each $\ell_1, \ell_2, \ell_3, \ell_4$,
	\begin{align*}
	\CE{\parens{g_1^{\ell_1}\cdot g_1^{\ell_2}}(X_{(2i+2)h_n},X_{(2i+1)h_n})
	}{\mathcal{G}_{(2i+1)h_n}^n}
	=h_nc^{\ell_1,\ell_2}(X_{(2i+1)h_n})+R(\theta,h_n^2,X_{(2i+1)h_n})
	\end{align*}
	and
	\begin{align*}
	&\CE{\parens{g_1^{\ell_1}\cdot g_1^{\ell_2}}(X_{(2i+2)h_n},X_{(2i+1)h_n})
		\parens{g_1^{\ell_3}\cdot g_1^{\ell_4}}(X_{(2i+1)h_n},X_{2ih_n})}{\mathcal{G}_{2ih_n}^n}\\
	&=h_n^2\parens{c^{\ell_1,\ell_2}\cdot c^{\ell_3,\ell_4}}(X_{2ih_n})+R(\theta,h_n^3,X_{2ih_n}).
	\end{align*}
	Therefore,
	\begin{align*}
	\CE{\parens{S_{(2i+2)h_n}-S_{(2i+1)h_n}}^2\parens{S_{(2i+1)h_n}-S_{2ih_n}}^2}{\mathcal{G}_{2ih_n}^n}
	&=h_n^2c_S^2(X_{2ih_n})+R(\theta,h_n^3,X_{2ih_n})
	\end{align*}
	and
	\begin{align*}
	\sum_{0\le 2i\le n-2}\CE{\parens{A_{2i}^n}^2}{\mathcal{G}_{2ih_n}^n}\cp 2\nu_0\parens{c_S^2(\cdot)}.
	\end{align*}
	Also we have
	\begin{align*}
	\E{\abs{\sum_{0\le 2i\le n-2}\CE{A_{2i}^n}{\mathcal{G}_{2ih_n}^n}^2}}\to 0.
	\end{align*}
	Finally we can see
	\begin{align*}
	\CE{\parens{S_{(2i+2)h_n}-S_{(2i+1)h_n}}^4\parens{S_{(2i+1)h_n}-S_{2ih_n}}^4}{\mathcal{G}_{2ih_n}^n}=R(\theta,h_n^4,X_{2ih_n})
	\end{align*}
	and hence
	\begin{align*}
	\E{\abs{\sum_{0\le 2i\le n-2}\CE{\parens{A_{2i}^n}^4}{\mathcal{G}_{2ih_n}^n}}}\to 0.
	\end{align*}
\end{proof}

\begin{lemma}\label{lem772}
	Under (A1)-(A4) and (AH),
	\begin{align*}
		&\frac{1}{k_n\Delta_n^2}\sum_{j=1}^{k_n-2}\parens{\lm{\mathscr{S}}{j+1}-\lm{\mathscr{S}}{j}}^4\\
		&\cp \begin{cases}
		\frac{4}{3}\nu_0\parens{c_S^2(\cdot)} &\text{ if }\tau\in(1,2)\\
		\frac{4}{3}\nu_0\parens{c_S^2(\cdot)}+\frac{2}{3}\nu_0\parens{c_S(\cdot)}\parens{\sum_{i_1=1}^{d}\sum_{i_2=1}^{d}\Lambda_{\star}^{i_1,i_2}}\\
		\qquad+\parens{2d^2+10d}\parens{\sum_{i_1=1}^{d}\sum_{i_2=1}^{d}\Lambda_{\star}^{i_1,i_2}}^2&\text{ if }\tau=2.
		\end{cases}
	\end{align*}
\end{lemma}

\begin{proof} We prove with the identical way as Theorem \ref*{thm743}. Note the evaluation
	\begin{align*}
		&\parens{\lm{\mathscr{S}}{j+1}-\lm{\mathscr{S}}{j}}^4\\
		&=\parens{\sum_{i=1}^{d}\parens{a^{i,\cdot}(X_{j\Delta_n})\parens{\zeta_{j+1,n}+\zeta_{j+2,n}'}+\parens{\Lambda_{\star}^{1/2}}^{i,\cdot}\parens{\lm{\epsilon}{j+1}-\lm{\epsilon}{j}}}}^4\\
		&\qquad+4\parens{\sum_{i=1}^{d}\parens{a^{i,\cdot}(X_{j\Delta_n})\parens{\zeta_{j+1,n}+\zeta_{j+2,n}'}+\parens{\Lambda_{\star}^{1/2}}^{i,\cdot}\parens{\lm{\epsilon}{j+1}-\lm{\epsilon}{j}}}}^3\parens{\sum_{i=1}^{d}e_{j,n}^{i}}\\
		&\qquad+6\parens{\sum_{i=1}^{d}\parens{a^{i,\cdot}(X_{j\Delta_n})\parens{\zeta_{j+1,n}+\zeta_{j+2,n}'}+\parens{\Lambda_{\star}^{1/2}}^{i,\cdot}\parens{\lm{\epsilon}{j+1}-\lm{\epsilon}{j}}}}^2\parens{\sum_{i=1}^{d}e_{j,n}^{i}}^2\\
		&\qquad+4\parens{\sum_{i=1}^{d}\parens{a^{i,\cdot}(X_{j\Delta_n})\parens{\zeta_{j+1,n}+\zeta_{j+2,n}'}+\parens{\Lambda_{\star}^{1/2}}^{i,\cdot}\parens{\lm{\epsilon}{j+1}-\lm{\epsilon}{j}}}}\parens{\sum_{i=1}^{d}e_{j,n}^{i}}^3\\
		&\qquad+\parens{\sum_{i=1}^{d}e_{j,n}^{i}}^4,\\
	\end{align*}
	and
	\begin{align*}
		&\parens{\sum_{i=1}^{d}\parens{a^{i,\cdot}(X_{j\Delta_n})\parens{\zeta_{j+1,n}+\zeta_{j+2,n}'}+\parens{\Lambda_{\star}^{1/2}}^{i,\cdot}\parens{\lm{\epsilon}{j+1}-\lm{\epsilon}{j}}}}^4\\
		&=\parens{\sum_{i=1}^{d}a^{i,\cdot}(X_{j\Delta_n})\parens{\zeta_{j+1,n}+\zeta_{j+2,n}'}}^4\\
		&\qquad+4\parens{\sum_{i=1}^{d}a^{i,\cdot}(X_{j\Delta_n})\parens{\zeta_{j+1,n}+\zeta_{j+2,n}'}}^3\parens{\sum_{i=1}^{d}\parens{\Lambda_{\star}^{1/2}}^{i,\cdot}\parens{\lm{\epsilon}{j+1}-\lm{\epsilon}{j}}}\\
		&\qquad+6\parens{\sum_{i=1}^{d}a^{i,\cdot}(X_{j\Delta_n})\parens{\zeta_{j+1,n}+\zeta_{j+2,n}'}}^2\parens{\sum_{i=1}^{d}\parens{\Lambda_{\star}^{1/2}}^{i,\cdot}\parens{\lm{\epsilon}{j+1}-\lm{\epsilon}{j}}}^2\\
		&\qquad+4\parens{\sum_{i=1}^{d}a^{i,\cdot}(X_{j\Delta_n})\parens{\zeta_{j+1,n}+\zeta_{j+2,n}'}}\parens{\sum_{i=1}^{d}\parens{\Lambda_{\star}^{1/2}}^{i,\cdot}\parens{\lm{\epsilon}{j+1}-\lm{\epsilon}{j}}}^3\\
		&\qquad+\parens{\sum_{i=1}^{d}\parens{\Lambda_{\star}^{1/2}}^{i,\cdot}\parens{\lm{\epsilon}{j+1}-\lm{\epsilon}{j}}}^4,
	\end{align*}
	and
	\begin{align*}
		\CE{\parens{\sum_{i=1}^{d}a^{i,\cdot}(X_{j\Delta_n})\parens{\zeta_{j+1,n}+\zeta_{j+2,n}'}}^4}{\mathcal{H}_j^n}
		&=3\parens{m_n+m_n'}^2\Delta_n^2c_S^2(X_{j\Delta_n}).
	\end{align*}
	It leads to
	\begin{align*}
		\frac{1}{k_n\Delta_n^2}\sum_{1\le 3j\le k_n-2}\CE{\parens{\sum_{i=1}^{d}a^{i,\cdot}(X_{3j\Delta_n})\parens{\zeta_{3j+1,n}+\zeta_{3j+2,n}'}}^4}{\mathcal{H}_{3j}^n}\cp \frac{4}{9}\nu_0\parens{c_S^2(\cdot)}.
	\end{align*}
	It is obvious that
	\begin{align*}
		\CE{\parens{\sum_{i=1}^{d}a^{i,\cdot}(X_{j\Delta_n})\parens{\zeta_{j+1,n}+\zeta_{j+2,n}'}}^3\parens{\sum_{i=1}^{d}\parens{\Lambda_{\star}^{1/2}}^{i,\cdot}\parens{\lm{\epsilon}{j+1}-\lm{\epsilon}{j}}}}{\mathcal{H}_j^n}=0
	\end{align*}
	and
	\begin{align*}
		\CE{\parens{\sum_{i=1}^{d}a^{i,\cdot}(X_{j\Delta_n})\parens{\zeta_{j+1,n}+\zeta_{j+2,n}'}}\parens{\sum_{i=1}^{d}\parens{\Lambda_{\star}^{1/2}}^{i,\cdot}\parens{\lm{\epsilon}{j+1}-\lm{\epsilon}{j}}}^3}
		{\mathcal{H}_j^n}=0.
	\end{align*}
	We can evaluate
	\begin{align*}
		&\CE{\parens{\sum_{i=1}^{d}a^{i,\cdot}(X_{j\Delta_n})\parens{\zeta_{j+1,n}+\zeta_{j+2,n}'}}^2
			\parens{\sum_{i=1}^{d}\parens{\Lambda_{\star}^{1/2}}^{i,\cdot}\parens{\lm{\epsilon}{j+1}-\lm{\epsilon}{j}}}^2}{\mathcal{H}_j^n}\\
		&=\parens{m_n+m_n'}\Delta_nc_S(X_{j\Delta_n})\frac{1}{p_n}\parens{\sum_{i_1=1}^{d}\sum_{i_2=1}^{d}\Lambda_{\star}^{i_1,i_2}}
	\end{align*}
	and hence
	\begin{align*}
		&\frac{1}{k_n\Delta_n^2}\sum_{1\le 3j\le k_n-2}\CE{\parens{\sum_{i=1}^{d}a^{i,\cdot}(X_{3j\Delta_n})\parens{\zeta_{3j+1,n}+\zeta_{3j+2,n}'}}^2
			\parens{\sum_{i=1}^{d}\parens{\Lambda_{\star}^{1/2}}^{i,\cdot}\parens{\lm{\epsilon}{3j+1}-\lm{\epsilon}{3j}}}^2}{\mathcal{H}_{3j}^n}\\
		&\cp \begin{cases}
		0 & \text{ if }\tau \in(1,2)\\
		\frac{2}{9}\nu_0\parens{c_S(\cdot)}\sum_{i_1=1}^{d}\sum_{i_2=1}^{d}\Lambda_{\star}^{i_1,i_2} & \text{ if }\tau=2.
		\end{cases}
	\end{align*}
	We also have
	\begin{align*}
	\tr\tuborg{\CE{\parens{\lm{\epsilon}{j+1}-\lm{\epsilon}{j}}^{\otimes2}
			\parens{\lm{\epsilon}{j+1}-\lm{\epsilon}{j}}^{\otimes 2}}{\mathcal{H}_j^n}}
	=\frac{2d^2+10d}{p_n^2}+\frac{2}{p_n^3}\parens{\sum_{i=1}^{d}\E{\parens{\epsilon_{0}^i}^4}-3d}
	\end{align*}
	and
	\begin{align*}
		&\CE{\parens{\sum_{i=1}^{d}\parens{\Lambda_{\star}^{1/2}}^{i,\cdot}\parens{\lm{\epsilon}{j+1}-\lm{\epsilon}{j}}}^4}{\mathcal{H}_j^n}\\
		&=\frac{2d^2+10d}{p_n^2}+\frac{2}{p_n^3}\parens{\sum_{i=1}^{d}\E{\parens{\epsilon_{0}^i}^4}-3d}\parens{\sum_{i_1=1}^{d}\sum_{i_2=1}^{d}\Lambda_{\star}^{i_1,i_2}}^2.
	\end{align*}
	It leads to
	\begin{align*}
		&\frac{1}{k_n\Delta_n^2}\sum_{1\le 3j\le k_n-2}\CE{\parens{\sum_{i=1}^{d}\parens{\Lambda_{\star}^{1/2}}^{i,\cdot}\parens{\lm{\epsilon}{3j+1}-\lm{\epsilon}{3j}}}^4}{\mathcal{H}_{3j}^n}\\
		 &\cp\begin{cases}
		0 &\text{ if }\tau\in(1,2)\\
		\parens{2d^2+10d}\parens{\sum_{i_1=1}^{d}\sum_{i_2=1}^{d}\Lambda_{\star}^{i_1,i_2}}^2&\text{ if }\tau=2.
		\end{cases}
	\end{align*}
	Since the following evaluation
	\begin{align*}
		\CE{\norm{e_{j,n}}^{l}}{\mathcal{H}_j^n}&\le C\Delta_n^l\parens{1+\norm{X_{j\Delta_n}}^C}\\
		\CE{\norm{\zeta_{j+1,n}+\zeta_{j+2,n}}^{l}}{\mathcal{H}_j^n}&\le C\Delta_n^{l/2}\\
		\CE{\norm{\lm{\epsilon}{j+1}-\lm{\epsilon}{j}}^{l}}{\mathcal{H}_j^n}&\le C\Delta_n^{l/2}
	\end{align*}
	for all $l\in\mathbf{N}$, we have
	\begin{align*}
		&\frac{1}{k_n\Delta_n^2}\sum_{1\le 3j\le k_n-2}
		\CE{\parens{\lm{\mathscr{S}}{3j+1}-\lm{\mathscr{S}}{3j}}^4}{\mathcal{H}_{3j}^n}\\
		&\cp\begin{cases}
		\frac{4}{9}\nu_0\parens{c_S^2(\cdot)} &\text{ if }\tau\in(1,2)\\
		\frac{4}{9}\nu_0\parens{c_S^2(\cdot)}+\frac{2}{9}\nu_0\parens{c_S(\cdot)}\parens{\sum_{i_1=1}^{d}\sum_{i_2=1}^{d}\Lambda_{\star}^{i_1,i_2}}\\
		\qquad+\frac{2d^2+10d}{3}\parens{\sum_{i_1=1}^{d}\sum_{i_2=1}^{d}\Lambda_{\star}^{i_1,i_2}}^2&\text{ if }\tau=2.
		\end{cases}
	\end{align*}
	Finally we obtain
	\begin{align*}
		&\E{\abs{\frac{1}{k_n^2\Delta_n^4}\sum_{1\le 3j\le k_n-2}\CE{\parens{\lm{\mathscr{S}}{3j+1}-\lm{\mathscr{S}}{3j}}^8}{\mathcal{H}_{3j}^n}}}\\
		&\le \frac{C}{k_n}\\
		&\to 0.
	\end{align*}
	and the proof is obtained because of the Lemma 9 in \citep{GeJ93}.
\end{proof}

\begin{proof}[Proof of Theorem \ref*{thm321}]
	Under $H_0$, the result of Lemma \ref*{lem772} is equivalent to
	\begin{align*}
		\frac{3}{4k_n\Delta_n^2}\sum_{j=1}^{k_n-2}\parens{\lm{\mathscr{S}}{j+1}-\lm{\mathscr{S}}{j}}^4 \cp
		\nu_0\parens{c_S^2(\cdot)}.
	\end{align*}
	for any $\tau\in(1,2]$. Therefore, Proposition \ref*{pro771}, Lemma \ref*{lem772} and Slutsky's theorem verify the proof.
\end{proof}

\begin{proof}[Proof of Theorem \ref*{thm322}]
	Assumption (T1) verifies
	\begin{align*}
		&\sum_{l_1}\sum_{l_2}\Lambda_{\star}^{l_1,l_2}>0.
	\end{align*}
	We firstly show 
	\begin{align*}
		\frac{1}{2n}\sum_{i=0}^{n-1}\parens{\mathscr{S}_{(i+1)h_n}- \mathscr{S}_{ih_n}}^2
		\cp\sum_{l_1}\sum_{l_2}\Lambda_{\star}^{l_1,l_2}
	\end{align*}
	under $H_1$ and (T1).
	We can decompose
	\begin{align*}
		&\frac{1}{2n}\sum_{i=0}^{n-1}\parens{\mathscr{S}_{(i+1)h_n}-\mathscr{S}_{ih_n}}^2
		-\sum_{l_1}\sum_{l_2}\Lambda_{\star}^{l_1,l_2}\\
		&=\frac{1}{2n}\sum_{i=0}^{n-1}\parens{\sum_{l=1}^{d}\parens{X_{(i+1)h_n}^{l}-X_{ih_n}^{l}}}^2\\
		&\qquad+\frac{1}{2n}\sum_{i=0}^{n-1}\parens{\sum_{l=1}^{d}\parens{\Lambda_{\star}^{1/2}}^{l,\cdot}\parens{\epsilon_{(i+1)h_n}-\epsilon_{ih_n}}}^2-\sum_{l_1}\sum_{l_2}\Lambda_{\star}^{l_1,l_2}\\
		&\qquad+\frac{1}{n}\sum_{i=0}^{n-1}\parens{\sum_{l=1}^{d}\parens{X_{(i+1)h_n}^{l}-X_{ih_n}^{l}}}
		\parens{\sum_{l=1}^{d}\parens{\Lambda_{\star}^{1/2}}^{l,\cdot}\parens{\epsilon_{(i+1)h_n}-\epsilon_{ih_n}}}.
	\end{align*}
	The first and fourth term of the right hand side is $o_P(1)$ since
	\begin{align*}
		\E{\abs{\frac{1}{2n}\sum_{i=0}^{n-1}\parens{\sum_{l=1}^{d}\parens{X_{(i+1)h_n}^{l}-X_{ih_n}^{l}}}^2}}\to 0.
	\end{align*}
	The fourth term is also $o_P(1)$ since
	\begin{align*}
		\E{\abs{\frac{1}{n}\sum_{i=0}^{n-1}\parens{\sum_{l=1}^{d}\parens{X_{(i+1)h_n}^{l}-X_{ih_n}^{l}}}
				\parens{\sum_{l=1}^{d}\parens{\Lambda_{\star}^{1/2}}^{l,\cdot}\parens{\epsilon_{(i+1)h_n}-\epsilon_{ih_n}}}}}\to 0.
	\end{align*}
	We also can evaluate the second and third term of the right hand side as
	\begin{align*}
		&\frac{1}{2n}\sum_{i=0}^{n-1}\parens{\sum_{l=1}^{d}\parens{\Lambda_{\star}^{1/2}}^{l,\cdot}\parens{\epsilon_{(i+1)h_n}-\epsilon_{ih_n}}}^2-\sum_{l_1}\sum_{l_2}\Lambda_{\star}^{l_1,l_2}\\
		&=-\frac{1}{n}\sum_{i=0}^{n-1}\sum_{l_1=1}^{d}\sum_{l_2=1}^{d}\parens{\Lambda_{\star}^{1/2}}^{l_1,\cdot}\parens{\epsilon_{(i+1)h_n}}
		\parens{\epsilon_{ih_n}}^T\parens{\Lambda_{\star}^{1/2}}^{\cdot,l_2}
		+o_P(1)
	\end{align*}
	because of law of large numbers. The first term can be evaluated as
	\begin{align*}
		\E{\abs{\frac{1}{n}\sum_{i=0}^{n-1}\sum_{l_1=1}^{d}\sum_{l_2=1}^{d}\parens{\Lambda_{\star}^{1/2}}^{l_1,\cdot}\parens{\epsilon_{(i+1)h_n}}
		\parens{\epsilon_{ih_n}}^T\parens{\Lambda_{\star}^{1/2}}^{\cdot,l_2}}^2}\to 0.
	\end{align*}
	With identical computation, we obtain
	\begin{align*}
		\frac{1}{n}\sum_{0\le 2i\le n-2}\parens{\mathscr{S}_{(2i+2)h_n}-\mathscr{S}_{2ih_n}}^2&\cp \sum_{l_1,l_2}\Lambda_{\star}^{l_1,l_2}
	\end{align*}
	There exists a constant $C_1$ such that
	\begin{align*}
		\frac{3}{4k_n\Delta_n^2}\sum_{j=1}^{k_n-2}\parens{\lm{\mathscr{S}}{j+1}-\lm{\mathscr{S}}{j}}^4 &\cp C_1.
	\end{align*}
	These convergences in probability and some computations verifies the result.
\end{proof}

{
	\section*{Acknowledgement}
	This work 
	was partially supported by 
	Overseas Study Program of MMDS,
	JST CREST,
	JSPS KAKENHI Grant Number 
	JP17H01100 
	and Cooperative Research Program
	of the Institute of Statistical Mathematics.
}

\bibliography{bibliography}

\end{document}